\documentclass[10pt,reqno]{amsart}
\usepackage{amsmath,amssymb,latexsym,esint,cite,mathrsfs}
\usepackage{verbatim,wasysym}
\usepackage[left=1.8cm,right=1.8cm,top=2.3cm,bottom=2.3cm]{geometry}
\usepackage{tikz,enumitem,graphicx, subfig, microtype, color}
\usepackage{epic,eepic}

\usepackage[colorlinks=true,urlcolor=blue, citecolor=red,linkcolor=blue,
linktocpage,pdfpagelabels, bookmarksnumbered,bookmarksopen]{hyperref}
\usepackage[hyperpageref]{backref}
\usepackage[english]{babel}
\usepackage{appendix}

\numberwithin{equation}{section}

\newtheorem{thm}{Theorem}[section]
\newtheorem{lem}[thm]{Lemma}

\newtheorem{Prop}[thm]{Proposition}

\newtheorem{Rem}[thm]{Remark}

\newcommand{\R}{\mathbb{R}}

\newcommand\cp{\mathcal{P}}

\def\sp {\quad}

\begin{document}
	\baselineskip=14pt
	
	\title[Nonlocal Hamiltonian System]{ Infinitely many synchronized solutions for a nonlocal critical Hamiltonian elliptic system}
\author[W. Ye]{ Weiwei Ye}
\author[M. Yang]{Minbo Yang}

	\address{Weiwei Ye  \newline\indent Department of Mathematics, Fuyang Normal University, \newline\indent
		Fuyang, Anhui, 236037, People's Republic of China}\email{ yeweiweime@163.com}

	\address{Minbo Yang  \newline\indent School of Mathematical Sciences, Zhejiang Normal University, \newline\indent
		Jinhua, Zhejiang, 321004, People's Republic of China}
	\email{ mbyang@zjnu.edu.cn}

	\thanks{$^\ddag$Minbo Yang is the corresponding author who was partially supported by National Natural Science Foundation of China (12471114) and the Natural Science Foundation of Zhejiang Province (LZ26A010002). Weiwei Ye was partially supported by National Natural Science Foundation of China (12571116) and Doctoral Research Startup Fund of Fuyang Normal University (2024KYQD0032).}

	\begin{abstract}
		We establish the existence of infinitely many synchronized solutions for a class of critical Hamiltonian elliptic systems with Hartree-type nonlocal interactions:
$$
	\left\{\begin{array}{ll}
		-\Delta u=K_{1}(x)(|x|^{-(N-\alpha)}*K_{1}(x)v^{2^{*}_{\alpha}})v^{2^{*}_{\alpha}-1}
		&\mbox{in}\ \R^N,\\[1mm]
		-\Delta v=K_{2}(x)(|x|^{-(N-\alpha)}*K_{2}(x)u^{2^{*}_{\alpha}})u^{2^{*}_{\alpha}-1}
		&\mbox{in}\ \R^N,
	\end{array}\right.
$$
The system involves the Hardy–Littlewood–Sobolev critical exponent \( 2^{*}_{\alpha} = \frac{N + \alpha}{N - 2} \), and is considered under the assumptions  \( N \geq 5 \) and \( \alpha < N - 5 + \frac{6}{N-2} \).
The coefficient functions \( K_{1}(|x'|, x'') \) and \( K_{2}(|x'|, x'') \), defined on \( \mathbb{R}^+ \times \mathbb{R}^{N-2} \), are bounded, nonnegative, and attain a common, topologically nontrivial critical point. We first study the nondegeneracy of positive solutions to the associated limiting system by employing the stereographic projection together with the Funk–Hecke formula for spherical harmonics. Building upon this result, we apply a Lyapunov–Schmidt reduction combined with refined localized Poho\v{z}aev-type identities to construct infinitely many synchronized solutions exhibiting multi-bubble concentration profiles. 

	\end{abstract}

	\maketitle
	
%	\begin{center}
%		\begin{minipage}{8.5cm}
%			\small
%			\tableofcontents
%		\end{minipage}
%	\end{center}
	%
	%\medskip
\section{Introduction and main results}
In this paper, we study the following critical Hamiltonian-type elliptic system:
\begin{equation}\label{CFL}
	\left\{\begin{array}{ll}
		-\Delta u=K_{1}(x)(|x|^{-(N-\alpha)}*K_{1}(x)v^{2^{*}_{\alpha}})v^{2^{*}_{\alpha}-1}
		&\mbox{in}\ \R^N,\\[1mm]
		-\Delta v=K_{2}(x)(|x|^{-(N-\alpha)}*K_{2}(x)u^{2^{*}_{\alpha}})u^{2^{*}_{\alpha}-1}
		&\mbox{in}\ \R^N,
	\end{array}\right.
\end{equation}
where $N\geq 5$, $\alpha<N-5+\frac{6}{N-2}$ and $2^{*}_{\alpha}=\frac{N+\alpha}{N-2}$. $K_{1}(|x'|, x'')$ and $K_{2}(|x'|, x'')$ are bounded nonnegative functions in $\mathbb{R}^{+}\times\mathbb{R}^{N-2}$.
 System \eqref{CFL} arises in various physical contexts related to Hartree–Fock theory. It plays an important role in classical quantum mechanics and is widely used in the theory of Bose–Einstein condensates, particularly in the study of mechanisms preventing collapse phenomena \cite{ES,LY}.
In the special case where $u=v$ and $K_1(x)=K_2(x)$, system \eqref{CFL} reduces to a single nonlocal equation that appears in several areas of mathematical physics. This equation was first introduced by Pekar \cite{P1} in 1954 in the quantum theory of a polaron at rest, and it also arises in the modeling of an electron trapped in its own polarization field as an approximation of the Hartree–Fock theory for a one-component plasma.

The study of system \eqref{CFL} also originates from the investigations into the following problem
\begin{equation}\label{eq1.3}
	\left\{\begin{array}{l}
		-\Delta u = K(x) u^{\frac{N+2}{N-2}}, \quad u > 0 \ \text{in } \mathbb{R}^N, \\[1mm]
		u \in D^{1,2}(\mathbb{R}^N).
	\end{array}\right.
\end{equation}
In fact, if $u=v$ and $K_1(x)=K_2(x)$, by letting $\alpha\rightarrow 0$ in \eqref{CFL}, up to a scaling of the Riesz potential, we can regard \eqref{eq1.3} as the limit problem  of the self-interaction case \eqref{CFL}. Over the past few decades, considerable effort has been devoted to exploring conditions on \(K\) under which solutions exist. For further details, see \cite{B,BC,CNY,CL,CY1,H,Li1,Li2,Y} and references therein.
 We first review some key findings related to equation \eqref{eq1.3} concerning infinitely many solutions. In \cite{Li}, Li established the existence of infinitely many solutions to \eqref{eq1.3} under the condition that \( K \) has a local maximum at \( |x| = r_0 > 0 \) and satisfies
$$
K(r) = K(r_0) - c_0 |r - r_0|^\kappa + O(|r - r_0|^{\kappa + \theta})
$$
as \( r \rightarrow r_0 \). Wei and Yan \cite{WY1} introduced new methods using the number of solution bubbles as a parameter to construct bubbling solutions with energy that can be made arbitrarily large.
In \cite{DLY}, Deng, Lin, and Yan obtained a local uniqueness result for bubbling solutions in the prescribed scalar curvature problem. They further demonstrated that if \( K(y) \) is periodic in \( y_1 \) with period 1 and has a local maximum at 0, then a bubbling solution with blow-up set
\[
\left\{ (jL, 0, \dots, 0) : j = 0, \pm 1, \pm 2, \dots \right\}
\]
must be periodic in \( y_1 \) when \( L \) is a sufficiently large positive integer.
Guo, Peng, and Yan \cite{GPY} studied polyharmonic equations with critical exponents and proved the existence and local uniqueness of solutions with infinitely many bubbles under certain optimal conditions on the coefficient \( K(y) \). In \cite{LWX}, the authors examined the critical semilinear elliptic equation
$$
-\Delta u = K(x) u^{\frac{N+2}{N-2}}, \quad u > 0 \ \text{in} \ \mathbb{R}^N,
$$
where \( n \geq 3 \) and \( K > 0 \) is periodic in \( (x_1, \dots, x_k) \) with \( 1 \leq k < \frac{n-2}{2} \). They demonstrated the existence of multi-bump solutions, with bump centers positioned on infinite lattices in \( \mathbb{R}^k \), and noted that such solutions do not exist for \( k \geq \frac{n-2}{2} \).
Peng, Wang, and Wei \cite{PWW} proved the existence of infinitely many solutions under general conditions on \( K \), allowing for saddle points to be stable critical points of \( K \). Guo, Musso, Peng, and Yan \cite{GMPY} established the non-degeneracy of positive bubble solutions and constructed a novel class of solutions by combining a large number of bubbles.

We would also like to recall the classical Hamiltonian  type ellptic system
\begin{equation}\label{eq1'}
\begin{cases}
-\Delta u= |v|^{p-1}v,\ & \text{in}\ \mathbb{R}^N ,\\
-\Delta v= |u|^{q-1}u, \  & \text{in}\  \mathbb{R}^N.
\end{cases}
\end{equation}
This system has attracted considerable interest due to its complex structure and intricate coupling.
Traditional analytical methods developed for single equations are often not directly applicable to such systems, which has contributed to a relatively limited focus on the existence of solutions and the qualitative properties of strongly indefinite systems. An early result on positive solutions of \eqref{eq1'} was presented in \cite{c-f-m} using topological methods. Subsequently, \cite{f-f} applied a variational approach based on a linking theorem to establish an existence result.
Further advancements, including those in \cite{b-s-r}, explored the existence, positivity, and uniqueness of ground state solutions for \eqref{eq1'}. For additional insights, one may consult the works in  \cite{a-c-m,s} and the surveys found in \cite{f}.
In \cite{GLP}, Guo, Liu, and Peng studied the critical elliptic system
\begin{equation}
	\left\{\begin{array}{ll}
		-\Delta u = K_1(y) v^p & \text{in} \ \mathbb{R}^N, \\[1mm]
		-\Delta v = K_2(y) u^q & \text{in} \ \mathbb{R}^N,
	\end{array}\right.
\end{equation}
where \( N \geq 5 \), \( p, q \in (0, \infty) \) with \( \frac{1}{p+1} + \frac{1}{q+1} = \frac{N-2}{N} \), and \( K_1(y) > 0 \), \( K_2(y) > 0 \) are radial potentials. They constructed an unbounded sequence of non-radial positive vector solutions with energy that can be made arbitrarily large and, using various Poho\v{z}aev  identities, proved a specific type of non-degeneracy result.

Compared with the local case, there are relatively few studies on the existence of infinitely many solutions for system \eqref{CFL}. In \cite{GMYZ}, Gao, Moroz, Yang, and Zhao established the nondegeneracy of the psoitive solutions of the critical Hartree equation and studied a class of critical Hartree equations with axisymmetric potentials:
$$
-\Delta u + V(|x'|, x'') u = \Big(|x|^{-4} \ast |u|^2 \Big) u \quad \text{in} \ \mathbb{R}^6,
$$
where \( (x', x'') \in \mathbb{R}^2 \times \mathbb{R}^4 \), and \( V(|x'|, x'') \) is a bounded, nonnegative function in \( \mathbb{R}^+ \times \mathbb{R}^4 \). By applying a finite-dimensional reduction argument and developing novel local Poho\v{z}aev  identities, they proved that if the function \( r^2 V(r, x'') \) has a topologically nontrivial critical point, then the problem admits infinitely many solutions with arbitrarily large energies.
Ye et al. studied the nondegeneracy of the psoitive solutions of the critical Hartree system and investigated the existence of infinitely many solutions for the critical Hartree system,
\begin{equation}\nonumber
	\left\{\begin{array}{ll}
		-\Delta u + P(|x'|, x'') u = \alpha_1 \big(|x|^{-4} \ast u^2\big) u + \beta \big(|x|^{-4} \ast v^2\big) u
		& \text{in} \ \mathbb{R}^6, \\[1mm]
		-\Delta v + Q(|x'|, x'') v = \alpha_2 \big(|x|^{-4} \ast v^2\big) v + \beta \big(|x|^{-4} \ast u^2\big) v
		& \text{in} \ \mathbb{R}^6,
	\end{array}\right.
\end{equation}
where \( (x', x'') \in \mathbb{R}^2 \times \mathbb{R}^4 \), and the potentials \( P, Q \geq 0 \) are axisymmetric, bounded, belong to \( C^1 \), and \( P, Q \not\equiv 0 \). When the functions \( r^2 P(r, x'') \) and \( r^2 Q(r, x'') \) have a common topologically nontrivial critical point, they constructed infinitely many synchronized-type solutions.
Recently, Cassani, Yang, and Zhang \cite{CYZ} considered the equation
$$
-\Delta u = K(|x'|, x'') \left( |x|^{-\alpha} * K(|x'|, x'') u^{2^*_\alpha} \right) u^{2^*_\alpha - 1} \quad \text{in} \ \mathbb{R}^N,
$$
where \( N \geq 5 \), \( \alpha > 5 - \frac{6}{N - 2} \), and \( K(|x'|, x'') \) is a bounded, nonnegative function. Under appropriate assumptions on the potential function \( K \), the authors proved the existence of infinitely many solutions for this nonlocal critical equation.

In the present paper we aim to explore the existence of infinitely many solutions of \eqref{CFL} under suitable assumptions of the nonlinear potentials $K_1(x), K_2(x)$.
 Firstly, we need to recall the well-known Hardy-Littlewood-Sobolev inequality (see \cite[Theorem 4.3]{LL}). We denote \(2_{\alpha}^{\ast} = \frac{N + \alpha}{N - 2}\) as the upper Hardy-Littlewood-Sobolev critical exponent.

\begin{Prop}\label{pro1.1}
	Let \( t, r > 1 \) and \( 0 < \mu < N \) be such that \( \frac{1}{t} + \frac{\mu}{N} + \frac{1}{r} = 2 \).
	Then there exists a sharp constant \( C(N, \mu, t) \) such that, for \( f \in L^{t}(\mathbb{R}^N) \)
	and \( h \in L^{r}(\mathbb{R}^N) \),
	$$
	\left|\int_{\mathbb{R}^{N}}\int_{\mathbb{R}^{N}} \frac{f(x) h(y)}{|x - y|^{\mu}} \, dx \, dy\right|
	\leq C(N, \mu, t) \, |f|_{L^t(\mathbb{R}^N)} |h|_{L^r(\mathbb{R}^N)}.
	$$
	If \( t = r = \frac{2N}{2N - \mu} \), then
	$$
	C(t, N, \mu, r) = C(N, \mu) = \pi^{\frac{\mu}{2}} \frac{\Gamma\left(\frac{N}{2} - \frac{\mu}{2}\right)}{\Gamma(N - \frac{\mu}{2})} \left\{ \frac{\Gamma(\frac{N}{2})}{\Gamma(N)} \right\}^{-1 + \frac{\mu}{N}}.
	$$
	In this case, equality is achieved if and only if \( f \equiv C h \) and
	$$
	h(x) = A (\gamma^2 + |x - a|^2)^{-(2N - \mu)/2}
	$$
	for some \( A \in \mathbb{C} \), \( 0 \neq \gamma \in \mathbb{R} \), and \( a \in \mathbb{R}^N \).
\end{Prop}

System \eqref{CFL} is closely related to the following Hartree-type system:
\begin{equation}\label{hemilton}
	\left\{\begin{array}{ll}
		-\Delta u = (|x|^{-(N-\alpha)} * v^{2^{*}_{\alpha}}) v^{2^{*}_{\alpha}-1}
		&\text{in} \ \mathbb{R}^N,\\[1mm]
		-\Delta v = (|x|^{-(N-\alpha)} * u^{2^{*}_{\alpha}}) u^{2^{*}_{\alpha}-1}
		&\text{in} \ \mathbb{R}^N.
	\end{array}\right.
\end{equation}
For the single equation with the critical exponent, Lei \cite{Lei}, Du and Yang \cite{DY}, and Guo et al. \cite{GHPS} independently classified the positive solutions of the critical Hartree equation
\begin{equation}\label{eq1.5}
	-\Delta u = \left( |x|^{-\mu} * |u|^{2_{\mu}^{\ast}} \right) |u|^{2_{\mu}^{\ast}-2} u
	\ \text{in} \ \mathbb{R}^N,
\end{equation}
and proved that every positive solution of \eqref{eq1.5} must have the form
\begin{equation}\label{REL}
	U_{z,\lambda}(x) = S^{\frac{(N-\mu)(2-N)}{4(N-\mu+2)}} C(N,\mu)^{\frac{2-N}{2(N-\mu+2)}} [N(N-2)]^{\frac{N-2}{4}} \left( \frac{\lambda}{1 + \lambda^{2}|x-z|^{2}} \right)^{\frac{N-2}{2}}.
\end{equation}
In \cite{Le}, by using the moving sphere method in integral form, the author classified all positive solutions to the system
\begin{equation}\label{H11}
	\left\{\begin{array}{ll}
		-\Delta u = (|x|^{-(N-\alpha)} * v^{p}) v^{p-1}
		&\text{in} \ \mathbb{R}^N,\\[1mm]
		-\Delta v = (|x|^{-(N-\beta)} * u^{q}) u^{q-1}
		&\text{in} \ \mathbb{R}^N.
	\end{array}\right.
\end{equation}

We have the following lemma:

\begin{lem}{\rm (Theorem 1.2, \cite{Le})}\label{CS}
If $(p,q) = \left( \frac{N+\alpha}{N-2}, \frac{N+\beta}{N-2} \right)$, then every positive classical solution $(u,v)$ of \eqref{H11} must assume the form
\begin{equation}\label{clasol}
	u(x) = c_1 \left( \frac{\lambda}{1 + \lambda^{2}|x - z|^2} \right)^{\frac{N-2}{2}}, \quad v(x) = c_2 \left( \frac{\lambda}{1 + \lambda^{2}|x - z|^2} \right)^{\frac{N-2}{2}},
\end{equation}
for some $\lambda > 0$ and $z \in \mathbb{R}^N$, where
$$
c_1^{1 - \frac{(N + 2\alpha + 2)(N + 2\beta + 2)}{(N-2)^2}} = R_N^{\frac{2N + 2\alpha}{N-2}} I \left( \frac{N-\beta}{2} \right)^{\frac{N + 2\alpha + 2}{N-2}} I \left( \frac{N-\alpha}{2} \right) I \left( \frac{N-2}{2} \right)^{\frac{2N + 2\alpha}{N-2}},
$$
$$
c_2^{1 - \frac{(N + 2\alpha + 2)(N + 2\beta + 2)}{(N-2)^2}} = R_N^{\frac{2N + 2\beta}{N-2}} I \left( \frac{N-\alpha}{2} \right)^{\frac{N + 2\beta + 2}{N-2}} I \left( \frac{N-\beta}{2} \right) I \left( \frac{N-2}{2} \right)^{\frac{2N + 2\beta}{N-2}}.
$$
Here, we denote $R_N = \frac{\Gamma \left( \frac{N-2}{2} \right)}{4 \pi^{\frac{N}{2}}}$ and $I(s) = \frac{\pi^{\frac{N}{2}} \Gamma \left( \frac{N-2s}{2} \right)}{\Gamma(N-s)}$ for $0 < s < \frac{N}{2}$.
\end{lem}

According to Lemma \ref{CS}, we can derive that the positive classical solution $(U_{\lambda,z}, V_{\lambda,z})$ of system \eqref{hemilton} has the following form:
\begin{equation}\label{CSF0}
	U_{z,\lambda}(x) = V_{z,\lambda}(x) = C_{N,\alpha} \left( \frac{\lambda}{1 + \lambda^{2}|x - z|^2} \right)^{\frac{N-2}{2}},
\end{equation}
where
$$
C_{N,\alpha} = \left( \frac{N(N-2) \Gamma \left( \frac{N+\alpha}{2} \right)}{\pi^{\frac{N}{2}} \Gamma \left( \frac{\alpha}{2} \right)} \right)^{\frac{N-2}{2\alpha+4}}.
$$
For simplicity, let's consider the case where $z = 0$ and $\lambda = 1$, i.e.
\begin{equation}\label{CSF}
	U_{0,1}(x) = V_{0,1}(x) = C_{N,\alpha} \left( \frac{1}{1 + |x|^2} \right)^{\frac{N-2}{2}}.
\end{equation}

The first task of this paper is to thoroughly examine the nondegeneracy property of the positive solutions of system \eqref{hemilton}. It is well known that the following equation
\begin{equation}\label{lcritical}
	-\Delta u = u^{\frac{N+2}{N-2}}, \quad x \in \mathbb{R}^{N},
\end{equation}
has a family of solutions of the following form
\begin{equation}\label{U0}
	U_{\xi,\lambda}(x) := [N(N-2)]^{\frac{N-2}{4}} \left( \frac{\lambda}{1 + \lambda^2 |x - \xi|^2} \right)^{\frac{N-2}{2}}.
\end{equation}
And it has an $(N+1)$-dimensional manifold of solutions given by
$$
\mathcal{Z} = \left\{ z_{\lambda,\xi} = [N(N-2)]^{\frac{N-2}{4}} \left( \frac{\lambda}{\lambda^2 + |x - \xi|^2} \right)^{\frac{N-2}{2}}, \quad \xi \in \mathbb{R}^N, \lambda \in \mathbb{R}^+ \right\}.
$$
If the linearized equation
\begin{equation}\label{Linearized}
	-\Delta v = Z^{\frac{4}{N-2}} v,
\end{equation}
in $D^{1,2}(\mathbb{R}^N)$ only admits solutions of the form
$$
\eta = a D_{\lambda} Z + \mathbf{b} \cdot \nabla Z,
$$
where $a \in \mathbb{R}$, $\mathbf{b} \in \mathbb{R}^N$, and $Z \in \mathcal{Z}$, we say that it is nondegenerate.
For equation \eqref{eq1.5}, the nondegeneracy property is fully established. If $\mu$ is close to $N$, the limiting equation of \eqref{eq1.5} is the critical Lane-Emden equation, whose nondegeneracy property is well-known. Based on this result, the authors demonstrated the nondegeneracy property using an approximation approach in \cite{DY}. Yang and Zhao \cite{YYZ} obtained the nondegeneracy result for the case of $N = 6$ and $\mu = 4$. Li et al. \cite{LXLTX} recently proved that positive bubble solutions for the Hartree equation \eqref{eq1.5} are nondegenerate by employing the key spherical harmonic decomposition and the Funk-Hecke formula of the spherical harmonic functions. In \cite{YGRY}, the author obtained a nondegeneracy result for the critical system \eqref{hemilton} in $N = 6$.

 In this paper, we will generalize the nondegenearcy property for the single Hartree equation to the Hamiltonian type elliptic system with dimension $N \geq 5$.
For the sake of calculation, let's consider the simple form $\left(U_{0,1}(x),V_{0,1}(x)\right)$. By differentiating it with respect to $x$ and $\lambda$ at $(z,\lambda) = (0,1)$, we have
\begin{equation}\label{U1}
	\varphi_{j} := \frac{\partial U_{0,1}}{\partial x_{j}} = (2 - N) U_{0,1}(x) \frac{x_j}{1 + |x|^2}, \quad 1 \leq j \leq N,
\end{equation}
\begin{equation}\label{U2}
	\varphi_{N+1} := \frac{\partial U_{0,1}}{\partial \lambda} = \frac{(N-2)}{2} U_{0,1}(x) + x \cdot \nabla U_{0,1}(x) = \frac{(N-2)}{2} U_{0,1}(x) \frac{1 - |x|^2}{1 + |x|^2},
\end{equation}
\begin{equation}\label{V1}
	\psi_{j} := \frac{\partial V_{0,1}}{\partial x_{j}} = (2 - N) V_{0,1}(x) \frac{x_j}{1 + |x|^2}, \quad 1 \leq j \leq N,
\end{equation}
and
\begin{equation}\label{V2}
	\psi_{N+1} := \frac{\partial V_{0,1}}{\partial \lambda} = \frac{(N-2)}{2} V_{0,1}(x) + x \cdot \nabla V_{0,1}(x) = \frac{(N-2)}{2} V_{0,1}(x) \frac{1 - |x|^2}{1 + |x|^2},
\end{equation}
then the main result of the nondegeneracy property can be expressed as follows:
\begin{thm}\label{Non}
	The linearization system of \eqref{hemilton} around the solution $(U_{0,1}(x), V_{0,1}(x))$ defined by
	\begin{equation}\label{linear}
		\left\{
		\begin{array}{ll}
			&-\Delta \varphi = 2^{*}_{\alpha} \left( |x|^{-(N-\alpha)} \ast V_{0,1}^{2^{*}_{\alpha}-1} \psi \right) V_{0,1}^{2^{*}_{\alpha}-1} + (2^{*}_{\alpha}-1) \left( |x|^{-(N-\alpha)} \ast V_{0,1}^{2^{*}_{\alpha}} \right) V_{0,1}^{2^{*}_{\alpha}-2} \psi
			\quad \mbox{in}  \quad \mathbb{R}^N,\\[1mm]
			&-\Delta \psi = 2^{*}_{\alpha} \left( |x|^{-(N-\alpha)} \ast U_{0,1}^{2^{*}_{\alpha}-1} \varphi \right) U_{0,1}^{2^{*}_{\alpha}-1} + (2^{*}_{\alpha}-1) \left( |x|^{-(N-\alpha)} \ast U_{0,1}^{2^{*}_{\alpha}} \right) U_{0,1}^{2^{*}_{\alpha}-2} \varphi
			\quad \mbox{in}  \quad \mathbb{R}^N,
		\end{array}
		\right.
	\end{equation}
	only admits solutions of the form
	\begin{equation}\nonumber
		\aligned
		(\varphi, \psi) = &\Big( a_1 \left[ \frac{(N-2)}{2} U_{0,1}(x) + x \cdot \nabla U_{0,1}(x) \right] + \sum_{j=1}^{N} b_j \frac{\partial U_{0,1}}{\partial x_j}, \\
		&\quad a_2 \left[ \frac{(N-2)}{2} V_{0,1}(x) + x \cdot \nabla V_{0,1}(x) \right] + \sum_{j=1}^{N} \hat{b}_j \frac{\partial V_{0,1}}{\partial x_j} \Big),
		\endaligned
	\end{equation}
	where $a_1, a_2, b_j, \hat{b}_j \in \mathbb{R}$.
\end{thm}

\begin{Rem}
	Analysis similar to that in Appendix A of \cite{LXLTX1}, we can use the Moser iteration method to show $L^{\infty}(\mathbb{R}^N)$-regularity of the $D^{1,2}(\mathbb{R}^N)$ solution of equation \eqref{linear}.
\end{Rem}

With the conclusion on nondegeneracy, the main purpose of this paper is to investigate the existence of infinitely many solutions of system \eqref{CFL}. To address this problem, we assume that the functions \( K_1 \) and \( K_2 \) are bounded and satisfy the following conditions:

\begin{itemize}
    \item[\textbf{(I)}] The functions \( K_1(r, x'') \) and \( K_2(r, x'') \) have a common critical point \( (r_0, x_0'') \), such that \( r_0 > 0 \), \( K_1(r_0, x_0'') = K_2(r_0, x_0'') = 1 \), and
    \[
    \deg(\nabla(K_1(r, x'') + K_2(r, x'')), (r_0, x_0'')) \neq 0.
    \]
\end{itemize}

\begin{itemize}
    \item[\textbf{(II)}] The functions \( K_1(r, x'') \in C^3(B_{\vartheta}(r_0, x_0'')) \) and \( K_2(r, x'') \in C^3(B_{\vartheta}(r_0, x_0'')) \), where \( \vartheta > 0 \) is sufficiently small, satisfy the conditions
    \[
    \Delta K_1(r_0, x_0'') = \frac{\partial^2 K_1(r_0, x_0'')}{\partial r^2} + \sum_{i=3}^{N} \frac{\partial^2 K_1(r_0, x_0'')}{\partial x_i^2} < 0,
    \]
    and
    \[
    \Delta K_2(r_0, x_0'') = \frac{\partial^2 K_2(r_0, x_0'')}{\partial r^2} + \sum_{i=3}^{N} \frac{\partial^2 K_2(r_0, x_0'')}{\partial x_i^2} < 0.
    \]
\end{itemize}

The main result of this paper is as follows:

\begin{thm}\label{EXS}
    Let \( (r, x'') \in \mathbb{R}^+ \times \mathbb{R}^{N-2} \) and \( N \geq 5 \). If the functions \( K_1(r, x'') \) and \( K_2(r, x'') \) have a stable, topologically nontrivial critical point as described by conditions \(\textbf{(I)}\) and \(\textbf{(II)}\), then the system \eqref{CFL} has infinitely many solutions, and the energy of the solutions diverges to \( +\infty \).
\end{thm}

It is important to note that condition \(\textbf{(II)}\) implies that \( K_1 \) and \( K_2 \) have a saddle point. This case is not covered in \cite{GLP}.

We define
$$\aligned
H_{s}=\Big\{&u\in D^{1,2}(\mathbb{R}^N),u(x_{1},-x_{2},x'')=u(x_{1},x_{2},x''),\\
&\hspace{4mm}u(r\cos\theta,r\sin\theta,x'')=u\Big(r\cos(\theta+\frac{2j\pi}{m}),r\sin(\theta+\frac{2j\pi}{m}),x''\Big)\Big\},
\endaligned$$
and let
\[
z_{j} = \left(\overline{r}\cos\frac{2(j-1)\pi}{m}, \overline{r}\sin\frac{2(j-1)\pi}{m}, \overline{x}''\right), \ j = 1, \dots, m,
\]
where \(\overline{x}''\) is a vector in \(\mathbb{R}^{N-2}\). By the weak symmetry of \(K_1(x)\) and \(K_2(x)\), we have \(K_1(z_j) = K_1(\overline{r}, \overline{x}'')\) and \(K_2(z_j) = K_2(\overline{r}, \overline{x}'')\), for \(j = 1, \dots, m\).
We will use \((U_{z_j,\lambda}, V_{z_j,\lambda})\) as an approximate solution. Let \(\delta > 0\) be a small constant, such that \(K_1(r, x'') > 0\) and \(K_2(r, x'') > 0\) if \(|(r, x'') - (r_0, x_0'')| \leq 10\delta\). Let \(\xi(x) = \xi(|x'|, x'')\) be a smooth function satisfying \(\xi = 1\) if \(|(r, x'') - (r_0, x_0'')| \leq \delta\), \(\xi = 0\) if \(|(r, x'') - (r_0, x_0'')| \geq 2\delta\), and \(0 \leq \xi \leq 1\). Denote
\[
Z_{z_j,\lambda}(x) = \xi U_{z_j,\lambda}(x), \quad
Z_{\overline{r}, \overline{x}'', \lambda}(x) = \sum_{j=1}^{m} Z_{z_j,\lambda}(x), \quad
Z_{\overline{r}, \overline{x}'', \lambda}^{\ast}(x) = \sum_{j=1}^{m} U_{z_j,\lambda}(x),
\]
\[
Y_{z_j,\lambda}(x) = \xi V_{z_j,\lambda}(x), \quad
Y_{\overline{r}, \overline{x}'', \lambda}(x) = \sum_{j=1}^{m} Y_{z_j,\lambda}(x), \quad
Y_{\overline{r}, \overline{x}'', \lambda}^{\ast}(x) = \sum_{j=1}^{m} V_{z_j,\lambda}(x),
\]
\[
Z_{j,1} = \frac{\partial Z_{z_j,\lambda}}{\partial \lambda}, \quad Z_{j,2} = \frac{\partial Z_{z_j,\lambda}}{\partial \overline{r}}, \quad Z_{j,k} = \frac{\partial Z_{z_j,\lambda}}{\partial \overline{x}_{k}''}, \quad k = 3, \dots, N, \quad j = 1, \dots, m,
\]
and
\[
Y_{j,1} = \frac{\partial Y_{z_j,\lambda}}{\partial \lambda}, \quad Y_{j,2} = \frac{\partial Y_{z_j,\lambda}}{\partial \overline{r}}, \quad Y_{j,k} = \frac{\partial Y_{z_j,\lambda}}{\partial \overline{x}_{k}''}, \quad k = 3, \dots, N, \quad j = 1, \dots, m.
\]
In this paper, we always assume that \(m > 0\) is a sufficiently large integer. The parameter \(\lambda \) satisfies \([L_0 m^{\frac{N-2}{N-4}}, L_1 m^{\frac{N-2}{N-4}}]\) for some constants \(L_1 > L_0 > 0\). Moreover,
\[
|(\overline{r}, \overline{x}'') - (r_0, x_0'')| \leq \frac{1}{\lambda^{1-\theta}},
\]
where \(\theta > 0\) is a small constant. We will prove the following result.

\begin{thm}\label{EXS1}
Under the assumptions of Theorem \ref{EXS}, there exists a positive integer \(m_0 > 0\), such that for any integer \(m \geq m_0\), \eqref{CFL} has a solution \((u_m, v_m)\) of the form
\[
u_m = Z_{\overline{r}_m, \overline{x}_m'', \lambda_m} + \phi_{\overline{r}_m, \overline{x}_m'', \lambda_m} = \sum_{j=1}^{m} \xi U_{z_j, \lambda_m} + \phi_{\overline{r}_m, \overline{x}_m'', \lambda_m},
\]
\[
v_m = Y_{\overline{r}_m, \overline{x}_m'', \lambda_m} + \varphi_{\overline{r}_m, \overline{x}_m'', \lambda_m} = \sum_{j=1}^{m} \xi V_{z_j, \lambda_m} + \varphi_{\overline{r}_m, \overline{x}_m'', \lambda_m},
\]
where \(\phi_{\overline{r}_m, \overline{x}_m'', \lambda_m} \in H_s\), \(\varphi_{\overline{r}_m, \overline{x}_m'', \lambda_m} \in H_s\), and \(\lambda_m \in [L_0 m^{\frac{N-2}{N-4}}, L_1 m^{\frac{N-2}{N-4}}]\). Moreover, as \(m \to \infty\), \((\overline{r}_m, \overline{x}_m'') \to (r_0, x_0'')\), \(\lambda_m^{-2} \| \phi_{\overline{r}_m, \overline{x}_m'', \lambda_m} \|_{L^\infty} \to 0\), and \(\lambda_m^{-2} \| \varphi_{\overline{r}_m, \overline{x}_m'', \lambda_m} \|_{L^\infty} \to 0\).
\end{thm}

\begin{Rem}
In general, we are also interested in the following nonlocal critical Lane-Emden system
\begin{align*}
	\left\lbrace
	\begin{aligned}
		-\Delta u&=\left(\int_{\R^N}\frac{v^{p}\left(y\right)}{\left|x-y\right|^{\mu}}dy\right)v^{p-1},&x\in\R^N,\\	
		-\Delta v&=\left(\int_{\R^N}\frac{u^{q}\left(y\right)}{\left|x-y\right|^{\mu}}dy\right)u^{q-1},&x\in\R^N,
	\end{aligned}
	\right.
\end{align*}
where $N\geq3$, $0<\mu<N$ and $p,q>1$ satisfy
\begin{align*}
	\frac{1}{p}+\frac{1}{q}=\frac{2\left(N-2\right)}{2N-\mu}.
\end{align*}
However, little is known about the uniqueness and nondegeneracy of the positive solutions of  the above generalized critical Hamiltonian elliptic system of Hartree type. In fact, the construction of multi-bubbling solutions in the present paper depends a lot on the qualitative properties of these positive solutions, and we only know the uniqueness and nondegeneracy of positive solutions for the special case $p=\frac{2N-\mu}{N-2}$. In a future paper will investigate the uniqueness and nondegeneracy of positive solutions of the general case and construct the multi-bubbling solutions of the following problem
\begin{align*}
	\left\lbrace
	\begin{aligned}
		-\Delta u&=K_1(x)\left(\int_{\R^N}\frac{K_1(x)v^{p}\left(y\right)}{\left|x-y\right|^{\mu}}dy\right)v^{p-1},&x\in\R^N,\\	
		-\Delta v&=K_2(x)\left(\int_{\R^N}\frac{K_2(x)u^{q}\left(y\right)}{\left|x-y\right|^{\mu}}dy\right)u^{q-1},&x\in\R^N,
	\end{aligned}
	\right.
\end{align*}
where $N\geq3$, $0<\mu<N$ and $p,q>1$ satisfy
\begin{align*}
	\frac{1}{p}+\frac{1}{q}=\frac{2\left(N-2\right)}{2N-\mu}.
\end{align*}
The functions \( K_1 \) and \( K_2 \) are bounded and satisfy the conditions \(\textbf{(I)}\) and \(\textbf{(II)}\).
\end{Rem}

This paper is organized as follows: In Section 2, we prove a nondegeneracy result for the critical Hamilton system \eqref{CFL}. In Section 3, we carry out the reduction procedure for the critical Hartree system \eqref{CFL}. Finally, in the last section, we establish local Poho\v{z}aev identities to prove our main results.

\section{A nondegeneracy property}
In this section, we aim to prove the nondegeneracy of the unique positive solution for the nonlinear Hamilton system \eqref{hemilton} with the upper critical exponent.

Inspired by the work in \cite{LXLTX}, we first transform equation \eqref{linear} onto the sphere using stereographic projection. Let $\mathbb{S}^N$ denote the unit sphere in $\mathbb{R}^{N+1}$, and define the mapping $S: \mathbb{R}^N \to \mathbb{S}^N \setminus \left\lbrace 0,0,\dots,0,-1 \right\rbrace$ by
\[
S(x) = \left( \frac{2x}{1+|x|^2}, \frac{1-|x|^2}{1+|x|^2} \right).
\]
We observe that the inverse mapping $S^{-1}$ is given by
\[
S^{-1}(\xi_1, \xi_2, \dots, \xi_{N+1}) = \left( \frac{\xi_1}{1 + \xi_{N+1}}, \frac{\xi_2}{1 + \xi_{N+1}}, \dots, \frac{\xi_N}{1 + \xi_{N+1}} \right).
\]
The Jacobian of this transformation is
\[
J_S(x) = \left( \frac{2}{1 + |x|^2} \right)^N.
\]
By straightforward calculation, we obtain the following distance relation:
\[
|S(x) - S(y)| = |x - y| r(x) r(y),
\]
where \( r(x) = \left( \frac{2}{1 + |x|^2} \right)^{\frac{1}{2}} \).

We define the following two functions. For any \( h: \mathbb{R}^N \to \mathbb{R} \), we define \( S_{*}h: \mathbb{S}^N \setminus \left\lbrace 0, 0, \dots, 0, -1 \right\rbrace \to \mathbb{R} \) as follows:
\begin{equation}\label{S2}
    S_{*}h(\xi) = J_S^{\frac{2-N}{2N}}(S^{-1}(\xi)) h(S^{-1}(\xi)).
\end{equation}
Additionally, for any \( H: \mathbb{S}^N \setminus \left\lbrace 0, 0, \dots, 0, -1 \right\rbrace \to \mathbb{R} \), we define \( S^{*}H: \mathbb{R}^N \to \mathbb{R} \) as follows:
\[
S^{*}H(x) = J_S^{\frac{N-2}{2N}}(x) H(Sx).
\]

Moreover, according to Proposition 2.1 in \cite{LXLTX}, we have the following result:

\begin{Prop}\label{L1}
    Let \( S_{*}: \mathbb{S}^N \setminus \left\lbrace 0, 0, \dots, 0, -1 \right\rbrace \to \mathbb{R} \) be defined by \eqref{S2}. Then, for any \( \xi \in \mathbb{S}^N \), we have
    \[
    S_{*}(\varphi_{j})(\xi) = (2-N) 2^{\frac{-N}{2}} C_{N,\alpha} \xi_{j}, \quad 1 \leq j \leq N,
    \]
    and
    \[
    S_{*}(\varphi_{N+1})(\xi) = (N-2) 2^{\frac{-N}{2}} C_{N,\alpha} \xi_{N+1},
    \]
    where \( \varphi_{j}, 1 \leq j \leq N+1 \) are defined by \eqref{U1} and \eqref{U2}.
\end{Prop}

%\begin{proof}
%An easy computation shows that
%$$
%\left|S^{-1} %\xi\right|^2=\frac{1-\xi_{N+1}^2}{(1+\xi_{N+1})^2},
%$$
%then
%\begin{equation}\nonumber
	%\begin{split}
	%	f(\varphi_{j})(\xi)&=(\frac{2}{1+|S^{-1}(\xi)|^2})^{\frac{4-N}{2}}(2-N)C_{N,\alpha}(\frac{1}{1+|S^{-1}(\xi)|^2})^{\frac{N-4}{2}}\frac{(S^{-1}(\xi))_j}{1+|S^{-1}(\xi)|^2}\\
	%	&=2^{\frac{2-N}{2}}(2-N)C_{N,\alpha}\left(\frac{\xi_{j}}{1+\xi_{N+1}} \right)\left(\frac{1+\xi_{N+1}}{2}\right) \\
	%	&=(2-N)2^{\frac{-N}{2}}C_{N,\alpha}\xi_{j},\sp 1\leq j\leq N.
%	\end{split}
%\end{equation}
%Similarly, we get $f(\varphi_{N+1})(\xi)=(N-2)2^{\frac{-N}{2}}C_{N,\alpha}\xi_{N+1}$.
%\end{proof}

 We will now discuss some results related to spherical harmonic functions. Let \( Y_{k}(x) \) be the restriction on \( \mathbb{S}^N \) of real harmonic \( k \)-order homogeneous polynomials. The linear space of \( \left\lbrace Y_{k,i}(x) \mid 1 \leq i \leq \dim \mathscr {H}_{k}^{N+1} \right\rbrace_{k} \) is denoted by \( \mathscr {H}_{k}^{N+1} \), which is the mutually orthogonal subspace of the restriction on \( \mathbb{S}^N \). We consider the following orthogonal decomposition:
$$
L^2(\mathbb{S}^N) = \bigoplus\limits_{k=0}^{\infty} \mathscr {H}_{k}^{N+1},
$$
and
\begin{equation}\nonumber
    \dim \mathscr {H}_{k}^{N+1} =
    \begin{cases}
        1, & \text{if } k = 0, \\
        N+1, & \text{if } k = 1, \\
        \binom{k+N}{k} - \binom{k-2}{k+N-2}, & \text{if } k \geq 2.
    \end{cases}
\end{equation}

The following Funk-Heck formula of the spherical harmonic function plays a key role in the proof of Theorem \ref{Non}.

\begin{lem}\label{FH01}{\rm{(\cite{AH,DX})}}
Let \( t \in (0,N) \), then for any \( Y \in \mathscr {H}_{k}^{N+1} \), we have
\begin{equation}\label{FH}
\int_{\mathbb{S}^N} \frac{Y(\eta)}{|\xi-\eta|^t} \, d\eta = \lambda_k(t) Y(\xi),
\end{equation}
where
\[
\lambda_k(t) = 2^{N-t} \pi^{\frac{N}{2}} \frac{\Gamma\left(k + \frac{t}{2}\right) \Gamma\left(\frac{N-t}{2}\right)}{\Gamma\left(\frac{t}{2}\right) \Gamma\left(k + N - \frac{t}{2}\right)}.
\]
\end{lem}
We observe immediately that \( \lambda_{k+1}(t) < \lambda_k(t) \) for every \( t \in (0,N) \), and obviously,
\[
\lambda_0(N-2) = \frac{2^3}{N} \frac{\pi^{\frac{N}{2}}}{\Gamma\left(\frac{N}{2}\right)},
\quad
\lambda_1(N-2) = \frac{N-2}{N+2} \lambda_0(N-2),
\]
\[
\lambda_0(N-\alpha) = 2^{\alpha} \pi^{\frac{N}{2}} \frac{\Gamma\left(\frac{\alpha}{2}\right)}{\Gamma\left(\frac{N+\alpha}{2}\right)},
\quad
\lambda_1(N-\alpha) = \frac{N-\alpha}{N+\alpha} \lambda_0(N-\alpha).
\]

The following lemma is a direct consequence of the above Funk-Heck formula.
\begin{lem}\label{FH11}{\rm{(\cite{LXLTX})}}
Let \( \alpha \in (0,N) \) and \( \lambda_k(N-\alpha) \) be defined as above, then for any \( Y \in \mathscr {H}_k^{N+1} \), we have
\begin{equation}\label{FH111}
\int_{\mathbb{S}^N} \int_{\mathbb{S}^N} \frac{1}{|\xi - \eta|^{N-2}} \frac{1}{|\eta - \sigma|^{N-\alpha}} Y(\eta) \, d\eta \, d\sigma = \lambda_k(N-2) \lambda_0(N-\alpha) Y(\xi),
\end{equation}
\begin{equation}\label{FH112}
\int_{\mathbb{S}^N} \int_{\mathbb{S}^N} \frac{1}{|\xi - \eta|^{N-2}} \frac{1}{|\eta - \sigma|^{N-\alpha}} Y(\sigma) \, d\sigma \, d\eta = \lambda_k(N-2) \lambda_k(N-\alpha) Y(\xi).
\end{equation}
\end{lem}

Next, we will provide some useful estimates. For convenience, define
\begin{equation}\label{T}
	\begin{aligned}
		T_1(\psi)(x) & := 2^{*}_{\alpha} \Big( |x|^{-(N-\alpha)} \ast V_{0,1}^{2^{*}_{\alpha}-1} \psi \Big) V_{0,1}^{2^{*}_{\alpha}-1} + (2^{*}_{\alpha}-1) \Big( |x|^{-(N-\alpha)} \ast V_{0,1}^{2^{*}_{\alpha}} \Big) V_{0,1}^{2^{*}_{\alpha}-2} \psi, \\
		T_2(\varphi)(x) & := 2^{*}_{\alpha} \Big( |x|^{-(N-\alpha)} \ast U_{0,1}^{2^{*}_{\alpha}-1} \varphi \Big) U_{0,1}^{2^{*}_{\alpha}-1} + (2^{*}_{\alpha}-1) \Big( |x|^{-(N-\alpha)} \ast U_{0,1}^{2^{*}_{\alpha}} \Big) U_{0,1}^{2^{*}_{\alpha}-2} \varphi.
	\end{aligned}
\end{equation}
Let \( \nu \in [0, N-2] \) and \( \tau(x) = (1 + |x|^2)^{\frac{1}{2}} \). If \( |\psi| \lesssim \frac{1}{\tau(x)^\nu} \) and \( |\varphi| \lesssim \frac{1}{\tau(x)^\nu} \), we can obtain the following estimates for both \( T_1(\psi) \) and \( T_2(\varphi) \):
\begin{equation}\label{Est1}
	|T_1(\psi)(x)| \lesssim \frac{1}{\tau(x)^{\nu+4}},
\end{equation}
and
\begin{equation}
	|T_2(\varphi)(x)| \lesssim \frac{1}{\tau(x)^{\nu+4}}
\end{equation}
by Lemma 5.1 in \cite{LXLTX}.
By the Riesz potential theory in \cite{S}, equation \eqref{linear} can be rewritten as:
\begin{equation}\label{Inte}
	\left\{
	\begin{array}{ll}
		\varphi(x) = C_N \int_{\mathbb{R}^N} \frac{1}{|x - y|^{N-2}} T_1(\psi)(y) \, dy, \\
		\psi(x) = C_N \int_{\mathbb{R}^N} \frac{1}{|x - y|^{N-2}} T_2(\varphi)(y) \, dy,
	\end{array}
	\right.
\end{equation}
where \( C_N = \frac{\Gamma\left( \frac{N}{2} \right)}{2(N-2)\pi^{\frac{N}{2}}} \).
By Lemma 5.2 in \cite{LXLTX}, we can improve the estimates of \( \psi(x) \) and \( \varphi(x) \) as follows:
\begin{equation}
	|\psi(x)| \lesssim \frac{1}{\tau(x)^{N-2}}, \quad |\varphi(x)| \lesssim \frac{1}{\tau(x)^{N-2}}.
\end{equation}

To prove the theorem, we need to transform the integral equation \eqref{Inte} on \( \mathbb{R}^N \) to one on \( \mathbb{S}^N \) using stereographic projection. For simplicity, let's consider the first equation in system \eqref{Inte} (the second equation is analogous). According to \eqref{S2} and the estimate \( |\varphi(x)| \lesssim \frac{1}{\tau(x)^{N-2}} \), we have:
\[
	\int_{\mathbb{S}^N} \left| S_{*} \varphi(\xi) \right|^2 \, d\xi
	= \int_{\mathbb{S}^N} J_S^{\frac{2-N}{N}} (S^{-1}(\xi)) |\varphi(S^{-1}(\xi))|^2 \, d\xi
	\lesssim \int_{\mathbb{R}^N} \frac{1}{\tau(x)^{2N-4}} J_S^{\frac{2}{N}}(x) \, dx
	\lesssim \int_{\mathbb{R}^N} \frac{1}{\tau(x)^{2N}} \, dx < +\infty.
\]
Thus, \( S_{*} \varphi \in L^2(\mathbb{S}^N) \).

Using \eqref{CSF} and the fact that \( J_S(x) = \left( \frac{2}{1 + |x|^2} \right)^N \), we obtain:
\begin{equation}\nonumber
	\begin{split}
		& 2^{*}_{\alpha} \Big( |x|^{-(N-\alpha)} \ast V_{0,1}^{2^{*}_{\alpha}-1} \psi \Big) V_{0,1}^{2^{*}_{\alpha}-1} \\
		& = C_{N,\alpha}^{\frac{2\alpha + 4}{N-2}} 2^{*}_{\alpha} \cdot 2^{-(2^{*}_{\alpha}-1)(N-2)} J_S^{\frac{\alpha+2}{2N}}(x) \int_{\mathbb{R}^N} \frac{1}{|x - y|^{N-\alpha}} J_S^{\frac{\alpha+2}{2N}}(y) \psi(y) \, dy \\
		& = C_{N,\alpha}^{\frac{2\alpha + 4}{N-2}} 2^{*}_{\alpha} \cdot 2^{-(2^{*}_{\alpha}-1)(N-2)} J_S^{\frac{N+2}{2N}}(x) \int_{\mathbb{R}^N} \frac{1}{|Sx - Sy|^{N-\alpha}} J_S^{\frac{2-N}{2N}}(y) \psi(y) J_S(y) \, dy \\
		& = C_{N,\alpha}^{\frac{2\alpha + 4}{N-2}} 2^{*}_{\alpha} \cdot 2^{-(2^{*}_{\alpha}-1)(N-2)} J_S^{\frac{N+2}{2N}}(x) \int_{\mathbb{S}^N} \frac{1}{|Sx - \eta|^{N-\alpha}} S_{*} \psi(\eta) \, d\eta.
	\end{split}
\end{equation}
Therefore, we have:
\begin{equation}\label{ST1}
	\begin{split}
		& \int_{\mathbb{R}^N} \frac{1}{|x - y|^{N-2}} 2^{*}_{\alpha} \Big( |x|^{-(N-\alpha)} \ast V_{0,1}^{2^{*}_{\alpha}-1} \psi \Big) V_{0,1}^{2^{*}_{\alpha}-1} \, dy \\
		& = C_{N,\alpha}^{\frac{2\alpha + 4}{N-2}} 2^{*}_{\alpha} \cdot 2^{-(2^{*}_{\alpha}-1)(N-2)} \int_{\mathbb{R}^N} \frac{1}{|x - y|^{N-2}} J_S^{\frac{N+2}{2N}}(y) \int_{\mathbb{S}^N} \frac{1}{|Sy - \sigma|^{N-\alpha}} S_{*} \psi(\sigma) \, d\sigma \, dy \\
		& = C_{N,\alpha}^{\frac{2\alpha + 4}{N-2}} 2^{*}_{\alpha} \cdot 2^{-(2^{*}_{\alpha}-1)(N-2)} J_S^{\frac{N-2}{2N}}(x) \int_{\mathbb{R}^N} \frac{1}{|Sx - Sy|^{N-2}} J_S(y) \int_{\mathbb{S}^N} \frac{1}{|Sy - \sigma|^{N-\alpha}} S_{*} \psi(\sigma) \, d\sigma \, dy \\
		&=C_{N,\alpha}^{\frac{2\alpha+4}{N-2}}2^{*}_{\alpha}\cdot 2^{-(2^{*}_{\alpha}-1)(N-2)}J_{S}^{\frac{N-2}{2N}}(x) \int_{\mathbb{S}^N}\int_{\mathbb{S} ^N}\frac{1}{|Sx-\eta|^{N-2}}\frac{1}{|\eta-\sigma|^{N-\alpha}}S_{*}\psi(\sigma)d\sigma d\eta.
	\end{split}
\end{equation}
Similarly,
\begin{equation}\label{ST2}
	\begin{split}
		&\int_{\R^N}\frac{1}{|x-y|^{N-2}}(2^{*}_{\alpha}-1)\Big(|x|^{-(N-\alpha)}\ast V_{0,1}^{2^{*}_{\alpha}}\Big)V_{0,1}^{2^{*}_{\alpha}-2}\psi dy\\
		&=C_{N,\alpha}^{\frac{2\alpha+4}{N-2}}(2^{*}_{\alpha}-1)\cdot 2^{-(2^{*}_{\alpha}-1)(N-2)}J_{S}^{\frac{N-2}{2N}}(x) \int_{\mathbb{S}^N}\frac{1}{|Sx-\eta|^{N-2}}S_{*}\psi(\eta)\int_{\mathbb{S} ^N}\frac{1}{|\eta-\sigma|^{N-\alpha}}d\eta d\sigma .
	\end{split}
\end{equation}
Hence that
\begin{equation}\nonumber
	\begin{aligned}
		\varphi(x)
		&=C_{N}C_{N,\alpha}^{\frac{2\alpha+4}{N-2}}2^{-(2^{*}_{\alpha}-1)(N-2)}J_{S}^{\frac{N-2}{2N}}(x)\Big[2^{*}_{\alpha} \int_{\mathbb{S}^N}\int_{\mathbb{S} ^N}\frac{1}{|Sx-\eta|^{N-2}}\frac{1}{|\eta-\sigma|^{N-\alpha}}S_{*}\psi(\sigma)d\sigma d\eta \\  &+(2^{*}_{\alpha}-1)\int_{\mathbb{S}^N}\frac{1}{|Sx-\eta|^{N-2}}S_{*}\psi(\eta)\int_{\mathbb{S} ^N}\frac{1}{|\eta-\sigma|^{N-\alpha}}d\eta d\sigma \Big]\\
		&:=C_{N}C_{N,\alpha}^{\frac{2\alpha+4}{N-2}}2^{-(2^{*}_{\alpha}-1)(N-2)}J_{S}^{\frac{N-2}{2N}}(x)T_{\mathbb{S}^N}S_{*}\psi(Sx),
	\end{aligned}
\end{equation}
where $$\aligned
T_{\mathbb{S}^N}S_{*}\psi(Sx)&=2^{*}_{\alpha} \int_{\mathbb{S}^N}\int_{\mathbb{S} ^N}\frac{1}{|Sx-\eta|^{N-2}}\frac{1}{|\eta-\sigma|^{N-\alpha}}S_{*}\psi(\sigma)d\sigma d\eta\\
&+(2^{*}_{\alpha}-1)\int_{\mathbb{S}^N}\frac{1}{|Sx-\eta|^{N-2}}S_{*}\psi(\eta)\int_{\mathbb{S} ^N}\frac{1}{|\eta-\sigma|^{N-\alpha}}d\eta d\sigma,
\endaligned $$
and finally that
\begin{equation}\nonumber
	S_{*}\varphi(\xi)=C_{N}C_{N,\alpha}^{\frac{2\alpha+4}{N-2}}2^{-(2^{*}_{\alpha}-1)(N-2)}T_{\mathbb{S}^N}S_{*}\psi(\xi).
\end{equation}
Apply this argument again, we can obtain that
\begin{equation}\label{Inte2}
	\left\{\begin{array}{ll}
		&		S_{*}\varphi(\xi)=C_{N}C_{N,\alpha}^{\frac{2\alpha+4}{N-2}}2^{-(2^{*}_{\alpha}-1)(N-2)}T_{\mathbb{S}^N}S_{*}\psi(\xi),\\[1mm]
		&		S_{*}\psi(\xi)=C_{N}C_{N,\alpha}^{\frac{2\alpha+4}{N-2}}2^{-(2^{*}_{\alpha}-1)(N-2)}T_{\mathbb{S}^N}S_{*}\varphi(\xi),
	\end{array}\right.
\end{equation}

{\bf Proof of Theorem \ref{Non}.}
According to the above arguments we know $S_{*}\varphi,S_{*}\psi\in L^2(\mathbb{S}^N)$. By the orthogonal decomposition $L^2(\mathbb{S}^N)=\bigoplus\limits_{k=0}^{\infty}\mathscr {H}_{k}^{N+1}$, we have
\begin{equation}\label{OD1}
	S_{*}\varphi(\xi)=\sum_{k=0}^{\infty}\sum_{i=1}^{\dim \mathscr {H}_{k}^{N+1}}S_{*}\varphi_{k,i}Y_{k.i}(\xi),
\end{equation}
and
\begin{equation}\label{OD2}
	S_{*}\psi(\xi)=\sum_{k=0}^{\infty}\sum_{i=1}^{\dim \mathscr {H}_{k}^{N+1}}S_{*}\psi_{k,i}Y_{k.i}(\xi),
\end{equation}
where $S_{*}\varphi_{k,i}=\int_{\mathbb{S}^N}S_{*}\varphi(\xi)Y_{k.i}(\xi)d\xi$ and $S_{*}\psi_{k,i}=\int_{\mathbb{S}^N}S_{*}\psi(\xi)Y_{k.i}(\xi)d\xi$.

On the other hand, since $S_{*}\varphi(\xi)$ and $S_{*}\psi(\xi)$ satisfy \eqref{Inte2}, by using Lemma \ref{FH11}, we can deduce that
\begin{equation}\label{Inte3}
	\left\{\begin{array}{ll}
		&		 S_{*}\varphi_{k,i}=C_{N}C_{N,\alpha}^{\frac{2\alpha+4}{N-2}}2^{-(2^{*}_{\alpha}-1)(N-2)}\lambda_{k}(N-2)\left[2^{*}_{\alpha}\lambda_{k}(N-\alpha)+(2^{*}_{\alpha}-1)\lambda_{0}(N-\alpha) \right]S_{*}\psi_{k,i} ,\\[1mm]
		&		 S_{*}\psi_{k,i}=C_{N}C_{N,\alpha}^{\frac{2\alpha+4}{N-2}}2^{-(2^{*}_{\alpha}-1)(N-2)}\lambda_{k}(N-2)\left[2^{*}_{\alpha}\lambda_{k}(N-\alpha)+(2^{*}_{\alpha}-1)\lambda_{0}(N-\alpha) \right]S_{*}\varphi_{k,i} ,
	\end{array}\right.
\end{equation}
Putting the second equality into the first one in \eqref{Inte3}, we must have $$\left\lbrace C_{N}C_{N,\alpha}^{\frac{2\alpha+4}{N-2}}2^{-(2^{*}_{\alpha}-1)(N-2)}\lambda_{k}(N-2)\left[2^{*}_{\alpha}\lambda_{k}(N-\alpha)+(2^{*}_{\alpha}-1)\lambda_{0}(N-\alpha) \right]\right\rbrace  ^{2}=1.$$
Obviously, $ C_{N}C_{N,\alpha}^{\frac{2\alpha+4}{N-2}}2^{-(2^{*}_{\alpha}-1)(N-2)}\lambda_{k}(N-2)\left[2^{*}_{\alpha}\lambda_{k}(N-\alpha)+(2^{*}_{\alpha}-1)\lambda_{0}(N-\alpha) \right]>0$, and so we know
$$ C_{N}C_{N,\alpha}^{\frac{2\alpha+4}{N-2}}2^{-(2^{*}_{\alpha}-1)(N-2)}\lambda_{k}(N-2)\left[2^{*}_{\alpha}\lambda_{k}(N-\alpha)+(2^{*}_{\alpha}-1)\lambda_{0}(N-\alpha) \right]=1.$$
An easy computation shows that the aboe equality holds if and only if $k=1$,
thus $S_{*}\varphi,S_{*}\psi\in \mathbb{Y}^{N+1}_{1}$.

Notice that the dimension of space $\mathbb{Y}^{N+1}_{1}$ is $N+1$, we have
$$
S_{*}\varphi\in \text{span}\left\lbrace \xi_{j}|\; 1\leq j\leq N+1 \; \right\rbrace .
$$
From Lemma \ref{L1}, we know the map $S_{*}$ is one to one from the subspace $\text{span}\left\lbrace \varphi_{j},\; 1\leq j\leq N+1 \; \right\rbrace$ to the subspace $\mathscr {H}_{1}^{N+1}$, and so is the weighted pullback map $S^{*}:\mathscr {H}_{1}^{N+1}\mapsto \text{span}\left\lbrace \varphi_{j},\; 1\leq j\leq N+1 \right\rbrace$, this gives
$$
\varphi\in \text{span}\left\lbrace \varphi_{j}|\; 1\leq j\leq N+1 \; \right\rbrace .
$$
The same reasoning applies to
$$
\psi\in \text{span}\left\lbrace \psi_{j}|\; 1\leq j\leq N+1 \; \right\rbrace .
$$
This completes the proof of the theorem.
$\hfill{} \Box$

\section{Finite-dimensional reduction}
In this section, we perform the finite-dimensional reduction argument in a weighted space to construct a reasonably good approximate solution.
First, we are going to give some estimates involving the convolution term.
\begin{lem}\label{B2}{\rm{(Lemma B.1, \cite{WY1})}}  For each fixed $k$ and $j$, $k\neq j$,  let
	$$
	g_{k,j}(x)=\frac{1}{(1+|x-z_{j}|)^{\alpha}}\frac{1}{(1+|x-z_{k}|)^{\beta}},
	$$
	where $\alpha\geq 1$ and $\beta\geq1$ are two constants.
	Then, for any constants $0<\delta\leq\min\{\alpha,\beta\}$, there is a constant $C>0$, such that
	$$
	g_{k,j}(x)\leq\frac{C}{|z_{k}-z_{j}|^{\delta}}\Big(\frac{1}{(1+|x-z_{j}|)^{\alpha+\beta-\delta}}+\frac{1}{(1+|x-z_{k}|)^{\alpha+\beta-\delta}}\Big).
	$$
\end{lem}

\begin{lem}\label{B3}{\rm{(Lemma B.2, \cite{WY1})}}  For any constant $0<\delta<N-2$, $N\geq5$, there is a constant $C>0$, such that
	$$
	\int_{\mathbb{R}^{N}}\frac{1}{|x-y|^{N-2}}\frac{1}{(1+|y|)^{2+\delta}}dy\leq \frac{C}{(1+|x|)^{\delta}}.
	$$
\end{lem}

\begin{lem}\label{B4}{\rm{(Lemma A.3, \cite{CYZ})}}
	For $N\geq5$ and $1\leq i\leq m$, there is a constant $C>0$, such that
	$$
	|x|^{-\alpha}\ast \frac{\lambda^{N-\frac{\alpha}{2}}}{(1+\lambda|x-z_{i}|)^{N-\alpha+\eta}}\leq \frac{C\lambda^{\frac{\alpha}{2}}}{(1+\lambda|x-z_{i}|)^{\min\left\lbrace \alpha,\eta\right\rbrace }},
	$$
	where $\eta>0$.
\end{lem}

\begin{lem}\label{B6}{\rm{(Lemma A.4, \cite{CYZ})}}
For $N\geq5$ and $1\leq i\leq m$, there is a constant $C>0$, such that
	\begin{equation}\label{p1}
		|x|^{-\alpha}\ast|U_{z_i,\lambda}|^{2_{\alpha}^{*}}
		=C\left( \frac{\lambda}{1+\lambda^{2}|x-z_{i}|^{2}}\right)^{\frac{\alpha}{2}}.
	\end{equation}
\end{lem}
%\begin{proof}
%	By identity (37) in \cite{DHQWF}
%	\begin{equation}
%		\int_{\R^N}\frac{1}{|x-y|^{2s}}\left(\frac{1}{1+|y|^2} \right) ^{N-s}dy=I(s)\left(\frac{1}{1+|x|^2} \right) ^{s},\; 0<s<\frac{N}{2},
%	\end{equation}
%where
%$$
%I(s)=\frac{\pi^{\frac{N}{2}}\Gamma(\frac{N-2s}{2})}{\Gamma(N-s)},\; \text{and}\;\Gamma(s)=\int_{0}^{+\infty}x^{s-1}e^{-x}dx,\;s>0,
%$$
%and \eqref{CSF0} we can obtain
%\begin{equation}\nonumber
%	|x|^{-(N-\alpha)}\ast|U_{z_i,\lambda}|^{2_{\alpha}^{*}}
	%=\int_{\R^N}\frac{U_{z_i,\lambda}^{2_{\alpha}^{*}}(y)}{|x-y|^{N-\alpha}}dy=C\left( \frac{\lambda}{1+\lambda^{2}|x-z_{i}|^{2}}\right)^{\frac{N-\alpha}{2}},
%\end{equation}
%where $N\geq5$ and $C=I(\frac{N-\alpha}{2})C_{N,\alpha}^{2_{\alpha}^{*}}$.
%\end{proof}

 To carry out the reduction arguments, we need to introduce some suitable working spaces. Let
$$
\|u\|_{\ast}=\sup_{x\in\mathbb{R}^N}\Big(\sum_{j=1}^{m}
\frac{1}{(1+\lambda|x-z_{j}|)^{\frac{N-2}{2}+\tau}}\Big)^{-1}\lambda^{-\frac{N-2}{2}}|u(x)|,
$$
and
$$
\|h\|_{\ast\ast}=\sup_{x\in\mathbb{R}^N}\Big(\sum_{j=1}^{m}
\frac{1}{(1+\lambda|x-z_{j}|)^{\frac{N+2}{2}+\tau}}\Big)^{-1}\lambda^{-\frac{N+2}{2}}|h(x)|,
$$
where $\tau=1+\overline{\eta},\ \overline{\eta}>0$. Set
$$
\|(u,v)\|_{\ast}=\|u\|_{\ast}+\|v\|_{\ast}
\ \ \mbox{and } \ \
\|(h,g)\|_{\ast\ast}=\|h\|_{\ast\ast}+\|g\|_{\ast\ast}.
$$

Let's consider the following system
\begin{equation}\label{c1}
	\left\{\begin{array}{ll}
		\mathcal{L}_m(\phi,\varphi)=(h,g)+\sum_{l=1}^{N}c_{l}\sum_{j=1}^{m}\Big((2^{*}_{\alpha}-1)\Big(|x|^{-(N-\alpha)}\ast  |Y_{z_j,\lambda}|^{2^{*}_{\alpha}}\Big)Y_{z_j,\lambda}^{2^{*}_{\alpha}-2}Y_{j,l}+2^{*}_{\alpha}\Big(|x|^{-(N-\alpha)}\ast  Y_{z_j,\lambda}^{2^{*}_{\alpha}-1}Y_{j,l}\Big)Y_{z_j,\lambda}^{2^{*}_{\alpha}-1},\\
		\displaystyle \hspace{14.14mm}(2^{*}_{\alpha}-1)\Big(|x|^{-(N-\alpha)}\ast |Z_{z_j,\lambda}|^{2^{*}_{\alpha}}\Big)Z_{z_j,\lambda}^{2^{*}_{\alpha}-2}Z_{j,l}+2^{*}_{\alpha}\Big(|x|^{-(N-\alpha)}\ast  Z_{z_j,\lambda}^{2^{*}_{\alpha}-1}Z_{j,l}\Big)Z_{z_j,\lambda}^{2^{*}_{\alpha}-1}
		\Big),\\	
		(\phi,\varphi)\in H_{s}\times H_{s},\\
\sum_{j=1}^{m}\Big< (2^{*}_{\alpha}-1)\Big(|x|^{-(N-\alpha)}\ast |Y_{z_j,\lambda}|^{2^{*}_{\alpha}}\Big)Y_{z_j,\lambda}^{2^{*}_{\alpha}-1}Y_{j,l}+2^{*}_{\alpha}\Big(|x|^{-(N-\alpha)}\ast (Y_{z_j,\lambda}^{2^{*}_{\alpha}-1}
Y_{j,l})\Big)Y_{z_j,\lambda}^{2^{*}_{\alpha}-1},\\
\displaystyle \hspace{14.14mm} (2^{*}_{\alpha}-1)\Big(|x|^{-(N-\alpha)}\ast |Z_{z_j,\lambda}|^{2^{*}_{\alpha}}\Big)Z_{z_j,\lambda}^{2^{*}_{\alpha}-1}Z_{j,l}+2^{*}_{\alpha}\Big(|x|^{-(N-\alpha)}\ast (Z_{z_j,\lambda}^{2^{*}_{\alpha}-1}
Z_{j,l})\Big)Z_{z_j,\lambda}^{2^{*}_{\alpha}-1}, (\phi,\varphi)\Big> =0,\\
\displaystyle \hspace{14.14mm} l=1,2,\dots,N.
	\end{array}\right.
\end{equation}
for some real numbers $c_{l}$,
where $\left\langle (u_1,u_2),(v_1,v_2)\right\rangle =\int_{\R^N}(u_1 v_1 +u_2 v_2)$ and
\begin{equation}
	\aligned
	&\mathcal{L}_m(\phi,\varphi)\\&=\Big(-\Delta \phi
	-(2^{*}_{\alpha}-1)K_{1}(x)\Big(|x|^{-(N-\alpha)}\ast K_{1}(x)|Y_{\overline{r},\overline{x}'',\lambda}|^{2^{*}_{\alpha}}\Big)Y_{\overline{r},\overline{x}'',\lambda}^{2^{*}_{\alpha}-2}\varphi-2^{*}_{\alpha}K_{1}(x)\Big(|x|^{-(N-\alpha)}\ast K_{1}(x)(Y_{\overline{r},\overline{x}'',\lambda}^{2^{*}_{\alpha}-1}
	\varphi)\Big)Y_{\overline{r},\overline{x}'',\lambda}^{2^{*}_{\alpha}-1},
	\\&-\Delta \varphi-(2^{*}_{\alpha}-1)K_{2}(x)\Big(|x|^{-(N-\alpha)}\ast K_{2}(x)|Z_{\overline{r},\overline{x}'',\lambda}|^{2^{*}_{\alpha}}\Big)Z_{\overline{r},\overline{x}'',\lambda}^{2^{*}_{\alpha}-2}\phi-2^{*}_{\alpha}K_{2}(x)\Big(|x|^{-(N-\alpha)}\ast K_{2}(x)(Z_{\overline{r},\overline{x}'',\lambda}^{2^{*}_{\alpha}-1}
	\phi)\Big)Z_{\overline{r},\overline{x}'',\lambda}^{2^{*}_{\alpha}-1}
	\Big).
	\endaligned
\end{equation}

\begin{lem}\label{C1}
	Suppose that $(\phi_{m},\varphi_{m})$ solves \eqref{c1} for $(h,g) = (h_m,g_{m})$. If $\|(h_m,g_{m})\|_{\ast\ast}\to 0$ as $m\to +\infty$, then $\|(\phi_{m},\varphi_{m})\|_{\ast}\to 0$.
\end{lem}
\begin{proof}
	To obtain a contrary, suppose that there exist $m\rightarrow+\infty$, $\overline{r}_m\rightarrow r_{0}$, $\overline{y}_{m}''\rightarrow y_0''$, $\lambda_{m}\in[L_{0}m^{\frac{N-2}{N-4}},L_{1}m^{\frac{N-2}{N-4}}]$ and $(\phi_{m},\varphi_{m})$ solving \eqref{c1} for $(h,g) = (h_m,g_{m})$, $\lambda=\lambda_{m}$, $\overline{r}=\overline{r}_{m}$, $\overline{y}''=\overline{y}_{m}''$,
	with $\|(h_m,g_{m})\|_{\ast\ast}\rightarrow 0$ and $\|(\phi_{m},\varphi_{m})\|_{\ast}\geq c>0$. Without loss of generality, we can assume $\|(\phi_{m},\varphi_{m})\|_{\ast}=1$.
	
	According to \eqref{c1}, we have
\begin{equation}\label{c2}
	\aligned
	&	|\phi_m(x)|\\&\leq
	C\left| \int_{\mathbb{R}^N}\frac{K_{1}(y)}{|y-x|^{N-2}}\Big(|y|^{-(N-\alpha)}\ast K_{1}(y) (Y_{\overline{r},\overline{x}'',\lambda}^{2^{*}_{\alpha}-1}
	\varphi_m)\Big)Y_{\overline{r},\overline{x}'',\lambda}^{2^{*}_{\alpha}-1}(y)dy\right| \\
	&+C\left| \int_{\mathbb{R}^N}\frac{K_{1}(y)}{|y-x|^{N-2}}\Big(|y|^{-(N-\alpha)}\ast K_{1}(y)|Y_{\overline{r},\overline{x}'',\lambda}|^{^{2^{*}_{\alpha}}}\Big)Y_{\overline{r},\overline{x}'',\lambda}^{2^{*}_{\alpha}-2}\varphi_m(y)dy\right|
	\\ &+C\sum_{l=1}^{N}|c_{l}|\Bigg[\Big|\sum_{j=1}^{m}\int_{\mathbb{R}^N}\frac{1}{|y-x|^{N-2}}\Big(|y|^{-(N-\alpha)}\ast (Y_{z_j,\lambda}^{2^{*}_{\alpha}-1} Y_{j,l})\Big)Y_{z_j,\lambda}^{2^{*}_{\alpha}-1}(y)dy\Big|\\
	&+\Big|\sum_{j=1}^{m}\int_{\mathbb{R}^N}\frac{1}{|y-x|^{N-2}}\Big(|y|^{-(N-\alpha)}\ast |Y_{z_j,\lambda}|^{2^{*}_{\alpha}}\Big)Y_{z_j,\lambda}^{2^{*}_{\alpha}-2}Y_{j,l}(y)dy\Big|\Bigg]+C\int_{\mathbb{R}^N}\frac{1}{|y-x|^{N-2}}|h_m(y)|dy
	\endaligned
\end{equation}
and
\begin{equation}\label{c002}
	\aligned
	&	|\varphi_m(x)|\\&\leq
	C\left| \int_{\mathbb{R}^N}\frac{K_{2}(y)}{|y-x|^{N-2}}\Big(|y|^{-(N-\alpha)}\ast K_{2}(y) (Z_{\overline{r},\overline{x}'',\lambda}^{2^{*}_{\alpha}-1}
	\phi_m)\Big)Z_{\overline{r},\overline{x}'',\lambda}^{2^{*}_{\alpha}-1}(y)dy\right| \\
	&+C\left| \int_{\mathbb{R}^N}\frac{K_{2}(y)}{|y-x|^{N-2}}\Big(|y|^{-(N-\alpha)}\ast K_{2}(y)|Z_{\overline{r},\overline{x}'',\lambda}|^{^{2^{*}_{\alpha}}}\Big)Z_{\overline{r},\overline{x}'',\lambda}^{2^{*}_{\alpha}-2}\phi_m(y)dy\right|
	\\ &+C\sum_{l=1}^{N}|c_{l}|\Bigg[\Big|\sum_{j=1}^{m}\int_{\mathbb{R}^N}\frac{1}{|y-x|^{N-2}}\Big(|y|^{-(N-\alpha)}\ast (Z_{z_j,\lambda}^{2^{*}_{\alpha}-1} Z_{j,l})\Big)Z_{z_j,\lambda}^{2^{*}_{\alpha}-1}(y)dy\Big|\\
	&+\Big|\sum_{j=1}^{m}\int_{\mathbb{R}^N}\frac{1}{|y-x|^{N-2}}\Big(|y|^{-(N-\alpha)}\ast |Z_{z_j,\lambda}|^{2^{*}_{\alpha}}\Big)Z_{z_j,\lambda}^{2^{*}_{\alpha}-2}Z_{j,l}(y)dy\Big|\Bigg]+C\int_{\mathbb{R}^N}\frac{1}{|y-x|^{N-2}}|g_m(y)|dy.
	\endaligned
\end{equation}

Let's estimate the right terms of \eqref{c2}. We know that $|y-z_{j}|\geq|y-z_{1}|, \ \forall x \in\Omega_{1}$, where
$$
\Omega_{j}=\left\{x=(x',x'')\in\mathbb{R}^{2}\times\mathbb{R}^{N-2}:\Big\langle\frac{x'}{|x'|},
\frac{z_{j}'}{|z_{j}'|}\Big\rangle\geq\cos\frac{\pi}{m}\right\},\ j=1,\cdot\cdot\cdot,m.
$$
According to \cite{CYZ}, we have
\begin{equation}\label{t0}
	\aligned
	&\left| |y|^{-(N-\alpha)}\ast K_{1}(y) (Y_{\overline{r},\overline{x}'',\lambda}^{2^{*}_{\alpha}-1}
	\varphi_m)\right| \\
	&\leq C\|\varphi_m\|_{\ast}\sum_{j=1}^{m}\int_{\Omega_{j}}\frac{1}{|x-y|^{N-\alpha}}\left( \frac{\lambda^{\frac{\alpha+2}{2}}}{(1+\lambda|y- z_{j}|)^{\alpha+2-\tau_{1}\frac{\alpha+2}{N-2}}}\sum_{i=1}^{m}\frac{\lambda^{\frac{N-2}{2}}}{(1+\lambda|y-z_{i}|)^{\frac{N-2}{2}+\tau}}\right) dy\\
	&\leq C\|\varphi_m\|_{\ast}\sum_{j=1}^{m}\int_{\Omega_{j}}\frac{1}{|x-y|^{N-\alpha}} \frac{\lambda^{\frac{N+\alpha}{2}}}{(1+\lambda |y-z_{j}|)^{\frac{N+2\alpha+2}{2}-\tau_{1}\frac{N+\alpha}{N-2}+\tau}}dy\\
	&\leq C\|\varphi_m\|_{\ast}\sum_{j=1}^{m}{\frac{\lambda^{\frac{N-\alpha}{2}}}{(1+\lambda |y-z_{j}|)^{\min\lbrace \frac{N+2}{2}-\tau_{1}\frac{N+\alpha}{N-2}+\tau,N-\alpha\rbrace}}dy}.
	\endaligned
\end{equation}
Therefore, the estimate of the first term in the right side of \eqref{c2} can be obtained. If $$\min\left\lbrace \frac{N+2}{2}-\tau_{1}\frac{N+\alpha}{N-2}+\tau,N-\alpha\right\rbrace=\frac{N+2}{2}-\tau_{1}\frac{N+\alpha}{N-2}+\tau,$$we have
\begin{equation}\label{t1}
	\aligned
	&\int_{\mathbb{R}^N}\frac{K_{1}(y)}{|y-x|^{N-2}}\Big(|y|^{-(N-\alpha)}\ast K_{1}(y) (Y_{\overline{r},\overline{x}'',\lambda}^{2^{*}_{\alpha}-1}
	|\varphi_m|)\Big)Y_{\overline{r},\overline{x}'',\lambda}^{2^{*}_{\alpha}-1}(y)dy\\
	%&\leq C\|\varphi_m\|_{\ast}\int_{\mathbb{R}^N}\frac{1}{|y-x|^{N-2}}
	%\sum_{j=1}^{m}\frac{\lambda^{\frac{N-\alpha}{2}}}{(1+\lambda|y-z_{j}|)^{{\frac{N+2}{2}-\tau_{1}\frac{N+\alpha}{N-2}+\tau}}}Y_{\overline{r},\overline{x}'',\lambda}^{2^{*}_{\alpha}-1}(y)\\
&\leq C\|\varphi_m\|_{\ast}\lambda^{\frac{N-2}{2}}\sum_{j=1}^{m}\int_{\Omega_{j}}\frac{1}{|y-\lambda x|^{N-2}}
	\sum_{j=1}^{m}\frac{1}{(1+|y-\lambda z_{j}|)^{\frac{N+2}{2}-\tau_{1}\frac{N+\alpha}{N-2}+\tau}}\frac{1}{(1+|y- \lambda z_{j}|)^{\alpha+2-\tau_{1}\frac{\alpha+2}{N-2}}}dy\\
	&\leq C\|\varphi_m\|_{\ast}\lambda^{\frac{N-2}{2}}\sum_{j=1}^{m}\int_{\Omega_{j}}\frac{1}{|y-\lambda x|^{N-2}}
	\frac{1}{(1+|y-\lambda z_{j}|)^{\frac{N+6}{2}+\alpha-\tau_{1}\frac{2N+2\alpha}{N-2}+\tau}}dy\\
	&\leq C\|\varphi_m\|_{\ast}\lambda^{\frac{N-2}{2}}\sum_{j=1}^{m}
	\frac{1}{(1+|x-\lambda z_{j}|)^{\frac{N-2}{2}+\tau+\theta}},
	\endaligned
\end{equation}
where $\theta=N+2+\alpha-\frac{2N+2\alpha}{N-2}\tau_{1}>0$.
If $\min\left\lbrace \frac{N+2}{2}-\tau_{1}\frac{N+\alpha}{N-2}+\tau,\alpha\right\rbrace=N-\alpha$, we can obtain the same result,
%\begin{equation}\nonumber
%\int_{\mathbb{R}^N}\frac{K_{1}(y)}{|y-x|^{N-2}}\Big(|y|^{-(N-\alpha)}\ast K_{1}(y)|Y_{\overline{r},\overline{x}'',\lambda}|^{^{2^{*}_{\alpha}}}\Big)Y_{\overline{r},\overline{x}'',\lambda}^{2^{*}_{\alpha}-2}|\varphi_m(y)|dy\leq C\|\varphi_m\|_{\ast}\lambda^{\frac{N-2}{2}}\sum_{j=1}^{m}
%\frac{1}{(1+|x-\lambda z_{j}|)^{\frac{N-2}{2}+\tau+\theta}},
%\end{equation}
where $\theta=\frac{N+2}{2}-\frac{N+\alpha}{N-2}\tau_{1}-\tau>0$.

Similarly, we can repeat the arguments for the second term and have
\begin{equation}\label{t2}
	C\int_{\mathbb{R}^N}\frac{K_{1}(y)}{|y-x|^{N-2}}\Big(|y|^{-(N-\alpha)}\ast K_{1}(y)|Y_{\overline{r},\overline{x}'',\lambda}|^{^{2^{*}_{\alpha}}}\Big)Y_{\overline{r},\overline{x}'',\lambda}^{2^{*}_{\alpha}-2}|\varphi_m(y)| dy\leq C\|\varphi_m\|_{\ast}\lambda^{\frac{N-2}{2}}\sum_{j=1}^{m}
	\frac{1}{(1+|x-\lambda z_{j}|)^{\frac{N-2}{2}+\tau+\theta}}.
\end{equation}
Applying Lemma \ref{B3} and Lemma \ref{B4}, we can estimate the third term as follow. If $l=1$, then we have
\begin{equation}\nonumber
	\aligned
	\Big|\sum_{j=1}^{m}\int_{\mathbb{R}^N}\frac{1}{|y-x|^{N-2}}\Big(|y|^{-(N-\alpha)}\ast (Y_{z_j,\lambda}^{2^{*}_{\alpha}-1} Y_{j,l})\Big)Y_{z_j,\lambda}^{2^{*}_{\alpha}-1}(y)dy\Big|
	%&\leq C\lambda^{-1}\sum_{j=1}^{m}\int_{\mathbb{R}^N}\frac{1}{|y-x|^{N-2}}\Big(|y|^{-(N-\alpha)}\ast U_{z_j,\lambda}^{2^{*}_{\alpha}} \Big)U_{z_j,\lambda}^{2^{*}_{\alpha}-1}(y)dy\\
	%&\leq C\lambda^{-1}\sum_{j=1}^{m}\int_{\mathbb{R}^N}\frac{1}{|y-x|^{N-2}}\frac{\lambda^{\frac{N-\alpha}{2}}}{(1+\lambda|y-z_{j}|)^{N-\alpha}}\frac{\lambda^{\frac{\alpha+2}{2}}}{(1+\lambda|y-z_{j}|)^{\alpha+2}}dy\\
	&\leq C\lambda^{-1}\sum_{j=1}^{m}\int_{\mathbb{R}^N}\frac{1}{|y-x|^{N-2}}\frac{\lambda^{\frac{N+2}{2}}}{(1+\lambda|y-z_{j}|)^{N+2}}\\
	&\leq C\lambda^{\frac{N-2}{2}-1}\sum_{j=1}^{m}\frac{1}{(1+\lambda|x-z_{j}|)^{\frac{N-2}{2}+\tau}},
\endaligned
\end{equation}
where $0<\tau<N-2$. If $l\neq 1$, we can obtain
\begin{equation}\nonumber
	\Big|\sum_{j=1}^{m}\int_{\mathbb{R}^N}\frac{1}{|y-x|^{N-2}}\Big(|y|^{-(N-\alpha)}\ast (Y_{z_j,\lambda}^{2^{*}_{\alpha}-1} Y_{j,l})\Big)Y_{z_j,\lambda}^{2^{*}_{\alpha}-1}(y)dy\Big|\leq C\lambda^{\frac{N-2}{2}+1}\sum_{j=1}^{m}\frac{1}{(1+\lambda|x-z_{j}|)^{\frac{N-2}{2}+\tau}}.
\end{equation}
The same estimate can be obtined for
\begin{equation}\nonumber
	\Big|\sum_{j=1}^{m}\int_{\mathbb{R}^N}\frac{1}{|y-x|^{N-2}}\Big(|y|^{-(N-\alpha)}\ast |Y_{z_j,\lambda}|^{2^{*}_{\alpha}} \Big)Y_{z_j,\lambda}^{2^{*}_{\alpha}-2}Y_{j,l}(y)dy\Big|.
\end{equation}
For the last term, by Lemma \ref{B3}, we get
\begin{equation}\label{t4}
	\int_{\mathbb{R}^N}\frac{1}{|y-x|^{N-2}}|h_m(y)|dy\leq C\|h_m\|_{**}\lambda^{\frac{N-2}{2}}\sum_{j=1}^{m}\frac{1}{(1+\lambda|x-z_{j}|)^{\frac{N-2}{2}+\tau}}.
\end{equation}

Next, we estimate $c_l, l=1,2,\cdots, N.$
First, multiplying \eqref{C1} by $(Y_{1,t},Z_{1,t})$($t=1,2,\cdots, N$) and integrating, we have
\begin{equation}\label{I1}
	\aligned
	\sum_{l=1}^{N}\sum_{j=1}^{m}&\Big<  \Big((2^{*}_{\alpha}-1)K_{1}(x)\Big(|x|^{-(N-\alpha)}\ast K_{1}(x) |Y_{z_j,\lambda}|^{2^{*}_{\alpha}}\Big)Y_{z_j,\lambda}^{2^{*}_{\alpha}-2}Y_{j,l}+2^{*}_{\alpha}K_{1}(x)\Big(|x|^{-(N-\alpha)}\ast K_{1}(x) Y_{z_j,\lambda}^{2^{*}_{\alpha}-1}Y_{j,l}\Big)Y_{z_j,\lambda}^{2^{*}_{\alpha}-1},\\
	&(2^{*}_{\alpha}-1)K_{2}(x)\Big(|x|^{-(N-\alpha)}\ast K_{2}(x) |Z_{z_j,\lambda}|^{2^{*}_{\alpha}}\Big)Z_{z_j,\lambda}^{2^{*}_{\alpha}-2}Z_{j,l}+2^{*}_{\alpha}K_{2}(x)\Big(|x|^{-(N-\alpha)}\ast K_{2}(x) Z_{z_j,\lambda}^{2^{*}_{\alpha}-1}Z_{j,l}\Big)Z_{z_j,\lambda}^{2^{*}_{\alpha}-1}
	\Big),\\
	&(Y_{1,t},Z_{1,t})\Big> c_{l}=
	\left\langle \mathcal{L}_m(\phi,\varphi), (Y_{1,t},Z_{1,t})\right\rangle- \left\langle (h_m,g_m), (Y_{1,t},Z_{1,t})\right\rangle.
	\endaligned
\end{equation}
By Lemma \ref{B2} , we get
\begin{equation}\label{c4}
		\aligned
	&|\langle (h_m,g_m), (Z_{1,t},Y_{1,t})\rangle|
		\\  &\leq C\lambda^{n_{t}}\|(h_m,g_m)\|_{\ast\ast} \int_{\R^N}\frac{\lambda^{\frac{N-2}{2}}}{(1+\lambda|x-z_1|)^{(N-2)}}\sum_{j=1}^{m}\frac{\lambda^{\frac{N+2}{2}}}{(1+\lambda|x-z_j|)^{\frac{N+2}{2}+\tau}}dx
&\leq C\lambda^{n_{t}}\|(h_m,g_m)\|_{\ast\ast},
\endaligned
\end{equation}
where $n_{t}=-1$, $n_{t}=1,l=2,\cdot\cdot\cdot,N$.

Following the calculation for Lemma 2.1 in \cite{CYZ}, we obtain
\begin{equation}
	\aligned
	&\left| K_{1}(x)\Big(|x|^{-(N-\alpha)}\ast K_{1}(x)(Y_{\overline{r},\overline{x}'',\lambda}^{2^{*}_{\alpha}-1}
	\varphi_{m})\Big)\right| \\
	&\leq C\|\varphi_{m}\|_{*}	\int_{\mathbb{R}^N}\frac{1}{|y|^{N-\alpha}}\sum_{j=1}^{m}\frac{\xi(x-y)\lambda^{\frac{\alpha+2}{2}}}{(1+\lambda|x-y- z_{i}|)^{\alpha+2-\tau_{1}\frac{\alpha+2}{N-2}}}\sum_{j=1}^{m}\frac{\lambda^{\frac{N-2}{2}}}{(1+\lambda|x-y-z_{j}|)^{\frac{N-2}{2}+\tau}}dy\\
	&\leq C\|\varphi_{m}\|_{*}\sum_{i=1}^{m}\int_{\mathbb{R}^N}\frac{1}{|y|^{N-\alpha}}\frac{\xi(x-y)\lambda^{\frac{N+\alpha}{2}}}{(1+\lambda|x-y-z_{j}|)^{\frac{N+2\alpha+2}{2}-\tau_{1}\frac{\alpha+2}{N-2}+\tau}}dy\\
	&+C\|\varphi_{m}\|_{*}\sum_{i\neq j}^{m}\int_{\mathbb{R}^N}\frac{1}{|y|^{N-\alpha}}\frac{\xi(x-y)\lambda^{\frac{\alpha+2}{2}}}{(1+\lambda|x-y- z_{1}|)^{\alpha+2-\tau_{1}\frac{\alpha+2}{N-2}}}\frac{\lambda^{\frac{N-2}{2}}}{(1+\lambda|x-y-z_{j}|)^{\frac{N-2}{2}+\tau}}dy=O(\frac{m^{2}\|\varphi_{m}\|_{*}}{\lambda^{\frac{N-\alpha-2}{2}}}).
	\endaligned
\end{equation}
Moreover, for some small constant $\varepsilon>0$ we have
\begin{equation}
	\aligned
	&\int_{\mathbb{R}^N}K_{1}(x)\Big(|x|^{-(N-\alpha)}\ast K_{1}(x)(Y_{\overline{r},\overline{x}'',\lambda}^{2^{*}_{\alpha}-1}
	\varphi_{m})\Big)Y_{\overline{r},\overline{y}'',\lambda}^{2^{*}_{\alpha}-1}Y_{1,t}dx\\
	&\leq C\|\varphi_{m}\|_{*}\|Y_{\overline{r},\overline{x}'',\lambda\|_{*}}\|\frac{m^{2}}{\lambda^{\frac{N-\alpha-2}{2}}}\int_{\mathbb{R}^N}\sum_{j=1}^{m}\frac{\lambda^{\frac{\alpha+2}{2}}}{(1+\lambda|x- z_{i}|)^{\alpha+2-\tau_{1}\frac{\alpha+2}{N-2}}}\frac{\xi\lambda^{\frac{N-2}{2}+n_{t}}}{(1+\lambda|x-z_{1}|)^{N-2}}dx\\
	&\leq C\|\varphi_{m}\|_{*}\|\frac{m^{2}\lambda^{n_{t}}}{\lambda^{\frac{N-\alpha-2}{2}}}\frac{m}{\lambda^{\frac{N-\alpha}{2}}}\int_{\mathbb{R}^N}\frac{1}{(1+\lambda|x-z_{1}|)^{N+\alpha-\tau_{1}\frac{\alpha+2}{N-2}}}dx=O(\frac{\lambda^{n_{t}}\|\varphi_{m}\|_{*}}{\lambda^{1+\varepsilon}}),\\
	\endaligned
\end{equation}
where $N-\alpha>5-\frac{6}{N-2}$. In the same way, we obtain
$$
\Big\langle\Big(K_{1}(x)|x|^{N-\alpha}\ast \left( K_{1}(x)Y_{\overline{r},\overline{x}'',\lambda}^{2_{\alpha}^{*}}\right) \Big)Y_{\overline{r},\overline{x}'',\lambda}^{2_{\alpha}^{*}-2}\varphi_m, Y_{1,t}\Big\rangle= O\Big(\frac{\lambda^{n_{t}}\|\varphi_m\|_{\ast}}{\lambda^{1+\varepsilon}}\Big),
$$
$$
\Big\langle\Big(K_{2}(x)|x|^{N-\alpha}\ast \left( K_{2}(x)Z_{\overline{r},\overline{x}'',\lambda}^{2_{\alpha}^{*}}\right) \Big)Z_{\overline{r},\overline{x}'',\lambda}^{2_{\alpha}^{*}-2}\phi_m, Z_{1,t}\Big\rangle= O\Big(\frac{\lambda^{n_{t}}\|\phi_m\|_{\ast}}{\lambda^{1+\varepsilon}}\Big),
$$
and
$$
\Big\langle\Big(K_{2}(x)|x|^{N-\alpha}\ast \left( K_{2}(x)Z_{\overline{r},\overline{x}'',\lambda}^{2_{\alpha}^{*}-1}\phi_{m}\right) \Big)Z_{\overline{r},\overline{x}'',\lambda}^{2_{\alpha}^{*}-1}, Z_{1,t}\Big\rangle= O\Big(\frac{\lambda^{n_{t}}\|\phi_m\|_{\ast}}{\lambda^{1+\varepsilon}}\Big).
$$
It is clear that
$$\aligned
\int_{\mathbb{R}^{N}}\Delta Y_{1,t}\phi_mdx
=\int_{\mathbb{R}^{N}}(\xi\Delta (V_{z_{1},\lambda})_{t}+(V_{z_{1},\lambda})_{t}\Delta\xi +2\nabla\xi\nabla (V_{z_{1},\lambda})_{t}) \phi_mdx,
\endaligned$$
where
$$
(V_{z_{1},\lambda})_{1}=\frac{\partial V_{z_{1},\lambda}}{\partial \lambda}, (V_{z_{1},\lambda})_{2}=\frac{\partial V_{z_{1},\lambda}}{\partial \overline{r}},
(V_{z_{1},\lambda})_{k}=\frac{\partial V_{z_{1},\lambda}}{\partial \overline{x}_{k}''},
\ \mbox{for }k=3,\cdot\cdot\cdot,N.
$$
By Lemma \ref{B4}, we get
$$\aligned
&\int_{\mathbb{R}^{N}}\xi\Delta (V_{z_{1},\lambda})_{t} \phi_mdx\\
&\leq C\|\phi_m\|_{\ast}\sum_{j=1}^{m}\int_{\mathbb{R}^N}\int_{\mathbb{R}^N}
\frac{\xi\lambda^{\frac{N+\alpha}{2}}}{(1+\lambda|x-z_{1}|)^{N+\alpha}}
\frac{1}{|x-y|^{N-\alpha}}\frac{\lambda^{\frac{\alpha+2}{2}+n_{t}}}{(1+\lambda|y-z_{1}|)^{\alpha+2}}
\frac{\lambda^\frac{N-2}{2}}{(1+\lambda|y-z_{j}|)^{\frac{N-2}{2}+\tau}}dxdy\\
&\leq C\lambda^{n_{t}}\|\phi_m\|_{\ast}\sum_{j=1}^{m}\frac{1}{(\lambda|z_{1}-z_{j}|)^{\frac{N}{2}}}\int_{\Omega_{j}}
\frac{1}{(1+\lambda|y-z_{j}|)^{N+1}}dy=O(\frac{\lambda^{n_{t}}\|\phi_m\|_{\ast}}{\lambda^{\lambda^{1+\varepsilon}}}).
\endaligned$$
On the other hand, direct calculation gives
$$
\int_{\mathbb{R}^{N}}(V_{z_{1},\lambda})_{t}\Delta\xi \phi_mdx
\leq C\|\phi_m\|_{\ast}
\sum_{j=1}^{m}\int_{\mathbb{R}^N}
\frac{\lambda^{\frac{N-2}{2}}}{(1+\lambda|x-z_{j}|)^{\frac{N-2}{2}+\tau}}
\frac{|\Delta\xi|\lambda^{\frac{N-2}{2}+n_{t}}}{(1+\lambda|x-z_{1}|)^{N-2}}dx
%\leq C\|\phi_m\|_{\ast}
%\frac{m\lambda^{n_{t}}}{\lambda^{2+\tau}}
= O(\frac{\lambda^{n_{t}}\|\phi_m\|_{\ast}}{\lambda^{1+\varepsilon}}),
$$
and
$$
\int_{\mathbb{R}^{N}}\nabla\xi\nabla (V_{z_{1},\lambda})_{t} \phi_mdx
\leq C\|\phi_m\|_{\ast}
\sum_{j=1}^{m}\int_{\mathbb{R}^N}
\frac{|\nabla\xi|\lambda^{\frac{N-2}{2}}}{(1+\lambda|x-z_{j}|)^{\frac{N-2}{2}+\tau}}
\frac{\lambda^{\frac{N}{2}+n_{t}}}{(1+\lambda|x-z_{1}|)^{N-1}}dx
%\leq C\|\phi_m\|_{\ast}
%\frac{m\lambda^{n_{t}}}{\lambda^{2+\tau}}
= O(\frac{\lambda^{n_{t}}\|\phi_m\|_{\ast}}{\lambda^{1+\varepsilon}}).
$$
So
$$
\Big\langle(-\Delta \phi_m,-\Delta \varphi_m), (Y_{1,t},Z_{1,t})\Big\rangle= O\Big(\frac{\lambda^{n_{t}}\|(\phi_m,\varphi_m)\|_{\ast}}{\lambda^{1+\varepsilon}}\Big).
$$
According to the above estimates, we know that
\begin{equation}\label{c7}
	\Big\langle \mathcal{L}_m(\phi_m,\varphi_m), (Y_{1,t},Z_{1,t})\Big\rangle- \Big\langle (h_m,g_m), (Y_{1,t},Z_{1,t})\Big\rangle=O\Big(\frac{\lambda^{n_{t}}\|(\phi_m,\varphi_m)\|_{\ast}}{\lambda^{1+\varepsilon}}+\lambda^{n_{t}}\|(h_m,g_m)\|_{\ast\ast}\Big)
\end{equation}
Furthermore, it is easy to check that
\begin{equation}\nonumber
	\sum_{j=1}^{m}\Big\langle K_{1}(x)\Big(|x|^{N-\alpha}\ast \left( K_{1}(x)|Y_{z_j,\lambda}|^{2_{\alpha}^{*}}\right) \Big)Y_{z_j,\lambda}^{2_{\alpha}^{*}-2}Y_{j,l}, Y_{1,t}\Big\rangle=(\overline{c}_{1}+o(1))\delta_{tl}\lambda^{n_{l}}\lambda^{n_{t}},
\end{equation}
\begin{equation}\nonumber
	\sum_{j=1}^{m}\Big\langle K_{1}(x)\Big(|x|^{N-\alpha}\ast \left( K_{1}(x)|Y_{z_j,\lambda}|^{2_{\alpha}^{*}-1}Y_{j,l}\right) \Big)Y_{z_j,\lambda}^{2_{\alpha}^{*}-1}, Y_{1,t}\Big\rangle=(\overline{c}_{2}+o(1))\lambda^{n_{l}}\lambda^{n_{t}},
\end{equation}
\begin{equation}\nonumber
	\sum_{j=1}^{m}\Big\langle K_{2}(x)\Big(|x|^{N-\alpha}\ast \left( K_{2}(x)|Z_{z_j,\lambda}|^{2_{\alpha}^{*}}\right) \Big)Z_{z_j,\lambda}^{2_{\alpha}^{*}-2}Z_{j,l}, Z_{1,t}\Big\rangle=(\overline{c}_{3}+o(1))\delta_{tl}\lambda^{n_{l}}\lambda^{n_{t}},
\end{equation}
and
\begin{equation}\nonumber
		\sum_{j=1}^{m}\Big\langle K_{2}(x)\Big(|x|^{N-\alpha}\ast \left( K_{2}(x)|Z_{z_j,\lambda}|^{2_{\alpha}^{*}-1}Z_{j,l}\right) \Big)Z_{z_j,\lambda}^{2_{\alpha}^{*}-1}, Z_{1,t}\Big\rangle=(\overline{c}_{4}+o(1))\lambda^{n_{l}}\lambda^{n_{t}},
\end{equation}
for some constant $\overline{c}_{1}>0$, $\overline{c}_{2}>0$, $\overline{c}_{3}>0$ and $\overline{c}_{4}> 0$. By substituting these and \eqref{c7} into \eqref{I1}, we obtain
\begin{equation}\label{c9}
	c_{l}=\frac{1}{\lambda^{n_{l}}}(o(\|(\phi_m,\varphi_m)\|_{\ast})+O(\|(h_m,g_m)\|_{\ast\ast})), \ l = 1, 2, \cdot\cdot\cdot,N.
\end{equation}
Therefore, we can conclude
\begin{equation}\label{c10}
	\|(\phi_m,\varphi_m)\|_{\ast}\leq o(1)+\|(h_m,g_m)\|_{\ast\ast}+\frac{\sum_{j=1}^{N}\frac{1}{(1+\lambda|x-z_{j}|)^{\frac{N-2}{2}+\tau+\theta}}}
	{\sum_{j=1}^{N}\frac{1}{(1+\lambda|x-z_{j}|)^{\frac{N-2}{2}+\tau}}}.
\end{equation}
Since $\|(\phi_m,\varphi_m)\|_{\ast}= 1$ and \eqref{c10}, we have that there is $R > 0$ such that
\begin{equation}\label{c11}
	\|\lambda^{-\frac{N-2}{2}}(\phi_m,\varphi_m)\|_{L^{\infty}(B_{\frac{R}{\lambda}}(z_{j}))}\geq a>0
\end{equation}
for some $j$. Let $\Big(\overline{\phi}_m (x)=\lambda^{-\frac{N-2}{2}}\phi_m(\lambda(x-z_{j})),\overline{\varphi}_m (x)=\lambda^{-\frac{N-2}{2}}\varphi_m(\lambda(x-z_{j}))\Big)$, then
$$
\int_{\R^N}(|\nabla \overline{\phi}_{m}|^{2}+|\nabla \overline{\varphi}_{m}|^{2})dx\leq C.
$$
We deduce that there are $(u,v)\in D^{1,2}(\mathbb{R}^{N})\times D^{1,2}(\mathbb{R}^{N})$, so
$$
(\overline{\phi}_m,\overline{\varphi}_m)\rightarrow (u,v),\ \ \mbox{weakly in }D^{1,2}(\mathbb{R}^{N})
$$
and
$$
(\overline{\phi}_m,\overline{\varphi}_m)\rightarrow (u,v),\ \ \mbox{strongly in }L_{loc}^{2}(\mathbb{R}^{N}),
$$
as $m\rightarrow+\infty$. It is a simple matter that $(u,v)\in D^{1,2}(\mathbb{R}^{N})\times D^{1,2}(\mathbb{R}^{N})$ satisfies
\begin{equation}\label{ib5}
	\left\{\begin{array}{ll}
		-\Delta u&=(2_{\alpha}^{*}-1)\big(|x|^{-(N-\alpha)}\ast |V_{0,\Lambda}|^{2_{\alpha}^{*}}\big)V_{0,\Lambda}^{2_{\alpha}^{*}-2}v+2_{\alpha}^{*}\big(|x|^{-(N-\alpha)}\ast(V_{0,\Lambda}^{(2_{\alpha}^{*}-1)}v)\big)V_{0,\Lambda}^{2_{\alpha}^{*}-1}
		\sp\mbox{in}\sp \R^N,\\[1mm]
		-\Delta v&=(2_{\alpha}^{*}-1)\big(|x|^{-(N-\alpha)}\ast |U_{0,\Lambda}|^{2_{\alpha}^{*}}\big)U_{0,\Lambda}^{2_{\alpha}^{*}-2}u+2_{\alpha}^{*}\big(|x|^{-(N-\alpha)}\ast(U_{0,\Lambda}^{(2_{\alpha}^{*}-1)}u)\big)U_{0,\Lambda}^{2_{\alpha}^{*}-1}
		\sp\mbox{in}\sp \R^N	\end{array}\right.
\end{equation}
for some $\Lambda\in[\Lambda_{1}, \Lambda_{2}]$. Since $(u,v)$ is perpendicular to the kernel of \eqref{ib5}, we can conclude that $(u,v) = (0,0)$ by the non-degeneracy of $(U_{0,1},V_{0,1})$, From \eqref{c11} we derive the contradiction.
\end{proof}

Let
$$\aligned
E=\Big\{(\phi,\varphi)\in H_{s}\times H_{s}, \  \sum_{j=1}^{m}\Big< (2^{*}_{\alpha}-1)\Big(|x|^{-(N-\alpha)}\ast |Y_{z_j,\lambda}|^{2^{*}_{\alpha}}\Big)Y_{z_j,\lambda}^{2^{*}_{\alpha}-1}Y_{j,l}+2^{*}_{\alpha}\Big(|x|^{-(N-\alpha)}\ast (Y_{z_j,\lambda}^{2^{*}_{\alpha}-1}
Y_{j,l})\Big)Y_{z_j,\lambda}^{2^{*}_{\alpha}-1},\\
\displaystyle \hspace{14.14mm} (2^{*}_{\alpha}-1)\Big(|x|^{-(N-\alpha)}\ast |Z_{z_j,\lambda}|^{2^{*}_{\alpha}}\Big)Z_{z_j,\lambda}^{2^{*}_{\alpha}-1}Z_{j,l}+2^{*}_{\alpha}\Big(|x|^{-(N-\alpha)}\ast (Z_{z_j,\lambda}^{2^{*}_{\alpha}-1}
Z_{j,l})\Big)Z_{z_j,\lambda}^{2^{*}_{\alpha}-1}, (\phi,\varphi)\Big> =0,\ l=1,2,\dots,N.
\Big\}
\endaligned$$
endowed with the usual inner product $[(\phi,\varphi),(u,v)]=\int_{\mathbb{R}^{N}}\nabla u\nabla\phi +\nabla v\nabla\varphi dx$. Hence, the resolution of Problem \eqref{c1} is equivalent to the determination of an element $(\phi,\varphi)\in E$ satisfying
$$
\aligned
\left[(\phi,\varphi),(u,v)\right]  =&(2^{*}_{\alpha}-1)\int_{\R^N}K_{1}(x)\Big(|x|^{-(N-\alpha)}\ast K_{1}(x)|Y_{\overline{r},\overline{x}'',\lambda}|^{2^{*}_{\alpha}}\Big)Y_{\overline{r},\overline{x}'',\lambda}^{2^{*}_{\alpha}-2}\varphi udx\\
&+2^{*}_{\alpha}\int_{\R^N}K_{1}(x)\Big(|x|^{-(N-\alpha)}\ast K_{1}(x)(Y_{\overline{r},\overline{x}'',\lambda}^{2^{*}_{\alpha}-1}
\varphi u)\Big)Y_{\overline{r},\overline{x}'',\lambda}^{2^{*}_{\alpha}-1} dx\\
&+(2^{*}_{\alpha}-1)\int_{\R^N}K_{2}(x)\Big(|x|^{-(N-\alpha)}\ast K_{2}(x)|Z_{\overline{r},\overline{x}'',\lambda}|^{2^{*}_{\alpha}}\Big)Z_{\overline{r},\overline{x}'',\lambda}^{2^{*}_{\alpha}-2}\phi v dx\\
&+2^{*}_{\alpha}\int_{\R^N}K_{2}(x)\Big(|x|^{-(N-\alpha)}\ast K_{2}(x)(Z_{\overline{r},\overline{x}'',\lambda}^{2^{*}_{\alpha}-1}
\phi v)\Big)Z_{\overline{r},\overline{x}'',\lambda}^{2^{*}_{\alpha}-1} dx+\int_{\R^N}(hu+
gv)dx, \ \ \forall(u,v)\in E.
\endaligned
$$
By invoking the same line of reasoning employed in the proof of Proposition 4.1 in \cite{DFM}—namely, an appeal to the Riesz representation theorem together with Fredholm’s alternative theorem—one establishes that, for every pair $(h,g)$, a unique solution exists whenever the following equation
\begin{equation}\label{0}
	\left\{\begin{array}{ll}
		\mathcal{L}_m(\phi,\varphi)=\sum_{l=1}^{N}c_{l}\sum_{j=1}^{m}\Big((2^{*}_{\alpha}-1)\Big(|x|^{-(N-\alpha)}\ast  |Y_{z_j,\lambda}|^{2^{*}_{\alpha}}\Big)Y_{z_j,\lambda}^{2^{*}_{\alpha}-2}Y_{j,l}+2^{*}_{\alpha}\Big(|x|^{-(N-\alpha)}\ast  Y_{z_j,\lambda}^{2^{*}_{\alpha}-1}Y_{j,l}\Big)Y_{z_j,\lambda}^{2^{*}_{\alpha}-1},\\
		\displaystyle \hspace{14.14mm}(2^{*}_{\alpha}-1)\Big(|x|^{-(N-\alpha)}\ast |Z_{z_j,\lambda}|^{2^{*}_{\alpha}}\Big)Z_{z_j,\lambda}^{2^{*}_{\alpha}-2}Z_{j,l}+2^{*}_{\alpha}\Big(|x|^{-(N-\alpha)}\ast  Z_{z_j,\lambda}^{2^{*}_{\alpha}-1}Z_{j,l}\Big)Z_{z_j,\lambda}^{2^{*}_{\alpha}-1}
		\Big),\\	
		(\phi,\varphi)\in H_{s}\times H_{s},\\
		\sum_{j=1}^{m}\Big< (2^{*}_{\alpha}-1)\Big(|x|^{-(N-\alpha)}\ast |Y_{z_j,\lambda}|^{2^{*}_{\alpha}}\Big)Y_{z_j,\lambda}^{2^{*}_{\alpha}-1}Y_{j,l}+2^{*}_{\alpha}\Big(|x|^{-(N-\alpha)}\ast (Y_{z_j,\lambda}^{2^{*}_{\alpha}-1}
		Y_{j,l})\Big)Y_{z_j,\lambda}^{2^{*}_{\alpha}-1},\\
		\displaystyle \hspace{14.14mm} (2^{*}_{\alpha}-1)\Big(|x|^{-(N-\alpha)}\ast |Z_{z_j,\lambda}|^{2^{*}_{\alpha}}\Big)Z_{z_j,\lambda}^{2^{*}_{\alpha}-1}Z_{j,l}+2^{*}_{\alpha}\Big(|x|^{-(N-\alpha)}\ast (Z_{z_j,\lambda}^{2^{*}_{\alpha}-1}
		Z_{j,l})\Big)Z_{z_j,\lambda}^{2^{*}_{\alpha}-1}, (\phi,\varphi)\Big> =0,\\
		\displaystyle \hspace{14.14mm} l=1,2,\dots,N.
	\end{array}\right.
\end{equation}
for certain constants $c_{l}$, has only trivial solution in $E$.
In fact, assume that it has a nontrivial solution $(\phi,\varphi)=(\phi_{m},\varphi_{m})$, which with
no loss of generality may be taken so that $\|(\phi_{m},\varphi_{m})\|_{\ast}=1$, Lemma \ref{C1} then forces $\|(\phi_{m},\varphi_{m})\|_{\ast}\rightarrow0$. This is certainly a contradiction
that proves that this equation only has the trivial solution in $E$.
Moreover, for every pair $(h,g)$, problem \eqref{c1} possesses a unique solution $(\phi,\varphi)$ satisfying the uniform estimate
\[
\|(\phi,\varphi)\|_{\ast}\leq C \|(h,g)\|_{\ast\ast}.
\]
Consequently, we record the following lemma.

\begin{lem}\label{C2}
	There exist $m_0 > 0$ and a constant $C > 0$, independent of $m$, such that for all
	$m \geq m_0$ and all $h,g\in L^{\infty}(\R^6)$, problem \eqref{c1} has a unique solution $(\phi,\varphi) \equiv L_m(h,g)$. Moreover,
	\begin{equation}\label{c13}
		\|L_m(h,g)\|_{\ast}\leq C\|(h,g)\|_{\ast\ast},\quad |c_l|\leq \frac{C}{\lambda^{n_{l}}}\|(h,g)\|_{\ast\ast}.
	\end{equation}
\end{lem}

Next, we consider the system
\begin{equation}\label{c14}
	\left\{\begin{array}{l}
		\displaystyle -\Delta (Z_{\overline{r},\overline{x}'',\lambda}+\phi)
		-K_{1}(r,x^{''})\Big(|x|^{-(N-\alpha)}\ast K_{1}(r,x^{''})|(Y_{\overline{r},\overline{x}'',\lambda}+\varphi)|^{2_{\alpha}^{*}}\Big)(Y_{\overline{r},\overline{x}'',\lambda}+\varphi)^{2_{\alpha}^{*}-1}
		\\
		\displaystyle \hspace{10.14mm}=
		\sum_{l=1}^{N}c_{l}\sum_{j=1}^{m}\Big[(2_{\alpha}^{*}-1)\Big(|x|^{-(N-\alpha)}\ast |Y_{z_j,\lambda}|^{2_{\alpha}^{*}}\Big)Y_{z_j,\lambda}^{2_{\alpha}^{*}-2}Y_{j,l}+2_{\alpha}^{*}\Big(|x|^{-(N-\alpha)}\ast (Y_{z_j,\lambda}^{2_{\alpha}^{*}-1}
		Y_{j,l})\Big)Y_{z_j,\lambda}^{2_{\alpha}^{*}-1}\Big]
		\hspace{4.14mm}\mbox{in}\hspace{1.14mm} \mathbb{R}^N,\\

		\displaystyle -\Delta (Y_{\overline{r},\overline{x}'',\lambda}+\varphi)
		-K_{2}(r,x^{''})\Big(|x|^{-(N-\alpha)}\ast K_{2}(r,x^{''})|(Z_{\overline{r},\overline{x}'',\lambda}+\phi)|^{2_{\alpha}^{*}}\Big)(Z_{\overline{r},\overline{x}'',\lambda}+\phi)^{2_{\alpha}^{*}-1}
		\\
		\displaystyle \hspace{10.14mm}=
		\sum_{l=1}^{N}c_{l}\sum_{j=1}^{m}\Big[(2_{\alpha}^{*}-1)\Big(|x|^{-(N-\alpha)}\ast |Z_{z_j,\lambda}|^{2_{\alpha}^{*}}\Big)Z_{z_j,\lambda}^{2_{\alpha}^{*}-2}Z_{j,l}+2_{\alpha}^{*}\Big(|x|^{-(N-\alpha)}\ast (Z_{z_j,\lambda}^{2_{\alpha}^{*}-1}
		Z_{j,l})\Big)Z_{z_j,\lambda}^{2_{\alpha}^{*}-1}\Big]
		\hspace{4.14mm}\mbox{in}\hspace{1.14mm} \mathbb{R}^N,\\
		\displaystyle \phi,\varphi\in H_{s}, \ \ \sum_{j=1}^{m}\Big\langle \Big((2_{\alpha}^{*}-1)\Big(|x|^{-(N-\alpha)}\ast |Y_{z_j,\lambda}|^{2_{\alpha}^{*}}\Big)Y_{z_j,\lambda}^{2_{\alpha}^{*}-2}Y_{j,l}+2_{\alpha}^{*}\Big(|x|^{-(N-\alpha)}\ast (Y_{z_j,\lambda}^{2_{\alpha}^{*}-1}
		Y_{j,l})\Big)Y_{z_j,\lambda}^{2_{\alpha}^{*}-1}\\
		\displaystyle	\hspace{20.14mm}(2_{\alpha}^{*}-1)\Big(|x|^{-(N-\alpha)}\ast |Z_{z_j,\lambda}|^{2_{\alpha}^{*}}\Big)Z_{z_j,\lambda}^{2_{\alpha}^{*}-2}Z_{j,l}+2_{\alpha}^{*}\Big(|x|^{-(N-\alpha)}\ast (Z_{z_j,\lambda}^{2_{\alpha}^{*}-1}
		Z_{j,l})\Big)Z_{z_j,\lambda}^{2_{\alpha}^{*}-1} \Big) , (\phi,\varphi) \Big\rangle=0,\\ \hspace{10.14mm}l=1,2,\dots,N.
	\end{array}
	\right.
\end{equation}

Rewrite \eqref{c14} as
\begin{equation}\label{c16}
	\left\{\begin{array}{l}
		\displaystyle -\Delta \phi
		-K_{1}(r,x^{''})(2_{\alpha}^{*}-1)\Big(|x|^{-(N-\alpha)}\ast K_{1}(r,x^{''})|Y_{\overline{r},\overline{x}'',\lambda}|^{2_{\alpha}^{*}}\Big)Y_{\overline{r},\overline{x}'',\lambda}^{2_{\alpha}^{*}-2}\varphi\\
			\displaystyle \hspace{10.14mm}
		-K_{1}(r,x^{''})2_{\alpha}^{*}\Big(|x|^{-(N-\alpha)}\ast K_{1}(r,x^{''}) (Y_{\overline{r},\overline{x}'',\lambda}^{2_{\alpha}^{*}-1}
		\varphi)\Big)Y_{\overline{r},\overline{x}'',\lambda}^{2_{\alpha}^{*}-1}=
		N_{1}(\varphi)+l_{m1}\\
		\displaystyle \hspace{10.14mm}+\sum_{l=1}^{N}c_{l}\sum_{j=1}^{m}\Big[(2_{\alpha}^{*}-1)\Big(|x|^{-(N-\alpha)}\ast |Y_{z_j,\lambda}|^{2_{\alpha}^{*}}\Big)Y_{z_j,\lambda}^{2_{\alpha}^{*}-2}Y_{j,l}+2_{\alpha}^{*}\Big(|x|^{-(N-\alpha)}\ast (Y_{z_j,\lambda}^{2_{\alpha}^{*}-1}
		Y_{j,l})\Big)Y_{z_j,\lambda}^{2_{\alpha}^{*}-1}\Big]
		\hspace{4.14mm}\mbox{in}\hspace{1.14mm} \mathbb{R}^N,\\
		\displaystyle -\Delta \varphi
		-K_{2}(r,x^{''})(2_{\alpha}^{*}-1)\Big(|x|^{-(N-\alpha)}\ast K_{2}(r,x^{''})|Z_{\overline{r},\overline{x}'',\lambda}|^{2_{\alpha}^{*}}\Big)Z_{\overline{r},\overline{x}'',\lambda}^{2_{\alpha}^{*}-2}\phi\\
		\displaystyle \hspace{10.14mm}
		-K_{2}(r,x^{''})2_{\alpha}^{*}\Big(|x|^{-(N-\alpha)}\ast K_{2}(r,x^{''}) (Z_{\overline{r},\overline{x}'',\lambda}^{2_{\alpha}^{*}-1}
		\phi)\Big)Z_{\overline{r},\overline{x}'',\lambda}^{2_{\alpha}^{*}-1}=
		N_{2}(\phi)+l_{m2}\\
		\displaystyle \hspace{10.14mm}+\sum_{l=1}^{N}c_{l}\sum_{j=1}^{m}\Big[(2_{\alpha}^{*}-1)\Big(|x|^{-(N-\alpha)}\ast |Z_{z_j,\lambda}|^{2_{\alpha}^{*}}\Big)Z_{z_j,\lambda}^{2_{\alpha}^{*}-2}Z_{j,l}+2_{\alpha}^{*}\Big(|x|^{-(N-\alpha)}\ast (Z_{z_j,\lambda}^{2_{\alpha}^{*}-1}
		Z_{j,l})\Big)Z_{z_j,\lambda}^{2_{\alpha}^{*}-1}\Big]
		\hspace{4.14mm}\mbox{in}\hspace{1.14mm} \mathbb{R}^N,\\
		\displaystyle \phi,\varphi\in H_{s}, \ \ \sum_{j=1}^{m}\Big\langle \Big((2_{\alpha}^{*}-1)\Big(|x|^{-(N-\alpha)}\ast |Y_{z_j,\lambda}|^{2_{\alpha}^{*}}\Big)Y_{z_j,\lambda}^{2_{\alpha}^{*}-2}Y_{j,l}+2_{\alpha}^{*}\Big(|x|^{-(N-\alpha)}\ast (Y_{z_j,\lambda}^{2_{\alpha}^{*}-1}
		Y_{j,l})\Big)Y_{z_j,\lambda}^{2_{\alpha}^{*}-1}\\
		\displaystyle	\hspace{20.14mm}(2_{\alpha}^{*}-1)\Big(|x|^{-(N-\alpha)}\ast |Z_{z_j,\lambda}|^{2_{\alpha}^{*}}\Big)Z_{z_j,\lambda}^{2_{\alpha}^{*}-2}Z_{j,l}+2_{\alpha}^{*}\Big(|x|^{-(N-\alpha)}\ast (Z_{z_j,\lambda}^{2_{\alpha}^{*}-1}
		Z_{j,l})\Big)Z_{z_j,\lambda}^{2_{\alpha}^{*}-1} \Big) , (\phi,\varphi)\Big\rangle=0,\\ \hspace{10.14mm}l=1,2,\dots,N,
	\end{array}
	\right.
\end{equation}
where
$$\aligned
N_{1}(\varphi)&=K_{1}(r,x^{''})\Big(|x|^{-(N-\alpha)}\ast K_{1}(r,x^{''})|(Y_{\overline{r},\overline{x}'',\lambda}+\varphi)|^{2_{\alpha}^{*}}\Big)(Y_{\overline{r},\overline{x}'',\lambda}+\varphi)^{2_{\alpha}^{*}-1}\\
&-K_{1}(r,x^{''})\Big(|x|^{-(N-\alpha)}\ast K_{1}(r,x^{''})|Y_{\overline{r},\overline{x}'',\lambda}|^{2_{\alpha}^{*}}\Big)Y_{\overline{r},\overline{x}'',\lambda}^{2_{\alpha}^{*}-1}\\
&-K_{1}(r,x^{''})(2_{\alpha}^{*}-1)\Big(|x|^{-(N-\alpha)}\ast K_{1}(r,x^{''})|Y_{\overline{r},\overline{x}'',\lambda}|^{2_{\alpha}^{*}}\Big)Y_{\overline{r},\overline{x}'',\lambda}^{2_{\alpha}^{*}-2}\varphi\\
&-K_{1}(r,x^{''})2_{\alpha}^{*}\Big(|x|^{-(N-\alpha)}\ast K_{1}(r,x^{''}) (Y_{\overline{r},\overline{x}'',\lambda}^{2_{\alpha}^{*}-1}
\varphi)\Big)Y_{\overline{r},\overline{x}'',\lambda}^{2_{\alpha}^{*}-1}
\endaligned$$
$$\aligned
l_{m1}=&K_{1}(r,x^{''})\Big(|x|^{-(N-\alpha)}\ast K_{1}(r,x^{''})|Y_{\overline{r},\overline{x}'',\lambda}|^{2_{\alpha}^{*}}\Big)Y_{\overline{r},\overline{x}'',\lambda}^{2_{\alpha}^{*}-1}-\xi\sum_{j=1}^{m}\Big(|x|^{-(N-\alpha)}\ast |V_{z_{j},\lambda}|^{2_{\alpha}^{*}}\Big)V_{z_{j},\lambda}^{2_{\alpha}^{*}-1}\\
&+Z_{\overline{r},\overline{x}'',\lambda}^{\ast}\Delta\xi+2\nabla\xi\nabla Z_{\overline{r},\overline{x}'',\lambda}^{\ast},
\endaligned$$
$$\aligned
N_{2}(\phi)&=K_{2}(r,x^{''})\Big(|x|^{-(N-\alpha)}\ast K_{2}(r,x^{''})|(Z_{\overline{r},\overline{x}'',\lambda}+\phi)|^{2_{\alpha}^{*}}\Big)(Z_{\overline{r},\overline{x}'',\lambda}+\phi)^{2_{\alpha}^{*}-1}\\
&-K_{2}(r,x^{''})\Big(|x|^{-(N-\alpha)}\ast K_{2}(r,x^{''})|Z_{\overline{r},\overline{x}'',\lambda}|^{2_{\alpha}^{*}}\Big)Z_{\overline{r},\overline{x}'',\lambda}^{2_{\alpha}^{*}-1}\\
&-K_{2}(r,x^{''})(2_{\alpha}^{*}-1)\Big(|x|^{-(N-\alpha)}\ast K_{2}(r,x^{''})|Z_{\overline{r},\overline{x}'',\lambda}|^{2_{\alpha}^{*}}\Big)Z_{\overline{r},\overline{x}'',\lambda}^{2_{\alpha}^{*}-2}\phi\\
&-K_{2}(r,x^{''})2_{\alpha}^{*}\Big(|x|^{-(N-\alpha)}\ast K_{2}(r,x^{''}) (Z_{\overline{r},\overline{x}'',\lambda}^{2_{\alpha}^{*}-1}
\phi)\Big)Z_{\overline{r},\overline{x}'',\lambda}^{2_{\alpha}^{*}-1},
\endaligned$$
and
$$\aligned
l_{m2}=&K_{2}(r,x^{''})\Big(|x|^{-(N-\alpha)}\ast K_{2}(r,x^{''})|Z_{\overline{r},\overline{x}'',\lambda}|^{2_{\alpha}^{*}}\Big)Z_{\overline{r},\overline{x}'',\lambda}^{2_{\alpha}^{*}-1}-\xi\sum_{j=1}^{m}\Big(|x|^{-(N-\alpha)}\ast |U_{z_{j},\lambda}|^{2_{\alpha}^{*}}\Big)U_{z_{j},\lambda}^{2_{\alpha}^{*}-1}\\
&+Y_{\overline{r},\overline{x}'',\lambda}^{\ast}\Delta\xi+2\nabla\xi\nabla Y_{\overline{r},\overline{x}'',\lambda}^{\ast}.
\endaligned$$

We are going to utilize the contraction mapping theorem to demonstrate that \eqref{c16} is uniquely solvable. At first we need to estimate $N(\phi,\varphi):=\left( N_{1}(\varphi), N_{2}(\phi)\right) $ and $l_{m}:=(l_{m1},l_{m2})$ respectively.

\begin{lem}\label{C4}
	There is a constant $C> 0$, such that
	\begin{equation}\label{c17}
		\|N(\phi,\varphi)\|_{\ast\ast}\leq C\|(\phi,\varphi)\|_{\ast}^{2},
	\end{equation}
\end{lem}
\begin{proof}
%We know that $2_{\alpha}^{*}<3$, then
%\begin{equation}\label{N1}
%\aligned
%|N_{1}(\varphi)|&\leq C\left|\Big(|x|^{-(N-\alpha)}\ast |Y_{\overline{r},\overline{x}'',\lambda}|^{2_{\alpha}^{*}-1}\varphi\Big)Y_{\overline{r},\overline{x}'',\lambda}^{2_{\alpha}^{*}-2}\varphi+\Big(|x|^{-(N-\alpha)}\ast |Y_{\overline{r},\overline{x}'',\lambda}|^{2_{\alpha}^{*}-2}\varphi^{2}\Big)Y_{\overline{r},\overline{x}'',\lambda}^{2_{\alpha}^{*}-1}\right|\\
%&+ C\left|\Big(|x|^{-(N-\alpha)}\ast |Y_{\overline{r},\overline{x}'',\lambda}|^{2_{\alpha}^{*}-2}\varphi^{2}\Big)Y_{\overline{r},\overline{x}'',\lambda}^{2_{\alpha}^{*}-2}\varphi\right|+C\left|\Big(|x|^{-(N-\alpha)}\ast \varphi^{2_{\alpha}^{*}}\Big)\varphi^{2_{\alpha}^{*}-1}\right|
%\endaligned
%\end{equation}
By the same method in \cite{CYZ}, we know that
$$\aligned
&\Big(|x|^{-(N-\alpha)}\ast |Y_{\overline{r},\overline{x}'',\lambda}|^{2_{\alpha}^{*}-1}\varphi\Big)Y_{\overline{r},\overline{x}'',\lambda}^{2_{\alpha}^{*}-2}\varphi\\
%&\leq C\|\varphi\|_{\ast}^{2}\lambda^{\frac{N+2}{2}}\sum_{j=1}^{m}{\frac{1}{(1+\lambda |x-z_{j}|)^{\min{\left\lbrace \frac{N+2}{2}-\tau_{1}\frac{N+\alpha}{N-2}+\tau,N-\alpha\right\rbrace }}}
%\left( \sum_{j=1}^{m}\frac{1}{(1+\lambda |x-z_{j}|)^{\frac{N-2}{2}+\tau}}\right)^{2_{\alpha}^{*}-1}
\\
&\leq C\|\varphi\|_{\ast}^{2}\lambda^{\frac{N+2}{2}}\sum_{j=1}^{m}
\frac{1}{(1+\lambda |x-z_{j}|)^{{\min{\left\lbrace \frac{N+2}{2}-\tau_{1}\frac{N+\alpha}{N-2}+\tau,N-\alpha\right\rbrace }}}}
\sum_{j=1}^{m}\frac{1}{(1+\lambda |x-z_{j}|)^{\frac{\alpha+2}{2}+\tau}}\left( \sum_{j=1}^{m}\frac{1}{(1+\lambda |x-z_{j}|)^{\tau}}\right)^{2_{\alpha}^{*}-2}
\\
&\leq C\|\varphi\|_{\ast}^{2}\lambda^{\frac{N+2}{2}}\sum_{j=1}^{m}\frac{1}{(1+\lambda |x-z_{j}|)^{\frac{N+2}{2}+\tau}}.
\endaligned$$
We can also use an estimate similar to the one in \eqref{t0} to get the following estimate
$$\aligned
&|x|^{-(N-\alpha)}\ast |Y_{\overline{r},\overline{x}'',\lambda}|^{2_{\alpha}^{*}-2}\varphi^{2}\\
%&\leq C\|\varphi\|_{\ast}^{2}\left(|x|^{-(N-\alpha)}\ast\left((\sum_{j=1}^{m}\frac{\lambda^{\frac{N-2}{2}}}{(1+\lambda |x-z_{j}|)^{N-2}})^{2_{\alpha}^{*}-2} (\sum_{j=1}^{m}\frac{\lambda^{\frac{N-2}{2}}}{(1+\lambda |x-z_{j}|)^{\frac{N-2}{2}+\tau}})^{2} \right) \right) \\
&\leq C\|\varphi\|_{\ast}^{2}\sum_{j=1}^{m}\int_{\Omega_{j}}\frac{1}{|x-y|^{N-\alpha}}\left(\frac{\lambda^{\frac{\alpha-N+4}{2}}}{(1+\lambda |x-z_{j}|)^{\alpha-N+4-\tau_{1}\frac{\alpha-N+4}{N-2}}}\frac{\lambda^{N-2}}{(1+\lambda |x-z_{j}|)^{N-2+2\tau-2\tau_{1}}} \right)\\
%&= C\|\varphi\|_{\ast}^{2}\sum_{j=1}^{m}\int_{\Omega_{j}}\frac{1}{|x-y|^{N-\alpha}}\frac{\lambda^{\frac{\alpha+N}{2}}}{(1+\lambda |x-z_{j}|)^{\alpha+2-\tau_{1}\frac{\alpha+N}{N-2}+2\tau}}\\
&\leq C\|\varphi\|_{\ast}^{2}\sum_{j=1}^{m}\frac{\lambda^{\frac{N-\alpha}{2}}}{(1+\lambda |x-z_{j}|)^{\min\left\lbrace 2-\tau_{1}\frac{\alpha+N}{N-2}+2\tau,N-\alpha\right\rbrace }},
\endaligned$$
and thus
$$\aligned
&\left( |x|^{-(N-\alpha)}\ast |Y_{\overline{r},\overline{x}'',\lambda}|^{2_{\alpha}^{*}-2}\varphi^{2}\right)Y_{\overline{r},\overline{x}'',\lambda}^{2_{\alpha}^{*}-1}\\
% &\leq
%C\|\varphi\|_{\ast}^{2}\sum_{j=1}^{m}\frac{\lambda^{\frac{N-\alpha}{2}}}{(1+\lambda |x-z_{j}|)^{\min\left\lbrace 2-\tau_{1}\frac{\alpha+N}{N-2}+2\tau},\alpha\right\rbrace }}\left( \sum_{j=1}^{m}\frac{1}{(1+\lambda |x-z_{j}|)^{\frac{N-2}{2}+\tau}}\right)^{2_{\alpha}^{*}-1}\\
 &\leq
C\|\varphi\|_{\ast}^{2}\lambda^{\frac{N+2}{2}}\sum_{j=1}^{m}\frac{\lambda^{\frac{N-\alpha}{2}}}{(1+\lambda |x-z_{j}|)^{\min\left\lbrace 2-\tau_{1}\frac{\alpha+N}{N-2}+2\tau,N-\alpha\right\rbrace }}\sum_{j=1}^{m}\frac{1}{(1+\lambda |x-z_{j}|)^{{\frac{\alpha+2}{2}+\tau}}}\left( \sum_{j=1}^{m}\frac{1}{(1+\lambda |x-z_{j}|)^{\tau}}\right)^{2_{\alpha}^{*}-2}
\\
&\leq C\|\varphi\|_{\ast}^{2}\lambda^{\frac{N+2}{2}}\sum_{j=1}^{m}\frac{1}{(1+\lambda |x-z_{j}|)^{\frac{N+2}{2}+\tau}}.
\endaligned$$
Likewise, we have
$$\aligned
\left( |x|^{-(N-\alpha)}\ast |Y_{\overline{r},\overline{x}'',\lambda}|^{2_{\alpha}^{*}-2}\varphi^{2}\right)Y_{\overline{r},\overline{x}'',\lambda}^{2_{\alpha}^{*}-2}\varphi\leq C\|\varphi\|_{\ast}^{2}\lambda^{\frac{N+2}{2}}\sum_{j=1}^{m}\frac{1}{(1+\lambda |x-z_{j}|)^{\frac{N+2}{2}+\tau}}.
\endaligned$$
We continue in the same way to obtain
$$\aligned
|x|^{-(N-\alpha)}\ast \varphi^{2_{\alpha}^{*}}
\leq C\|\varphi\|_{\ast}^{2_{\alpha}^{*}}\sum_{j=1}^{m} \frac{\lambda^{\frac{N-\alpha}{2}}}{(1+\lambda |x-z_{j}|)^{\min\left\lbrace \frac{N-\alpha}{2}+\tau,N-\alpha\right\rbrace }}.
\endaligned$$
Hence,
$$\aligned
\Big(|x|^{-(N-\alpha)}\ast \varphi^{2_{\alpha}^{*}}\Big)\varphi^{2_{\alpha}^{*}-1}
\leq C\|\varphi\|_{\ast}^{22_{\alpha}^{*}-1}\lambda^{\frac{N+2}{2}}\sum_{j=1}^{m}\frac{1}{(1+\lambda |x-z_{j}|)^{{\frac{N+2}{2}+\tau}}}.
\endaligned$$
By combining the above estimates together, we conclude
$$
\|N_{1}(\varphi)\|_{\ast\ast}\leq C\|\varphi\|_{\ast}^{\min\left\lbrace 22_{\alpha}^{*}-1,2\right\rbrace }\leq C\|\varphi\|_{\ast}^{2}.
$$
Moreover,
$$
\|N_{2}(\phi)\|_{\ast\ast}\leq C\|\phi\|_{\ast}^{2}
$$
can be estimated in the same way. Finally,
\begin{equation}\nonumber
	\|N(\phi,\varphi)\|_{\ast\ast}\leq C\|(\phi,\varphi)\|_{\ast}^{2}.
\end{equation}
\end{proof}
\begin{lem}\label{C5}
	There is a small
	 constant $\varepsilon> 0$, such that
	\begin{equation}\label{c18}
		\|l_{m}\|_{\ast\ast}\leq C(\frac{1}{\lambda})^{1+\varepsilon}.
	\end{equation}
\end{lem}
\begin{proof}
Since $	\|l_{m}\|_{\ast\ast}=\|l_{m1}\|_{\ast\ast}+	\|l_{m2}\|_{\ast\ast}$, we only need to deal with $\|l_{m1}\|_{\ast\ast}$ here. It is clear that
$$\aligned
l_{m1}=&K_{1}(r,x^{''})\Big(|x|^{-(N-\alpha)}\ast K_{1}(r,x^{''})|Y_{\overline{r},\overline{x}'',\lambda}|^{2_{\alpha}^{*}}\Big)Y_{\overline{r},\overline{x}'',\lambda}^{2_{\alpha}^{*}-1}-\xi\sum_{j=1}^{m}\Big(|x|^{-(N-\alpha)}\ast |V_{z_{j},\lambda}|^{2_{\alpha}^{*}}\Big)V_{z_{j},\lambda}^{2_{\alpha}^{*}-1}\\
&+Z_{\overline{r},\overline{x}'',\lambda}^{\ast}\Delta\xi+2\nabla\xi\nabla Z_{\overline{r},\overline{x}'',\lambda}^{\ast}\\
=&K_{1}(r,x^{''})\left[ \Big(|x|^{-(N-\alpha)}\ast K_{1}(r,x^{''})|Y_{\overline{r},\overline{x}'',\lambda}|^{2_{\alpha}^{*}}\Big)Y_{\overline{r},\overline{x}'',\lambda}^{2_{\alpha}^{*}-1}-\xi\sum_{j=1}^{m}\Big(|x|^{-(N-\alpha)}\ast (K_{1}(r,x^{''})|V_{z_{j},\lambda}|^{2_{\alpha}^{*}})\Big)V_{z_{j},\lambda}^{2_{\alpha}^{*}-1}\right] \\
&+K_{1}(r,x^{''})\xi\sum_{j=1}^{m}\Big(|x|^{-(N-\alpha)}\ast (K_{1}(r,x^{''})-1)|V_{z_{j},\lambda}|^{2_{\alpha}^{*}})\Big)V_{z_{j},\lambda}^{2_{\alpha}^{*}-1}\\
&+(K_{1}(r,x^{''})-1)\xi\sum_{j=1}^{m}\Big(|x|^{-(N-\alpha)}\ast |V_{z_{j},\lambda}|^{2_{\alpha}^{*}}\Big)V_{z_{j},\lambda}^{2_{\alpha}^{*}-1}+\left( Z_{\overline{r},\overline{x}'',\lambda}^{\ast}\Delta\xi+2\nabla\xi\nabla Z_{\overline{r},\overline{x}'',\lambda}^{\ast}\right)\\
&:=J_{1}+J_{2}+J_{3}+J_{4}.
\endaligned$$
We know that $|x-z_{j}|\geq|x-z_{1}|, \ \forall x \in\Omega_{1}$, the estimate of $J_{1}$ can be directly obtained from Lemma 2.4 of \cite{CYZ} then
$$
\|J_{1}\|_{**}\leq \left( C\frac{1}{\lambda}\right)^{1+\varepsilon}.
$$
Next we estimate $J_3$. First, we use the Taylor expansion ro rewrite the function $K_1$ and $K_2$ in the neighborhood of $x_0$ in the following form
$$
K_{1}(x)=K(x_{0})+\nabla K_{1}(x_{0})(x-x_{0})+\sum_{i,j=1}^{N}\frac{1}{2}\frac{\partial^{2}K_{1}(x_{0})}{\partial x_{i}\partial x_{j}}(x_{i}-x_{0i})(x_{j}-x_{0j})+o\left(|x-x_{0}|^{2} \right),
$$
and
$$
K_{2}(x)=K_{2}(x_{0})+\nabla K_{2}(x_{0})(x-x_{0})+\sum_{i,j=1}^{N}\frac{1}{2}\frac{\partial^{2}K_{2}(x_{0})}{\partial x_{i}\partial x_{j}}(x_{i}-x_{0i})(x_{j}-x_{0j})+o\left(|x-x_{0}|^{2} \right).
$$
We assume $K_{1}(x_{0})=K_{2}(x_{0})=1$ and in the region $|(r,x^{"})-(r_{0},x^{''}_{0})|\leq \frac{a}{\lambda^{\frac{1}{2}+\varepsilon}}$, where $a>0$ is a fixed constant. Thus
$$
\aligned
|J_{2}|&=\left| K_{1}(r,x^{''})\xi\sum_{j=1}^{m}\Big(|x|^{-(N-\alpha)}\ast (K_{1}(r,x^{''})-1)|V_{z_{j},\lambda}|^{2_{\alpha}^{*}})\Big)V_{z_{j},\lambda}^{2_{\alpha}^{*}-1}\right|\\
&\leq C\left|\xi\sum_{j=1}^{m}\Big(|x|^{-(N-\alpha)}\ast \left( \sum_{i,j=1}^{N}\frac{1}{2}\frac{\partial^{2}K_{1}(x_{0})}{\partial x_{i}\partial x_{j}}(x_{i}-x_{0i})(x_{j}-x_{0j})+o\left(|x-x_{0}|^{2} \right)\right) |V_{z_{j},\lambda}|^{2_{\alpha}^{*}}\Big)V_{z_{j},\lambda}^{2_{\alpha}^{*}-1}\right| \\
&\leq \frac{C}{\lambda^{1+2\varepsilon}}\sum_{j=1}^{m}\frac{\xi\lambda^{\frac{N-\alpha}{2}}}{(1+\lambda|x-z_{j}|)^{N-\alpha}}\frac{\lambda^{\frac{2+\alpha}{2}}}{(1+\lambda|x-z_{j}|)^{2+\alpha}}\\
&\leq\frac{C}{\lambda^{1+2\varepsilon}}\sum_{j=1}^{m}\frac{\lambda^{\frac{N+2}{2}}}{(1+\lambda|x-z_{j}|)^{N+1}}\leq C\left( \frac{1}{\lambda}\right) ^{1+\varepsilon}\sum_{j=1}^{m}\frac{\lambda^{\frac{N+2}{2}}}{(1+\lambda|x-z_{j}|)^{\frac{N+2}{2}+\tau}}.
\endaligned
$$
We consider the case that in the region $\frac{a}{\lambda^{\frac{1}{2}+\varepsilon}}\leq|(r,x^{'})-(r_{0},x^{''}_{0})|\leq 2\sigma$, by \cite{PWW}
we know
$$
\frac{1}{1+\lambda|x-z_{j}|}\leq \frac{C}{\lambda^{\frac{1}{2}-\varepsilon}}.
$$
Then
$$
\aligned
|J_{2}|&\leq C\sum_{j=1}^{m}\frac{\xi\lambda^{\frac{N+2}{2}}}{(1+\lambda|x-z_{j}|)^{N+2}}\\
&\leq C\left( \frac{1}{\lambda}\right) ^{1+\varepsilon}\sum_{j=1}^{m}\frac{\lambda^{\frac{N+2}{2}}}{(1+\lambda|x-z_{j}|)^{\frac{N+2}{2}+\tau}}\frac{\lambda^{1+\varepsilon}}{(1+\lambda|x-z_{j}|)^{\frac{N}{2}-\tau}}\\
&\leq C\left( \frac{1}{\lambda}\right) ^{1+\varepsilon}\sum_{j=1}^{m}\frac{\lambda^{\frac{N+2}{2}}}{(1+\lambda|x-z_{j}|)^{\frac{N+2}{2}+\tau}}
\endaligned
$$
Combing these,  we conclude that
$$
\|J_{2}\|_{**}\leq  C\left( \frac{1}{\lambda}\right)^{1+\varepsilon}.
$$
Similarly,
$$
\|J_{3}\|_{**}\leq C\left( \frac{1}{\lambda}\right)^{1+\varepsilon}.
$$
By Lemma 2.4 in \cite{PWW} the last term satisfies
$$
\|J_{4}\|_{**}\leq C\left( \frac{1}{\lambda}\right)^{1+\varepsilon}.
$$
We can conclude that
\begin{equation}\nonumber
	\|l_{m1}\|_{\ast\ast}\leq C\left( \frac{1}{\lambda}\right) ^{1+\varepsilon},
\end{equation}
Finally,
\begin{equation}\nonumber
	\|l_{m}\|_{\ast\ast}\leq C\left( \frac{1}{\lambda}\right) ^{1+\varepsilon}.
\end{equation}
\end{proof}

We will apply the contraction mapping argument to make the following estimates for the constants $c_l$ and the solutions $\phi$ and $\varphi$. We know that $\lambda\in [L_0m^{\frac{N-2}{N-4}},L_1m^{\frac{N-2}{N-4}}]$, set
$$\aligned
&\mathcal{N}=\Big\{(\phi,\varphi):\phi,\varphi\in C(\mathbb{R}^N)\cap H_{s}, \ \|(\phi,\varphi)\|_{\ast}\leq\frac{1}{\lambda}, \\
&\displaystyle \hspace{14.14mm}\sum_{j=1}^{m}\Big< \Big((2_{\alpha}^{*}-1)\Big(|x|^{-(N-\alpha)}\ast |Y_{z_j,\lambda}|^{2_{\alpha}^{*}}\Big)Y_{z_j,\lambda}^{2_{\alpha}^{*}-2}Y_{j,l}+2_{\alpha}^{*}\Big(|x|^{-(N-\alpha)}\ast (Y_{z_j,\lambda}^{2_{\alpha}^{*}-1}
Y_{j,l})\Big)Y_{z_j,\lambda}^{2_{\alpha}^{*}-1},\\
&\displaystyle	\hspace{14.14mm}(2_{\alpha}^{*}-1)\Big(|x|^{-(N-\alpha)}\ast |Z_{z_j,\lambda}|^{2_{\alpha}^{*}}\Big)Z_{z_j,\lambda}^{2_{\alpha}^{*}-2}Z_{j,l}+2_{\alpha}^{*}\Big(|x|^{-(N-\alpha)}\ast (Z_{z_j,\lambda}^{2_{\alpha}^{*}-1}
Z_{j,l})\Big)Z_{z_j,\lambda}^{2_{\alpha}^{*}-1} \Big) , (\phi,\varphi)\Big>=0,\\
&\displaystyle	\hspace{14.14mm}l=1,2,\dots,N.
\Big\}
\endaligned$$
and
$$
\mathcal{A}(\phi,\varphi):=L_{m}(N(\phi,\varphi)+L_{m}(l_{m}),
$$
where $L_m$ is defined in Lemma \ref{C2}. In that case, \eqref{c16} can be written as
\begin{equation}\label{c19}
	(\phi,\varphi)=\mathcal{A}(\phi,\varphi).
\end{equation}
\begin{lem}\label{C3}
	There is an integer $m_0 > 0$, such that for each $m \geq m_0$, $\lambda\in[L_0m^{\frac{N-2}{N-4}},L_1m^{\frac{N-2}{N-4}}]$,
	$|(\overline{r},\overline{x}^{''})-(r_{0},x_{0}^{''})|\leq\frac{1}{\lambda^{1-\theta}}$, where $\theta> 0$ is a fixed small constant, \eqref{c14} has a unique solution $(\phi,\varphi)= (\phi_{\overline{r},\overline{x}'',\lambda},\varphi_{\overline{r},\overline{x}'',\lambda})\in H_{s}\times H_{s}$,
	satisfying
	\begin{equation}\label{c15}
		\|(\phi,\varphi)\|_{\ast}\leq C(\frac{1}{\lambda})^{1+\varepsilon}, \ \ |c_{l}|\leq C(\frac{1}{\lambda})^{1+n_{l}+\varepsilon},
	\end{equation}
	where $\varepsilon> 0$ is a small constant.
\end{lem}
\begin{proof}
	We will show that $\mathcal{A}$ is a contraction mapping from $\mathcal{N}$ to $\mathcal{N}$. Firstly, $\mathcal{A}$ maps $\mathcal{N}$ to $\mathcal{N}$ because for any $(\phi,\varphi)\in \mathcal{N}$,
	$$
	\|\mathcal{A}\|_{\ast}\leq  C(\|N_{1}(\varphi)\|_{\ast\ast}+\|N_{2}(\phi)\|_{\ast\ast}+\|l_{m1}\|_{\ast\ast}+\|l_{m2}\|_{\ast\ast})\leq C(\|(\phi,\varphi)\|_{\ast}^{2}+(\frac{1}{\lambda})^{1+\varepsilon})\leq
	\frac{1}{\lambda}.
	$$
	Secondly, it is easy to check that
	$$\aligned
	|N_{1}(\varphi_{1})-N_{1}(\varphi_{2})|\leq |N^{'}_{1}(\varphi_{1}+\theta(\varphi_{2}-\varphi_{1}))||\varphi_{2}-\varphi_{1}|\leq C(f(\varphi_{1})+f(\varphi_{2}))|\varphi_{2}-\varphi_{1}|,
	\endaligned$$
	where $\varphi_{1}$, $\varphi_{2}$ in $\mathcal{N}$ and
		$$\aligned
	f(\varphi)&=\Big(|x|^{-(N-\alpha)}\ast Y_{\overline{r},\overline{x}'',\lambda}^{2_{\alpha}^{*}-1}\Big)Y_{\overline{r},\overline{x}'',\lambda}^{2_{\alpha}^{*}-2}\varphi+\Big(|x|^{-(N-\alpha)}\ast Y_{\overline{r},\overline{x}'',\lambda}^{2_{\alpha}^{*}-1}\varphi\Big)Y_{\overline{r},\overline{x}'',\lambda}^{2_{\alpha}^{*}-2}\\
	&+\Big(|x|^{-(N-\alpha)}\ast Y_{\overline{r},\overline{x}'',\lambda}^{2_{\alpha}^{*}-2}\varphi\Big)\Big(Y_{\overline{r},\overline{x}'',\lambda}^{2_{\alpha}^{*}-1}+Y_{\overline{r},\overline{x}'',\lambda}^{2_{\alpha}^{*}-2}\varphi\Big)+\Big(|x|^{-(N-\alpha)}\ast Y_{\overline{r},\overline{x}'',\lambda}^{2_{\alpha}^{*}-2}\varphi^{2}\Big)Y_{\overline{r},\overline{x}'',\lambda}^{2_{\alpha}^{*}-2}\\
    &+\Big(|x|^{-(N-\alpha)}\ast \varphi^{2_{\alpha}^{*}-1}\Big)
	\varphi^{2_{\alpha}^{*}-1}+\Big(|x|^{-(N-\alpha)}\ast \varphi^{2_{\alpha}^{*}}\Big)
	\varphi^{2_{\alpha}^{*}-2}.\endaligned$$
	Hence that
	$$\aligned
	\|\mathcal{A}((\phi_1,\varphi_{1}))-\mathcal{A}((\phi_2,\varphi_{2}))\|_{\ast}&=\|L_{m}(N_(\phi_1,\varphi_{1}))-L_{m}(N(\phi_2,\varphi_{2})\|_{\ast}\\
	&\leq C\|N_{1}(\varphi_{1})-N_{1}(\varphi_{2})\|_{\ast\ast}
	+C\|N_{2}(\phi_1)-N_{2}(\phi_2)\|_{\ast\ast}\leq\frac{1}{2}\|(\phi_1,\varphi_{1})-(\phi_2,\varphi_{2})\|_{\ast}.
	\endaligned$$
	Therefore $\mathcal{A}$ is a contraction mapping.
	We conclude that there exists a unique $(\phi,\varphi)\in\mathcal{N}$ such that \eqref{c19} holds, due to the Contraction Mapping Theorem. Moreover, we can get
	$$
	\|(\phi,\varphi)\|_{\ast}\leq C(\frac{1}{\lambda})^{1+\varepsilon},
	$$
	by Lemmas \ref{C2}, \ref{C4} and \ref{C5}. From \eqref{c13} we have
	$$
	|c_{l}|\leq C(\frac{1}{\lambda})^{1+n_{l}+\varepsilon}.
	$$
\end{proof}

\section{Local Poho\v{z}aev identities methods}
Inspired by \cite{PWY}, we will use local Poho\v{z}aev identities to identify the location of the bubbles. We know that the functional corresponding to \eqref{CFL} is defined by
$$\aligned
J(u,v)=&\int_{\mathbb{R}^{N}}\nabla u\cdot\nabla vdx
-\frac{1}{2\cdot 2_{\alpha}^{*}}
\int_{\mathbb{R}^N}\int_{\mathbb{R}^N}
\frac{K_{1}(r,x^{''})K_{1}(r,y^{''})|v(x)|^{2_{\alpha}^{*}}|v(y)|^{2_{\alpha}^{*}}}{|x-y|^{N-\alpha}}dxdy\\
&-\frac{1}{2\cdot 2_{\alpha}^{*}}
\int_{\mathbb{R}^N}\int_{\mathbb{R}^N}
\frac{K_{2}(r,x^{''})K_{2}(r,y^{''})|u(x)|^{2_{\alpha}^{*}}|u(y)|^{2_{\alpha}^{*}}}{|x-y|^{N-\alpha}}dxdy.
\endaligned$$	

\begin{lem}\label{D1}
	Suppose that $(\overline{r},\overline{x}'',\lambda)$ satisfies
	\begin{equation}\label{d1}
		\aligned
		&\int_{D_{\rho}}\left( -\Delta u_{m}- K_{1}(|x'|,x'')\Big(|x|^{-(N-\alpha)}\ast K_{1}(|x'|,x'')|v_{m}|^{2_{\alpha}^{*}}\Big)v_{m}^{2_{\alpha}^{*}-1}\right) \langle x,\nabla v_{m}\rangle dx\\
		&+\int_{D_{\rho}}\left( -\Delta v_{m}- K_{2}(|x'|,x'')\Big(|x|^{-(N-\alpha)}\ast K_{2}(|x'|,x'')|u_{m}|^{2_{\alpha}^{*}}\Big)u_{m}^{2_{\alpha}^{*}-1}\right) \langle x,\nabla u_{m}\rangle dx=0,
		\endaligned
	\end{equation}
	\begin{equation}\label{d2}
		\aligned
		&\int_{D_{\rho}}\left( -\Delta u_{m}- K_{1}(|x'|,x'')\Big(|x|^{-(N-\alpha)}\ast K_{1}(|x'|,x'')|v_{m}|^{2_{\alpha}^{*}}\Big)v_{m}^{2_{\alpha}^{*}-1}\right)\frac{\partial v_{m}}{\partial x_{i}} dx\\
		&+\int_{D_{\rho}}\left( -\Delta v_{m}- K_{2}(|x'|,x'')\Big(|x|^{-(N-\alpha)}\ast K_{2}(|x'|,x'')|u_{m}|^{2_{\alpha}^{*}}\Big)u_{m}^{2_{\alpha}^{*}-1}\right)\frac{\partial u_{m}}{\partial x_{i}} dx=0, i=3,\cdots,N,
		\endaligned
	\end{equation}
	and
	\begin{equation}\label{d3}
		\aligned	
		&\int_{\mathbb{R}^{N}}\left( -\Delta u_{m}- K_{1}(|x'|,x'')\Big(|x|^{-(N-\alpha)}\ast K_{1}(|x'|,x'')|v_{m}|^{2_{\alpha}^{*}}\Big)v_{m}^{2_{\alpha}^{*}-1}\right)\frac{\partial Y_{\overline{r},\overline{x}'',\lambda}}{\partial \lambda} dx\\
		&+\int_{\mathbb{R}^{N}}\left( -\Delta v_{m}- K_{2}(|x'|,x'')\Big(|x|^{-(N-\alpha)}\ast K_{2}(|x'|,x'')|u_{m}|^{2_{\alpha}^{*}}\Big)u_{m}^{2_{\alpha}^{*}-1}\right)\frac{\partial Z_{\overline{r},\overline{x}'',\lambda}}{\partial \lambda} dx=0,
		\endaligned
	\end{equation}
	where $u_{m}=Z_{\overline{r},\overline{x}'',\lambda}+
	\phi_{\overline{r},\overline{x}'',\lambda}$, $v_{m}=Y_{\overline{r},\overline{x}'',\lambda}+
	\varphi_{\overline{r},\overline{x}'',\lambda}$ and $D_{\rho}=\{(r, x''):|(r, x'')-(r_0, x_0'' )|\leq\rho\}$ with $\rho\in(2\delta, 5\delta)$. Then $c_{i}=0, i=1,\cdot\cdot\cdot,N$.
\end{lem}
\begin{proof}
	We know that  $Z_{\overline{r},\overline{x}'',\lambda}=0$ and $Y_{\overline{r},\overline{x}'',\lambda}=0$ in $\mathbb{R}^{N}\backslash D_{\rho}$. When \eqref{d1}, \eqref{d2} and \eqref{d3} hold, we deduce that
\begin{equation}\label{d4}
	\aligned
	\sum_{l=1}^{N}c_{l}\sum_{j=1}^{m}& \int_{\mathbb{R}^{N}}\Big\lbrace\Big[(2_{\alpha}^{*}-1)\Big(|x|^{-(N-\alpha)}\ast |Y_{z_j,\lambda}|^{2_{\alpha}^{*}}\Big)Y_{z_j,\lambda}^{2_{\alpha}^{*}-2}Y_{j,l}+2_{\alpha}^{*}\Big(|x|^{-(N-\alpha)}\ast (Y_{z_j,\lambda}^{2_{\alpha}^{*}-1}
	Y_{j,l})\Big)Y_{z_j,\lambda}^{2_{\alpha}^{*}-1}\Big]f \\
	+&\Big[(2_{\alpha}^{*}-1)\Big(|x|^{-(N-\alpha)}\ast |Z_{z_j,\lambda}|^{2_{\alpha}^{*}}\Big)Z_{z_j,\lambda}^{2_{\alpha}^{*}-2}Z_{j,l}+2_{\alpha}^{*}\Big(|x|^{-(N-\alpha)}\ast (Z_{z_j,\lambda}^{2_{\alpha}^{*}-1}
	Z_{j,l})\Big)Z_{z_j,\lambda}^{2_{\alpha}^{*}-1}\Big]g \Big\rbrace dx=0,
	\endaligned
\end{equation}
by \eqref{c14} for $f=\langle x,\nabla v_{m}\rangle$,$\frac{\partial v_{m}}{\partial x_{i}}$, $\frac{\partial Y_{\overline{r},\overline{x}'',\lambda}}{\partial \lambda}$ and $g=\langle x,\nabla u_{m}\rangle$, $\frac{\partial u_{m}}{\partial x_{i}}$, $\frac{\partial Z_{\overline{r},\overline{x}'',\lambda}}{\partial \lambda}$, where $i=3,\cdot\cdot\cdot,N$.
An easy calculation shows that
$$\aligned
&\left|\int_{\mathbb{R}^{N}}\Big(|x|^{-(N-\alpha)}\ast |Y_{z_j,\lambda}|^{2_{\alpha}^{*}}\Big)Y_{z_j,\lambda}^{2_{\alpha}^{*}-2}Y_{j,i}
\frac{\partial Y_{z_l,\lambda}}{\partial x_{i}}dx\right|\\
&\leq C\lambda^{2}\int_{\mathbb{R}^{N}}\frac{1}{(1+|x-\lambda z_{j} |^{2})^{2}}\frac{(x_{i}-\lambda\overline{x}_{i})^{2}}{(1+|x-\lambda z_{j} |^{2})^{\frac{N}{2}}}
\frac{1}{(1+|x-\lambda z_{l} |^{2})^{\frac{N}{2}}}dx,
\endaligned$$
where $(\overline{x}_{3}, \overline{x}_{4}, \cdot\cdot\cdot,\overline{x}_{N})=\overline{x}$ and $ i=3,\cdot\cdot\cdot,N$.  Furthermore, if $l=j$ we have
$$
\left|\int_{\mathbb{R}^{N}}\Big(|x|^{-4}\ast |Y_{z_l,\lambda}|^{2_{\alpha}^{*}}\Big)Y_{z_{l},\lambda}^{2_{\alpha}^{*}-2}Y_{l,i}
\frac{\partial Y_{z_l,\lambda}}{\partial x_{i}}dx\right|=O (\lambda^{2}),
$$
and if $l\neq j$
$$
\left|\int_{\mathbb{R}^{N}}\Big(|x|^{-(N-\alpha)}\ast |Y_{z_j,\lambda}|^{2_{\alpha}^{*}}\Big)Y_{z_{j},\lambda}^{2_{\alpha}^{*}-2}Y_{j,i}
\left(\sum_{l=1,\neq j}^{m}\frac{\partial Y_{z_l,\lambda}}{\partial x_{i}}\right)dx\right|=O (\lambda^{2-\varepsilon}),
$$
for some $\varepsilon> 0$. We apply this argument again to obtain
\begin{equation}\label{d5}
	\aligned
	\sum_{l=1}^{N}c_{l}\sum_{j=1}^{m}& \int_{\mathbb{R}^{N}}\Big\lbrace\Big[(2_{\alpha}^{*}-1)\Big(|x|^{-(N-\alpha)}\ast |Y_{z_j,\lambda}|^{2_{\alpha}^{*}}\Big)Y_{z_j,\lambda}^{2_{\alpha}^{*}-2}Y_{j,i}+2_{\alpha}^{*}\Big(|x|^{-(N-\alpha)}\ast (Y_{z_j,\lambda}^{2_{\alpha}^{*}-1}
	Y_{j,i})\Big)Y_{z_j,\lambda}^{2_{\alpha}^{*}-1}\Big]\frac{\partial Y_{\overline{r},\overline{x}'',\lambda}}{\partial x_i} \\
	+&\Big[(2_{\alpha}^{*}-1)\Big(|x|^{-(N-\alpha)}\ast |Z_{z_j,\lambda}|^{2_{\alpha}^{*}}\Big)Z_{z_j,\lambda}^{2_{\alpha}^{*}-2}Z_{j,i}+2_{\alpha}^{*}\Big(|x|^{-(N-\alpha)}\ast (Z_{z_j,\lambda}^{2_{\alpha}^{*}-1}
	Z_{j,i})\Big)Z_{z_j,\lambda}^{2_{\alpha}^{*}-1}\Big]\frac{\partial Z_{\overline{r},\overline{x}'',\lambda}}{\partial x_i} \Big\rbrace dx\\
	=&m(a_{1}+o(1))\lambda^{2}, \ i=3,\cdot\cdot\cdot, N,
	\endaligned
\end{equation}
for some constants $a_{1}\neq0$. Similar argument apply to
\begin{equation}\label{d6}
	\aligned
	\sum_{l=1}^{N}c_{l}\sum_{j=1}^{m}& \int_{\mathbb{R}^{N}}\Big\lbrace\Big[(2_{\alpha}^{*}-1)\Big(|x|^{-(N-\alpha)}\ast |Y_{z_j,\lambda}|^{2_{\alpha}^{*}}\Big)Y_{z_j,\lambda}^{2_{\alpha}^{*}-2}Y_{j,2}+2_{\alpha}^{*}\Big(|x|^{-(N-\alpha)}\ast (Y_{z_j,\lambda}^{2_{\alpha}^{*}-1}
	Y_{j,2})\Big)Y_{z_j,\lambda}^{2_{\alpha}^{*}-1}\Big]\langle x^{'},\nabla_{x^{'}} Y_{\overline{r},\overline{x}'',\lambda}\rangle \\
	+&\Big[(2_{\alpha}^{*}-1)\Big(|x|^{-(N-\alpha)}\ast |Z_{z_j,\lambda}|^{2_{\alpha}^{*}}\Big)Z_{z_j,\lambda}^{2_{\alpha}^{*}-2}Z_{j,2}+2_{\alpha}^{*}\Big(|x|^{-(N-\alpha)}\ast (Z_{z_j,\lambda}^{2_{\alpha}^{*}-1}
	Z_{j,2})\Big)Z_{z_j,\lambda}^{2_{\alpha}^{*}-1}\Big]\langle x^{'},\nabla_{x^{'}} Z_{\overline{r},\overline{x}'',\lambda}\rangle dx\\
	=&m(a_{2}+o(1))\lambda^{2},
	\endaligned
\end{equation}
and
\begin{equation}\label{d7}
	\aligned
	\sum_{l=1}^{N}c_{l}\sum_{j=1}^{m}& \int_{\mathbb{R}^{N}}\Big\lbrace\Big[(2_{\alpha}^{*}-1)\Big(|x|^{-(N-\alpha)}\ast |Y_{z_j,\lambda}|^{2_{\alpha}^{*}}\Big)Y_{z_j,\lambda}^{2_{\alpha}^{*}-2}Y_{j,1}+2_{\alpha}^{*}\Big(|x|^{-(N-\alpha)}\ast (Y_{z_j,\lambda}^{2_{\alpha}^{*}-1}
	Y_{j,1})\Big)Y_{z_j,\lambda}^{2_{\alpha}^{*}-1}\Big]\frac{\partial Y_{\overline{r},\overline{y}'',\lambda}}{\partial \lambda} \\
	+&\Big[(2_{\alpha}^{*}-1)\Big(|x|^{-(N-\alpha)}\ast |Z_{z_j,\lambda}|^{2_{\alpha}^{*}}\Big)Z_{z_j,\lambda}^{2_{\alpha}^{*}-2}Z_{j,1}+2_{\alpha}^{*}\Big(|x|^{-(N-\alpha)}\ast (Z_{z_j,\lambda}^{2_{\alpha}^{*}-1}
	Z_{j,1})\Big)Z_{z_j,\lambda}^{2_{\alpha}^{*}-1}\Big]\frac{\partial Z_{\overline{r},\overline{y}'',\lambda}}{\partial \lambda} \Big\rbrace dx\\
	=&\frac{m}{\lambda^{2}}(a_{3}+o(1)),
	\endaligned
\end{equation}
for some constants $a_{2}\neq0$ and $a_3>0$.

It is easy to see that
$$
	\aligned
&\int_{\mathbb{R}^{N}}\Big(|x|^{-(N-\alpha)}\ast |Y_{z_j,\lambda}|^{2_{\alpha}^{*}}\Big)Y_{z_j,\lambda}^{2_{\alpha}^{*}-2}Y_{j,1}\frac{\partial \varphi_{\overline{r},\overline{x}'',\lambda}}{\partial x_{i}}dx\\
&=\int_{\mathbb{R}^{N}}\int_{\mathbb{R}^{N}}\frac{|Y_{z_j,\lambda}(y)|^{2_{\alpha}^{*}}\left(Y_{z_j,\lambda}^{2_{\alpha}^{*}-2}(x)\frac{\partial Y_{j,1}}{\partial x_{i}}+\frac{\partial Y_{z_j,\lambda}^{2_{\alpha}^{*}-2}}{\partial x_{i}} Y_{j,1}(x)\right) \varphi_{\overline{r},\overline{x}'',\lambda}(x)}{|x-y|^{N-\alpha}}dxdy\\
&+\int_{\mathbb{R}^{N}}\int_{\mathbb{R}^{N}}\frac{|Y_{z_j,\lambda}(y)|^{2_{\alpha}^{*}}(x_{i}-y_{i})Y_{z_j,\lambda}(x)^{2_{\alpha}^{*}-2} Y_{j,1}(x) \varphi_{\overline{r},\overline{x}'',\lambda}(x)}{|x-y|^{N-\alpha+2}}dxdy,
	\endaligned
$$
where $j=1,\dots, m$ and $i=3,\dots, N$. By Lemma \ref{C3} we have
$$
\aligned
&\int_{\mathbb{R}^{N}}\int_{\mathbb{R}^{N}}\frac{|Y_{z_j,\lambda}(y)|^{2_{\alpha}^{*}}Y_{z_j,\lambda}^{2_{\alpha}^{*}-2}(x)\frac{\partial Y_{j,1}}{\partial x_{i}} \varphi_{\overline{r},\overline{x}'',\lambda}(x)}{|x-y|^{N-\alpha}}dxdy\\
&\leq C\|\varphi\|_{*}\int_{\mathbb{R}^{N}}\left( \frac{\lambda}{1+\lambda^{2}|x-z_{i}|^{2}}\right)^{\frac{N-\alpha}{2}}\left( \frac{\lambda}{1+\lambda^{2}|x-z_{i}|^{2}}\right)^{\frac{4+\alpha-N}{2}}\frac{\lambda^{\frac{N}{2}}|x-z_{j}|}{(1+\lambda^{2}|x-z_{j}|^{2})^{\frac{N}{2}}}\sum_{k=1}^{m}\frac{\lambda^{\frac{N-2}{2}}}{(1+\lambda|x-z_{k}|)^{\frac{N-2}{2}+\tau}}dx\\
&\leq C\|\varphi\|_{*}\int_{\mathbb{R}^{N}}\frac{|x-\lambda z_{j}|}{(1+|x-\lambda z_{j}|^{2})^{\frac{N+4}{2}}}\sum_{k=1}^{m}\frac{1}{(1+|x-\lambda z_{k}|)^{\frac{N-2}{2}+\tau}}dx\\
&\leq mC\|\varphi\|_{*}\int_{\Omega_{1}}\frac{1}{(1+|x-\lambda z_{1}|^{2})^{N+3}}\frac{1}{(1+|x-\lambda z_{1}|)^{\frac{N-2}{2}}}dx= mC\|\varphi\|_{*}=O(\frac{1}{\lambda^{\varepsilon}}).
\endaligned
$$
Similarly, we have
$$
\int_{\mathbb{R}^{N}}\int_{\mathbb{R}^{N}}\frac{|Y_{z_j,\lambda}(y)|^{2_{\alpha}^{*}}\frac{\partial Y_{z_j,\lambda}^{2_{\alpha}^{*}-2}}{\partial x_{i}} Y_{j,1}(x) \varphi_{\overline{r},\overline{x}'',\lambda}(x)}{|x-y|^{N-\alpha}}dxdy=O(\frac{1}{\lambda^{\varepsilon}})
$$
and
$$
\int_{\mathbb{R}^{N}}\int_{\mathbb{R}^{N}}\frac{|Y_{z_j,\lambda}(y)|^{2_{\alpha}^{*}}(x_{i}-y_{i})Y_{z_j,\lambda}(x)^{2_{\alpha}^{*}-2} Y_{j,1}(x) \varphi_{\overline{r},\overline{x}'',\lambda}(x)}{|x-y|^{N-\alpha+2}}dxdy=O(\frac{1}{\lambda^{\varepsilon}}).
$$
Consequently, we have the following estimates
$$
\int_{\mathbb{R}^{N}}\Big(|x|^{-(N-\alpha)}\ast |Y_{z_j,\lambda}|^{2_{\alpha}^{*}}\Big)Y_{z_j,\lambda}^{2_{\alpha}^{*}-2}Y_{j,1}\frac{\partial \phi_{\overline{r},\overline{x}'',\lambda}}{\partial x_{i}}dx=O(\frac{1}{\lambda^{\varepsilon}})
$$
and
$$
\int_{\mathbb{R}^{N}}\Big(|x|^{-(N-\alpha)}\ast Y_{z_j,\lambda}^{2_{\alpha}^{*}-1}Y_{j,1}\Big)Y_{z_j,\lambda}^{2_{\alpha}^{*}-1}\frac{\partial \phi_{\overline{r},\overline{x}'',\lambda}}{\partial x_{i}}dx=O(\frac{1}{\lambda^{\varepsilon}}).
$$
Similarly,
$$
\int_{\mathbb{R}^{N}}\Big(|x|^{-(N-\alpha)}\ast |Z_{z_j,\lambda}|^{2_{\alpha}^{*}}\Big)Z_{z_j,\lambda}^{2_{\alpha}^{*}-2}Z_{j,1}\frac{\partial \varphi_{\overline{r},\overline{x}'',\lambda}}{\partial x_{i}}dx=O(\frac{1}{\lambda^{\varepsilon}}),
$$
and
$$
\int_{\mathbb{R}^{N}}\Big(|x|^{-(N-\alpha)}\ast Z_{z_j,\lambda}^{2_{\alpha}^{*}-1}Z_{j,1}\Big)Z_{z_j,\lambda}^{2_{\alpha}^{*}-1}\frac{\partial \varphi_{\overline{r},\overline{x}'',\lambda}}{\partial x_{i}}dx=O(\frac{1}{\lambda^{\varepsilon}}).
$$
We thus get
$$
\aligned
c_{1}\sum_{j=1}^{m}& \int_{\mathbb{R}^{N}}\Big\lbrace\Big[(2_{\alpha}^{*}-1)\Big(|x|^{-(N-\alpha)}\ast |Y_{z_j,\lambda}|^{2_{\alpha}^{*}}\Big)Y_{z_j,\lambda}^{2_{\alpha}^{*}-2}Y_{j,1}+2_{\alpha}^{*}\Big(|x|^{-(N-\alpha)}\ast (Y_{z_j,\lambda}^{2_{\alpha}^{*}-1}
Y_{j,1})\Big)Y_{z_j,\lambda}^{2_{\alpha}^{*}-1}\Big]\frac{\partial \phi_{\overline{r},\overline{x}'',\lambda}}{\partial x_{i}} \\
+&\Big[(2_{\alpha}^{*}-1)\Big(|x|^{-(N-\alpha)}\ast |Z_{z_j,\lambda}|^{2_{\alpha}^{*}}\Big)Z_{z_j,\lambda}^{2_{\alpha}^{*}-2}Z_{j,1}+2_{\alpha}^{*}\Big(|x|^{-(N-\alpha)}\ast (Z_{z_j,\lambda}^{2_{\alpha}^{*}-1}
Z_{j,1})\Big)Z_{z_j,\lambda}^{2_{\alpha}^{*}-1}\Big]\frac{\partial \varphi_{\overline{r},\overline{x}'',\lambda}}{\partial x_{i}} \Big\rbrace dx\\
=&o(m|c_{1}|), \;i=3,\cdots, N.
\endaligned$$
The proof for
$$
\aligned
\sum_{l=2}^{N}c_{l}\sum_{j=1}^{m}& \int_{\mathbb{R}^{N}}\Big\lbrace\Big[(2_{\alpha}^{*}-1)\Big(|x|^{-(N-\alpha)}\ast |Y_{z_j,\lambda}|^{2_{\alpha}^{*}}\Big)Y_{z_j,\lambda}^{2_{\alpha}^{*}-2}Y_{j,l}+2_{\alpha}^{*}\Big(|x|^{-(N-\alpha)}\ast (Y_{z_j,\lambda}^{2_{\alpha}^{*}-1}
Y_{j,l})\Big)Y_{z_j,\lambda}^{2_{\alpha}^{*}-1}\Big]\frac{\partial \phi_{\overline{r},\overline{x}'',\lambda}}{\partial x_{i}} \\
+&\Big[(2_{\alpha}^{*}-1)\Big(|x|^{-(N-\alpha)}\ast |Z_{z_j,\lambda}|^{2_{\alpha}^{*}}\Big)Z_{z_j,\lambda}^{2_{\alpha}^{*}-2}Z_{j,l}+2_{\alpha}^{*}\Big(|x|^{-(N-\alpha)}\ast (Z_{z_j,\lambda}^{2_{\alpha}^{*}-1}
Z_{j,l})\Big)Z_{z_j,\lambda}^{2_{\alpha}^{*}-1}\Big]\frac{\partial \varphi_{\overline{r},\overline{x}'',\lambda}}{\partial x_{i}} \Big\rbrace dx\\
=&o(m\lambda^{2})\sum_{l=2}^{N}|c_{l}|, \;i=3,\cdots, N
\endaligned$$
are similar, we can obtain that
$$
\aligned
\sum_{l=1}^{N}c_{l}\sum_{j=1}^{m}& \int_{\mathbb{R}^{N}}\Big\lbrace\Big[(2_{\alpha}^{*}-1)\Big(|x|^{-(N-\alpha)}\ast |Y_{z_j,\lambda}|^{2_{\alpha}^{*}}\Big)Y_{z_j,\lambda}^{2_{\alpha}^{*}-2}Y_{j,l}+2_{\alpha}^{*}\Big(|x|^{-(N-\alpha)}\ast (Y_{z_j,\lambda}^{2_{\alpha}^{*}-1}
Y_{j,l})\Big)Y_{z_j,\lambda}^{2_{\alpha}^{*}-1}\Big]\frac{\partial \phi_{\overline{r},\overline{x}'',\lambda}}{\partial x_{i}} \\
+&\Big[(2_{\alpha}^{*}-1)\Big(|x|^{-(N-\alpha)}\ast |Z_{z_j,\lambda}|^{2_{\alpha}^{*}}\Big)Z_{z_j,\lambda}^{2_{\alpha}^{*}-2}Z_{j,l}+2_{\alpha}^{*}\Big(|x|^{-(N-\alpha)}\ast (Z_{z_j,\lambda}^{2_{\alpha}^{*}-1}
Z_{j,l})\Big)Z_{z_j,\lambda}^{2_{\alpha}^{*}-1}\Big]\frac{\partial \varphi_{\overline{r},\overline{x}'',\lambda}}{\partial x_{i}} \Big\rbrace dx\\
=&o(m|c_{1}|)+o(m\lambda^{2})\sum_{l=2}^{N}|c_{l}|.
\endaligned$$
Repeat the similar procedure as above, we have
$$
\aligned
\sum_{l=1}^{N}c_{l}\sum_{j=1}^{m}& \int_{\mathbb{R}^{N}}\Big\lbrace\Big[(2_{\alpha}^{*}-1)\Big(|x|^{-(N-\alpha)}\ast |Y_{z_j,\lambda}|^{2_{\alpha}^{*}}\Big)Y_{z_j,\lambda}^{2_{\alpha}^{*}-2}Y_{j,l}+2_{\alpha}^{*}\Big(|x|^{-(N-\alpha)}\ast (Y_{z_j,\lambda}^{2_{\alpha}^{*}-1}
Y_{j,l})\Big)Y_{z_j,\lambda}^{2_{\alpha}^{*}-1}\Big]\langle x,\nabla \phi_{\overline{r},\overline{x}'',\lambda} \rangle \\
+&\Big[(2_{\alpha}^{*}-1)\Big(|x|^{-(N-\alpha)}\ast |Z_{z_j,\lambda}|^{2_{\alpha}^{*}}\Big)Z_{z_j,\lambda}^{2_{\alpha}^{*}-2}Z_{j,l}+2_{\alpha}^{*}\Big(|x|^{-(N-\alpha)}\ast (Z_{z_j,\lambda}^{2_{\alpha}^{*}-1}
Z_{j,l})\Big)Z_{z_j,\lambda}^{2_{\alpha}^{*}-1}\Big]\langle x,\nabla \varphi_{\overline{r},\overline{x}'',\lambda} \rangle  \Big\rbrace dx\\
=&o(m|c_{1}|)+o(m\lambda^{2})\sum_{l=2}^{N}|c_{l}|.
\endaligned$$
Consequently,
\begin{equation}\label{Q1}
	\aligned
	\sum_{l=1}^{N}c_{l}\sum_{j=1}^{m}& \int_{\mathbb{R}^{N}}\Big\lbrace\Big[(2_{\alpha}^{*}-1)\Big(|x|^{-(N-\alpha)}\ast |Y_{z_j,\lambda}|^{2_{\alpha}^{*}}\Big)Y_{z_j,\lambda}^{2_{\alpha}^{*}-2}Y_{j,l}+2_{\alpha}^{*}\Big(|x|^{-(N-\alpha)}\ast (Y_{z_j,\lambda}^{2_{\alpha}^{*}-1}
	Y_{j,l})\Big)Y_{z_j,\lambda}^{2_{\alpha}^{*}-1}\Big]\langle x,\nabla Z_{\overline{r},\overline{x}'',\lambda} \rangle \\
	+&\Big[(2_{\alpha}^{*}-1)\Big(|x|^{-(N-\alpha)}\ast |Z_{z_j,\lambda}|^{2_{\alpha}^{*}}\Big)Z_{z_j,\lambda}^{2_{\alpha}^{*}-2}Z_{j,l}+2_{\alpha}^{*}\Big(|x|^{-(N-\alpha)}\ast (Z_{z_j,\lambda}^{2_{\alpha}^{*}-1}
	Z_{j,l})\Big)Z_{z_j,\lambda}^{2_{\alpha}^{*}-1}\Big]\langle x,\nabla Y_{\overline{r},\overline{x}'',\lambda} \rangle  \Big\rbrace dx\\
	=&o(m|c_{1}|)+o(m\lambda^{2})\sum_{l=2}^{N}|c_{l}|.
	\endaligned
\end{equation}
And
\begin{equation}\label{Q2}
\aligned
\sum_{l=1}^{N}c_{l}\sum_{j=1}^{m}& \int_{\mathbb{R}^{N}}\Big\lbrace\Big[(2_{\alpha}^{*}-1)\Big(|x|^{-(N-\alpha)}\ast |Y_{z_j,\lambda}|^{2_{\alpha}^{*}}\Big)Y_{z_j,\lambda}^{2_{\alpha}^{*}-2}Y_{j,l}+2_{\alpha}^{*}\Big(|x|^{-(N-\alpha)}\ast (Y_{z_j,\lambda}^{2_{\alpha}^{*}-1}
Y_{j,l})\Big)Y_{z_j,\lambda}^{2_{\alpha}^{*}-1}\Big]\frac{\partial Z_{\overline{r},\overline{x}'',\lambda}}{\partial x_{i}} \\
+&\Big[(2_{\alpha}^{*}-1)\Big(|x|^{-(N-\alpha)}\ast |Z_{z_j,\lambda}|^{2_{\alpha}^{*}}\Big)Z_{z_j,\lambda}^{2_{\alpha}^{*}-2}Z_{j,l}+2_{\alpha}^{*}\Big(|x|^{-(N-\alpha)}\ast (Z_{z_j,\lambda}^{2_{\alpha}^{*}-1}
Z_{j,l})\Big)Z_{z_j,\lambda}^{2_{\alpha}^{*}-1}\Big]\frac{\partial Y_{\overline{r},\overline{x}'',\lambda}}{\partial x_{i}} \Big\rbrace dx\\
=&o(m|c_{1}|)+o(m\lambda^{2})\sum_{l=2}^{N}|c_{l}|, \;i=3,\cdots, N.
\endaligned
\end{equation}
can be proved by \eqref{d4}.

Next, we will prove that
\begin{equation}\nonumber
	c_{i}=o(\frac{1}{\lambda^{2}})c_{1},\ i=2,\cdot\cdot\cdot, N.
\end{equation}
Notice that
$$
\langle x,\nabla Z_{\overline{r},\overline{x}'',\lambda}\rangle=\langle x',\nabla_{x'} Z_{\overline{r},\overline{x}'',\lambda}\rangle+\langle x'',\nabla_{x''} Z_{\overline{r},\overline{x}'',\lambda}\rangle,
$$
and
$$
\langle x,\nabla Y_{\overline{r},\overline{x}'',\lambda}\rangle=\langle x',\nabla_{x'} Y_{\overline{r},\overline{x}'',\lambda}\rangle+\langle x'',\nabla_{x''} Y_{\overline{r},\overline{x}'',\lambda}\rangle,
$$
we have
\begin{equation}\label{Q3}
\aligned
\sum_{l=1}^{N}c_{l}\sum_{j=1}^{m}& \int_{\mathbb{R}^{N}}\Big\lbrace\Big[(2_{\alpha}^{*}-1)\Big(|x|^{-(N-\alpha)}\ast |Y_{z_j,\lambda}|^{2_{\alpha}^{*}}\Big)Y_{z_j,\lambda}^{2_{\alpha}^{*}-2}Y_{j,l}+2_{\alpha}^{*}\Big(|x|^{-(N-\alpha)}\ast (Y_{z_j,\lambda}^{2_{\alpha}^{*}-1}
Y_{j,l})\Big)Y_{z_j,\lambda}^{2_{\alpha}^{*}-1}\Big]\langle x,\nabla Z_{\overline{r},\overline{x}'',\lambda} \rangle \\
+&\Big[(2_{\alpha}^{*}-1)\Big(|x|^{-(N-\alpha)}\ast |Z_{z_j,\lambda}|^{2_{\alpha}^{*}}\Big)Z_{z_j,\lambda}^{2_{\alpha}^{*}-2}Z_{j,l}+2_{\alpha}^{*}\Big(|x|^{-(N-\alpha)}\ast (Z_{z_j,\lambda}^{2_{\alpha}^{*}-1}
Z_{j,l})\Big)Z_{z_j,\lambda}^{2_{\alpha}^{*}-1}\Big]\langle x,\nabla Y_{\overline{r},\overline{x}'',\lambda} \rangle  \Big\rbrace dx\\
=c_{2}\sum_{j=1}^{m}& \int_{\mathbb{R}^{N}}\Big\lbrace\Big[(2_{\alpha}^{*}-1)\Big(|x|^{-(N-\alpha)}\ast |Y_{z_j,\lambda}|^{2_{\alpha}^{*}}\Big)Y_{z_j,\lambda}^{2_{\alpha}^{*}-2}Y_{j,2}+2_{\alpha}^{*}\Big(|x|^{-(N-\alpha)}\ast (Y_{z_j,\lambda}^{2_{\alpha}^{*}-1}
Y_{j,2})\Big)Y_{z_j,\lambda}^{2_{\alpha}^{*}-1}\Big]\langle x',\nabla_{x'} Z_{\overline{r},\overline{x}'',\lambda}\rangle \\
+&\Big[(2_{\alpha}^{*}-1)\Big(|x|^{-(N-\alpha)}\ast |Z_{z_j,\lambda}|^{2_{\alpha}^{*}}\Big)Z_{z_j,\lambda}^{2_{\alpha}^{*}-2}Z_{j,2}+2_{\alpha}^{*}\Big(|x|^{-(N-\alpha)}\ast (Z_{z_j,\lambda}^{2_{\alpha}^{*}-1}
Z_{j,2})\Big)Z_{z_j,\lambda}^{2_{\alpha}^{*}-1}\Big]\langle x',\nabla_{x'} Y_{\overline{r},\overline{x}'',\lambda}\rangle \Big\rbrace dx\\
+&o(m|c_{1}|)+o(m\lambda^{2})\sum_{l=3}^{N}|c_{l}|
\endaligned\end{equation}
and
\begin{equation}\label{Q4}
\aligned
\sum_{l=1}^{N}c_{l}\sum_{j=1}^{m}& \int_{\mathbb{R}^{N}}\Big\lbrace\Big[(2_{\alpha}^{*}-1)\Big(|x|^{-(N-\alpha)}\ast |Y_{z_j,\lambda}|^{2_{\alpha}^{*}}\Big)Y_{z_j,\lambda}^{2_{\alpha}^{*}-2}Y_{j,l}+2_{\alpha}^{*}\Big(|x|^{-(N-\alpha)}\ast (Y_{z_j,\lambda}^{2_{\alpha}^{*}-1}
Y_{j,l})\Big)Y_{z_j,\lambda}^{2_{\alpha}^{*}-1}\Big]\frac{\partial Z_{\overline{r},\overline{x}'',\lambda}}{\partial x_{i}} \\
+&\Big[(2_{\alpha}^{*}-1)\Big(|x|^{-(N-\alpha)}\ast |Z_{z_j,\lambda}|^{2_{\alpha}^{*}}\Big)Z_{z_j,\lambda}^{2_{\alpha}^{*}-2}Z_{j,l}+2_{\alpha}^{*}\Big(|x|^{-(N-\alpha)}\ast (Z_{z_j,\lambda}^{2_{\alpha}^{*}-1}
Z_{j,l})\Big)Z_{z_j,\lambda}^{2_{\alpha}^{*}-1}\Big]\frac{\partial Y_{\overline{r},\overline{x}'',\lambda}}{\partial x_{i}} \Big\rbrace dx\\
=c_{i}\sum_{j=1}^{m}& \int_{\mathbb{R}^{N}}\Big\lbrace\Big[(2_{\alpha}^{*}-1)\Big(|x|^{-(N-\alpha)}\ast |Y_{z_j,\lambda}|^{2_{\alpha}^{*}}\Big)Y_{z_j,\lambda}^{2_{\alpha}^{*}-2}Y_{j,i}+2_{\alpha}^{*}\Big(|x|^{-(N-\alpha)}\ast (Y_{z_j,\lambda}^{2_{\alpha}^{*}-1}
Y_{j,i})\Big)Y_{z_j,\lambda}^{2_{\alpha}^{*}-1}\Big]\frac{\partial Z_{\overline{r},\overline{x}'',\lambda}}{\partial x_{i}} \\
+&\Big[(2_{\alpha}^{*}-1)\Big(|x|^{-(N-\alpha)}\ast |Z_{z_j,\lambda}|^{2_{\alpha}^{*}}\Big)Z_{z_j,\lambda}^{2_{\alpha}^{*}-2}Z_{j,i}+2_{\alpha}^{*}\Big(|x|^{-(N-\alpha)}\ast (Z_{z_j,\lambda}^{2_{\alpha}^{*}-1}
Z_{j,i})\Big)Z_{z_j,\lambda}^{2_{\alpha}^{*}-1}\Big]\frac{\partial Y_{\overline{r},\overline{x}'',\lambda}}{\partial x_{i}} \Big\rbrace dx\\
+&o(m|c_{1}|)+o(m\lambda^{2})\sum_{l\neq1,i}^{N}|c_{l}|, \;i=3,\cdots, N.
\endaligned\end{equation}
Combining \eqref{Q1}, \eqref{Q2}, \eqref{Q3}, and \eqref{Q4},  we have
$$
\aligned
c_{2}\sum_{j=1}^{m}& \int_{\mathbb{R}^{N}}\Big\lbrace\Big[(2_{\alpha}^{*}-1)\Big(|x|^{-(N-\alpha)}\ast |Y_{z_j,\lambda}|^{2_{\alpha}^{*}}\Big)Y_{z_j,\lambda}^{2_{\alpha}^{*}-2}Y_{j,2}+2_{\alpha}^{*}\Big(|x|^{-(N-\alpha)}\ast (Y_{z_j,\lambda}^{2_{\alpha}^{*}-1}
Y_{j,2})\Big)Y_{z_j,\lambda}^{2_{\alpha}^{*}-1}\Big]\langle x',\nabla_{x'} Z_{\overline{r},\overline{x}'',\lambda}\rangle \\
+&\Big[(2_{\alpha}^{*}-1)\Big(|x|^{-(N-\alpha)}\ast |Z_{z_j,\lambda}|^{2_{\alpha}^{*}}\Big)Z_{z_j,\lambda}^{2_{\alpha}^{*}-2}Z_{j,2}+2_{\alpha}^{*}\Big(|x|^{-(N-\alpha)}\ast (Z_{z_j,\lambda}^{2_{\alpha}^{*}-1}
Z_{j,2})\Big)Z_{z_j,\lambda}^{2_{\alpha}^{*}-1}\Big]\langle x',\nabla_{x'} Y_{\overline{r},\overline{x}'',\lambda}\rangle \Big\rbrace dx\\
=&o(m|c_{1}|)+o(m\lambda^{2})\sum_{l\neq1,i}^{N}|c_{l}|,
\endaligned$$
and
$$
\aligned
c_{i}\sum_{j=1}^{m}& \int_{\mathbb{R}^{N}}\Big\lbrace\Big[(2_{\alpha}^{*}-1)\Big(|x|^{-(N-\alpha)}\ast |Y_{z_j,\lambda}|^{2_{\alpha}^{*}}\Big)Y_{z_j,\lambda}^{2_{\alpha}^{*}-2}Y_{j,i}+2_{\alpha}^{*}\Big(|x|^{-(N-\alpha)}\ast (Y_{z_j,\lambda}^{2_{\alpha}^{*}-1}
Y_{j,i})\Big)Y_{z_j,\lambda}^{2_{\alpha}^{*}-1}\Big]\frac{\partial Z_{\overline{r},\overline{x}'',\lambda}}{\partial x_{i}} \\
+&\Big[(2_{\alpha}^{*}-1)\Big(|x|^{-(N-\alpha)}\ast |Z_{z_j,\lambda}|^{2_{\alpha}^{*}}\Big)Z_{z_j,\lambda}^{2_{\alpha}^{*}-2}Z_{j,i}+2_{\alpha}^{*}\Big(|x|^{-(N-\alpha)}\ast (Z_{z_j,\lambda}^{2_{\alpha}^{*}-1}
Z_{j,i})\Big)Z_{z_j,\lambda}^{2_{\alpha}^{*}-1}\Big]\frac{\partial Y_{\overline{r},\overline{x}'',\lambda}}{\partial x_{i}} \Big\rbrace dx\\
+&o(m|c_{1}|)+o(m\lambda^{2})\sum_{l\neq1,i}^{N}|c_{l}|, \;i=3,\cdots, N.
\endaligned$$
Consider these together with \eqref{d5} and \eqref{d6}, we conclude that
\begin{equation}\nonumber
	c_{i}=o(\frac{1}{\lambda^{2}})c_{1},\ i=2,\cdot\cdot\cdot, N,
\end{equation}
and hence
$$
\aligned
0=\sum_{l=1}^{N}c_{l}\sum_{j=1}^{m}& \int_{\mathbb{R}^{N}}\Big\lbrace\Big[(2_{\alpha}^{*}-1)\Big(|x|^{-(N-\alpha)}\ast |Y_{z_j,\lambda}|^{2_{\alpha}^{*}}\Big)Y_{z_j,\lambda}^{2_{\alpha}^{*}-2}Y_{j,l}+2_{\alpha}^{*}\Big(|x|^{-(N-\alpha)}\ast (Y_{z_j,\lambda}^{2_{\alpha}^{*}-1}
Y_{j,l})\Big)Y_{z_j,\lambda}^{2_{\alpha}^{*}-1}\Big]\frac{\partial Z_{\overline{r},\overline{x}'',\lambda}}{\partial x_{i}} \\
+&\Big[(2_{\alpha}^{*}-1)\Big(|x|^{-(N-\alpha)}\ast |Z_{z_j,\lambda}|^{2_{\alpha}^{*}}\Big)Z_{z_j,\lambda}^{2_{\alpha}^{*}-2}Z_{j,l}+2_{\alpha}^{*}\Big(|x|^{-(N-\alpha)}\ast (Z_{z_j,\lambda}^{2_{\alpha}^{*}-1}
Z_{j,l})\Big)Z_{z_j,\lambda}^{2_{\alpha}^{*}-1}\Big]\frac{\partial Y_{\overline{r},\overline{x}'',\lambda}}{\partial x_{i}} \Big\rbrace dx\\
=&c_{1}\sum_{j=1}^{m} \int_{\mathbb{R}^{N}}\Big\lbrace\Big[(2_{\alpha}^{*}-1)\Big(|x|^{-(N-\alpha)}\ast |Y_{z_j,\lambda}|^{2_{\alpha}^{*}}\Big)Y_{z_j,\lambda}^{2_{\alpha}^{*}-2}Y_{j,1}+2_{\alpha}^{*}\Big(|x|^{-(N-\alpha)}\ast (Y_{z_j,\lambda}^{2_{\alpha}^{*}-1}
Y_{j,1})\Big)Y_{z_j,\lambda}^{2_{\alpha}^{*}-1}\Big]\frac{\partial Z_{\overline{r},\overline{x}'',\lambda}}{\partial x_{i}} \\
+&\Big[(2_{\alpha}^{*}-1)\Big(|x|^{-(N-\alpha)}\ast |Z_{z_j,\lambda}|^{2_{\alpha}^{*}}\Big)Z_{z_j,\lambda}^{2_{\alpha}^{*}-2}Z_{j,1}+2_{\alpha}^{*}\Big(|x|^{-(N-\alpha)}\ast (Z_{z_j,\lambda}^{2_{\alpha}^{*}-1}
Z_{j,1})\Big)Z_{z_j,\lambda}^{2_{\alpha}^{*}-1}\Big]\frac{\partial Y_{\overline{r},\overline{x}'',\lambda}}{\partial x_{i}} \Big\rbrace dx\\
+&o(\frac{m}{\lambda^{2}})c_{1}
=\frac{m}{\lambda^2}(a_{3}+o(1))c_{1}+o(\frac{m}{\lambda^{2}})c_{1},
\endaligned$$
where $a_3>0$, then $c_1=0$, and finally that $c_{i}=0$, $i=2,\cdots, N$.
\end{proof}

\begin{lem}\label{P3}
	There holds that
	$$
	\frac{\partial J(Z_{\overline{r},\overline{x}'',\lambda},Y_{\overline{r},\overline{x}'',\lambda})}{\partial \overline{x}''_{t}}=\frac{\partial J(Z_{\overline{r},\overline{x}'',\lambda}^{\ast},Y_{\overline{r},\overline{x}'',\lambda}^{\ast})}{\partial \overline{x}''_{t}}+O(\frac{m}{\lambda^{1+\varepsilon}}).
	$$
\end{lem}

\begin{proof}By direction calculation, we know
	$$\aligned
	&\frac{\partial J(Z_{\overline{r},\overline{x}'',\lambda}^{\ast},Y_{\overline{r},\overline{x}'',\lambda}^{\ast})}{\partial \overline{x}''_{t}}-\frac{\partial J(Z_{\overline{r},\overline{x}'',\lambda},Y_{\overline{r},\overline{x}'',\lambda})}{\partial \overline{x}''_{t}}\\
	&=\int_{\mathbb{R}^{N}}\Delta Z_{\overline{r},\overline{x}'',\lambda}^{\ast}\frac{\partial Y_{\overline{r},\overline{x}'',\lambda}^{\ast}}{\partial \overline{x}''_{t}}dx-\int_{\mathbb{R}^{N}}\Delta Z_{\overline{r},\overline{x}'',\lambda}\frac{\partial Y_{\overline{r},\overline{x}'',\lambda}}{\partial \overline{x}''_{t}}dx+\int_{\mathbb{R}^{N}}\Delta Y_{\overline{r},\overline{x}'',\lambda}^{\ast}\frac{\partial Z_{\overline{r},\overline{x}'',\lambda}^{\ast}}{\partial \overline{x}''_{t}}dx-\int_{\mathbb{R}^{N}}\Delta Y_{\overline{r},\overline{x}'',\lambda}\frac{\partial Z_{\overline{r},\overline{x}'',\lambda}}{\partial \overline{x}''_{t}}dx\\
	&+\int_{\mathbb{R}^N}\int_{\mathbb{R}^N}
	\frac{K_{1}(x)K_{1}(y)|Y_{\overline{r},\overline{x}'',\lambda}^{\ast}(x)|^{2_{\alpha}^{*}}
		(Y_{\overline{r},\overline{x}'',\lambda}^{\ast}(y))^{2_{\alpha}^{*}-1}\frac{\partial Y_{\overline{r},\overline{x}'',\lambda}^{\ast}}{\partial \overline{x}''_{t}}(y)}{|x-y|^{N-\alpha}}dxdy\\
	&-\int_{\mathbb{R}^N}\int_{\mathbb{R}^N}
	\frac{K_{1}(x)K_{1}(y)|Y_{\overline{r},\overline{x}'',\lambda}(x)|^{2_{\alpha}^{*}}
		Y_{\overline{r},\overline{x}'',\lambda}^{2_{\alpha}^{*}-1}(y)\frac{\partial Y_{\overline{r},\overline{x}'',\lambda}}{\partial \overline{x}''_{t}}(y)}{|x-y|^{N-\alpha}}dxdy\\
	&+\int_{\mathbb{R}^N}\int_{\mathbb{R}^N}
	\frac{K_{2}(x)K_{2}(y)|Z_{\overline{r},\overline{x}'',\lambda}^{\ast}(x)|^{2_{\alpha}^{*}}
		(Z_{\overline{r},\overline{x}'',\lambda}^{\ast}(y))^{2_{\alpha}^{*}-1}\frac{\partial Z_{\overline{r},\overline{x}'',\lambda}^{\ast}}{\partial \overline{x}''_{t}}(y)}{|x-y|^{N-\alpha}}dxdy\\
	&-\int_{\mathbb{R}^N}\int_{\mathbb{R}^N}
	\frac{K_{2}(x)K_{2}(y)|Z_{\overline{r},\overline{x}'',\lambda}(x)|^{2_{\alpha}^{*}}
		Z_{\overline{r},\overline{x}'',\lambda}^{2_{\alpha}^{*}-1}(y)\frac{\partial Z_{\overline{r},\overline{x}'',\lambda}}{\partial \overline{x}''_{t}}(y)}{|x-y|^{N-\alpha}}dxdy.
	\endaligned$$
We start by proving the following estimates
$$
\int_{\mathbb{R}^N}\int_{\mathbb{R}^N}
\frac{|Y_{\overline{r},\overline{x}'',\lambda}^{\ast}(x)|^{2_{\alpha}^{*}}
	(Y_{\overline{r},\overline{x}'',\lambda}^{\ast}(y))^{2_{\alpha}^{*}-1}\frac{\partial Y_{\overline{r},\overline{x}'',\lambda}^{\ast}}{\partial \overline{x}''_{t}}(y)}{|x-y|^{N-\alpha}}dxdy-\int_{\mathbb{R}^N}\int_{\mathbb{R}^N}
\frac{|Y_{\overline{r},\overline{x}'',\lambda}(x)|^{2_{\alpha}^{*}}
	Y_{\overline{r},\overline{x}'',\lambda}^{2_{\alpha}^{*}-1}(y)\frac{\partial Y_{\overline{r},\overline{x}'',\lambda}}{\partial \overline{x}''_{t}}(y)}{|x-y|^{N-\alpha}}dxdy=O(\frac{m}{\lambda^{1+\varepsilon}})
$$
and
$$
\int_{\mathbb{R}^N}\int_{\mathbb{R}^N}
\frac{|Z_{\overline{r},\overline{x}'',\lambda}^{\ast}(x)|^{2_{\alpha}^{*}}
	(Z_{\overline{r},\overline{x}'',\lambda}^{\ast}(y))^{2_{\alpha}^{*}-1}\frac{\partial Z_{\overline{r},\overline{x}'',\lambda}^{\ast}}{\partial \overline{x}''_{t}}(y)}{|x-y|^{N-\alpha}}dxdy-\int_{\mathbb{R}^N}\int_{\mathbb{R}^N}
\frac{|Z_{\overline{r},\overline{x}'',\lambda}(x)|^{2_{\alpha}^{*}}
	Z_{\overline{r},\overline{x}'',\lambda}^{2_{\alpha}^{*}-1}(y)\frac{\partial Z_{\overline{r},\overline{x}'',\lambda}}{\partial \overline{x}''_{t}}(y)}{|x-y|^{N-\alpha}}dxdy=O(\frac{m}{\lambda^{1+\varepsilon}}),
$$
where $t=3,\cdots,N.$ Here we only need to establish the last one, and the first one can done in the same way. It is obvious that
$$\aligned
&\int_{\mathbb{R}^N}\int_{\mathbb{R}^N}
\frac{|U_{z_j,\lambda}(x)|^{2_{\alpha}^{*}}
	U_{z_j,\lambda}^{2_{\alpha}^{*}-1}(x)\frac{\partial U_{z_j,\lambda}}{\partial \overline{x}''_{t}}(y)}{|x-y|^{N-\alpha}}dxdy-\int_{\mathbb{R}^N}\int_{\mathbb{R}^N}
\frac{|Z_{z_j,\lambda}(x)|^{2_{\alpha}^{*}}
	Z_{z_j,\lambda}^{2_{\alpha}^{*}-1}(y)\frac{\partial Z_{z_j,\lambda}}{\partial \overline{x}''_{t}}(y)}{|x-y|^{N-\alpha}}dxdy\\
=&\int_{\mathbb{R}^N}\int_{\mathbb{R}^N}
\frac{|U_{z_j,\lambda}(x)|^{2_{\alpha}^{*}}
	(1-\xi^{2_{\alpha}^{*}}(x))U_{z_j,\lambda}^{2_{\alpha}^{*}-1}(y)\frac{\partial U_{z_j,\lambda}}{\partial \overline{x}''_{t}}(y)}{|x-y|^{N-\alpha}}dxdy\\&+\int_{\mathbb{R}^N}\int_{\mathbb{R}^N}
\frac{(1-\xi^{2_{\alpha}^{*}})| U_{z_j,\lambda}(x)|^{2_{\alpha}^{*}}
	|\xi|^{2_{\alpha}^{*}} U_{z_j,\lambda}^{2_{\alpha}^{*}-1}(y)\frac{\partial U_{z_j,\lambda}}{\partial \overline{x}''_{t}}(y)}{|x-y|^{N-\alpha}}dxdy,
\endaligned$$
where $j=1,\dots, m$. By the Hardy-Littlewood-Sobolev inequality, we know
$$\aligned
\int_{\mathbb{R}^N}\int_{\mathbb{R}^N}
&\frac{|U_{z_j,\lambda}(x)|^{2_{\alpha}^{*}}
	(1-\xi^{2_{\alpha}^{*}}(x))U_{z_j,\lambda}^{2_{\alpha}^{*}-1}(y)\frac{\partial U_{z_j,\lambda}}{\partial \overline{x}''_{t}}(y)}{|x-y|^{N-\alpha}}dxdy\\
=&C\int_{\mathbb{R}^{N}}\int_{\mathbb{R}^{N}}\frac{\lambda^{\frac{N+\alpha}{2}}}{(1+\lambda^{2}|x-z_{j} |^{2})^{\frac{N+\alpha}{2}}}\frac{1}{|x-y|^{N-\alpha}}\frac{(1-\xi^{2_{\alpha}^{*}}(x))\lambda^{\frac{N+\alpha+4}{2}}(y_{t}-\overline{x}''_{t}) }
{(1+\lambda^{2}|y-z_{j} |^{2})^{\frac{N+\alpha+2}{2}}}dxdy\\
\leq &C\lambda \left(\int_{\mathbb{R}^{N}}\left[\frac{(1-\xi^{2_{{\alpha}^{*}}}(x+z_{j}))^{\frac{1}{2}}\lambda^{\frac{N+\alpha}{2}}}{(1+\lambda^{2}|x |^{2})^{\frac{N+\alpha}{2}}}\right]^{\frac{2N}{N+\alpha}}dx\right)^{\frac{N+\alpha}{2N}}
\left(\int_{\mathbb{R}^{N}}\left[\frac{(1-\xi^{2_{{\alpha}^{*}}}(x+z_{j}))^{\frac{1}{2}}\lambda^{\frac{N+\alpha}{2}}}{(1+\lambda^{2}|y|^{2})^{\frac{N+\alpha}{2}}}\right]^{\frac{2N}{N+\alpha}}dy\right)^{\frac{N+\alpha}{2N}}\\
=&O(\frac{1}{\lambda^{N+\alpha-1}}),
\endaligned$$
where we used the following fact
$$
\int_{\mathbb{R}^{N}}\left[\frac{(1-\xi^{2_{{\alpha}^{*}}}(x+z_{j}))^{\frac{1}{2}}\lambda^{\frac{N+\alpha}{2}}}{(1+\lambda^{2}|x|^{2})^{\frac{N+\alpha}{2}}}\right]^{\frac{2N}{N+\alpha}}dy=O(\frac{1}{\lambda^{N}}).
$$
Likewise
$$\aligned
\int_{\mathbb{R}^N}\int_{\mathbb{R}^N}
\frac{(1-\xi^{2_{\alpha}^{*}})| U_{z_j,\lambda}(x)|^{2_{\alpha}^{*}}
	|\xi|^{2_{\alpha}^{*}} U_{z_j,\lambda}^{2_{\alpha}^{*}-1}(y)\frac{\partial U_{z_j,\lambda}}{\partial \overline{x}''_{t}}(y)}{|x-y|^{N-\alpha}}dxdy=O(\frac{1}{\lambda^{N+\alpha-1}}).
\endaligned$$
Thus we have
$$\aligned
&\int_{\mathbb{R}^N}\int_{\mathbb{R}^N}
\frac{|U_{z_j,\lambda}(x)|^{2_{\alpha}^{*}}
	U_{z_j,\lambda}^{2_{\alpha}^{*}-1}(x)\frac{\partial U_{z_j,\lambda}}{\partial \overline{x}_{t}}(y)}{|x-y|^{N-\alpha}}dxdy-\int_{\mathbb{R}^N}\int_{\mathbb{R}^N}
\frac{|Z_{z_j,\lambda}(x)|^{2_{\alpha}^{*}}
	Z_{z_j,\lambda}^{2_{\alpha}^{*}-1}(y)\frac{\partial Z_{z_j,\lambda}}{\partial \overline{x}_{t}}(y)}{|x-y|^{N-\alpha}}dxdy=O(\frac{1}{\lambda^{N+\alpha-1}}),
\endaligned$$
where $j=1,2, \cdot\cdot\cdot,m$. Other cases are similarly available. Therefore
$$\aligned
\int_{\mathbb{R}^N}\int_{\mathbb{R}^N}
&\frac{|Z_{\overline{r},\overline{x}'',\lambda}^{\ast}(x)|^{2_{\alpha}^{*}}
	(Z_{\overline{r},\overline{x}'',\lambda}^{\ast}(y))^{2_{\alpha}^{*}-1}\frac{\partial Z_{\overline{r},\overline{x}'',\lambda}^{\ast}}{\partial \overline{x}_{t}}(y)}{|x-y|^{N-\alpha}}dxdy-\int_{\mathbb{R}^N}\int_{\mathbb{R}^N}
\frac{|Z_{\overline{r},\overline{x}'',\lambda}(x)|^{2_{\alpha}^{*}}
	Z_{\overline{r},\overline{x}'',\lambda}^{2_{\alpha}^{*}-1}(y)\frac{\partial Z_{\overline{r},\overline{x}'',\lambda}}{\partial \overline{x}_{t}}(y)}{|x-y|^{N-\alpha}}dxdy\\
=&O(\frac{m^{2\cdot2_{\alpha}^{*} }}{\lambda^{N+\alpha-1}})=O(\frac{m}{\lambda^{1+\varepsilon}}).
\endaligned$$

It is easy to check that
$$\aligned
\int_{\mathbb{R}^{N}}&\Delta Z_{\overline{r},\overline{x}'',\lambda}^{\ast}\frac{\partial Y_{\overline{r},\overline{x}'',\lambda}^{\ast}}{\partial \overline{x}''_{t}}dx-\int_{\mathbb{R}^{N}}\Delta Z_{\overline{r},\overline{x}'',\lambda}\frac{\partial Y_{\overline{r},\overline{x}'',\lambda}}{\partial \overline{x}''_{t}}dx\\
&=\int_{\mathbb{R}^{N}}\Delta Z_{\overline{r},\overline{x}'',\lambda}^{\ast}\frac{\partial Y_{\overline{r},\overline{x}'',\lambda}^{\ast}}{\partial \overline{x}''_{t}}dx-\int_{\mathbb{R}^{N}}\xi(\xi\Delta Z_{\overline{r},\overline{x}'',\lambda}^{\ast}+Z_{\overline{r},\overline{x}'',\lambda}^{\ast}\Delta\xi +2\nabla\xi\nabla Z_{\overline{r},\overline{x}'',\lambda}^{\ast})\frac{\partial Y_{\overline{r},\overline{x}'',\lambda}^{\ast}}{\partial \overline{x}''_{t}}dx.
\endaligned$$
Obviously,
$$\aligned
\int_{\mathbb{R}^{N}}(1-\xi^{2})\Delta Z_{\overline{r},\overline{x}'',\lambda}^{\ast}\frac{\partial Y_{\overline{r},\overline{x}'',\lambda}^{\ast}}{\partial \overline{x}''_{t}}dx&=\sum_{j=1}^{m}\int_{\mathbb{R}^N}\int_{\mathbb{R}^N}
\frac{(1-\xi^{2}(x))|V_{z_j,\lambda}(x)|^{2_{\alpha}^{*}}
	V_{z_j,\lambda}^{2_{\alpha}^{*}-1}(y)\frac{\partial Y_{\overline{r},\overline{x}'',\lambda}^{\ast}}{\partial \overline{x}''_{t}}(y)}{|x-y|^{N-\alpha}}dxdy\\
&=O(\frac{m^{2}}{\lambda^{N+\alpha-1
}})
\endaligned$$
can be obtained by the claim. A similar estimate can be drawn for
$$\aligned
\left| \int_{\mathbb{R}^{N}}\xi U_{z_j,\lambda}(y)\Delta\xi \frac{\partial V_{z_j,\lambda}(y)}{\partial \overline{x}''_{t}}dx\right| &\leq C\left| \int_{\mathbb{R}^{N}}\frac{\xi(y)\lambda^{\frac{N-2}{2}}}{(1+\lambda^{2}|y-z_{j} |^{2})^{\frac{N-2}{2}}}\frac{\lambda^{\frac{N+2}{2}}(y_{t}-\overline{x}''_{t})}
{(1+\lambda^{2}|y-z_{j} |^{2})^{\frac{N}{2}}}dy\right| \\
&\leq C\int_{\mathbb{R}^{N}\setminus B_{\delta-\vartheta} (0)}\frac{\lambda^{\frac{N-2}{2}}}{(1+\lambda^{2}|y |^{2})^{\frac{N-2}{2}}}\frac{\lambda^{\frac{N+2}{2}}|y|}{(1+\lambda^{2}|y|^{2})^{\frac{N}{2}}}dy
%&\leq C\int_{\mathbb{R}^{N}\setminus B_{\delta-\vartheta} (0)}\frac{\lambda^{\frac{N-2}{2}}}{\lambda^{N-2}|y |^{N-2}}\frac{\lambda^{\frac{N+2}{2}}|y|}{\lambda^{N}|y|^{N}}dy
=O(\frac{1}{\lambda^{N-3}}),
\endaligned$$
so
$$
\int_{\mathbb{R}^{N}}\xi Z_{\overline{r},\overline{x}'',\lambda}^{\ast}\Delta\xi \frac{\partial Y_{\overline{r},\overline{x}'',\lambda}^{\ast}}{\partial \overline{x}''_{t}}dx=O(\frac{m^{2}}{\lambda^{N-3}}).
$$
Combining with the following
$$
\int_{\mathbb{R}^{N}}\xi\nabla\xi\nabla Z_{\overline{r},\overline{x}'',\lambda}^{\ast}\frac{\partial Y_{\overline{r},\overline{x}'',\lambda}^{\ast}}{\partial \overline{x}''_{t}}dx%\leq C\sum_{j=1}^{m}\int_{\mathbb{R}^{N}}\frac{\xi|\nabla\xi|\lambda^{\frac{N+2}{2}}|y-z_{j} |}{(1+\lambda^{2}|y-z_{j} |^{2})^{\frac{N}{2}}}\frac{\partial Y_{\overline{r},\overline{x}'',\lambda}^{\ast}}{\partial \overline{x}''_{t}}dy
=O(\frac{m^{2}}{\lambda^{N-3}}),
$$
we conclude
$$
\int_{\mathbb{R}^{N}}\Delta Z_{\overline{r},\overline{x}'',\lambda}^{\ast}\frac{\partial Y_{\overline{r},\overline{x}'',\lambda}^{\ast}}{\partial \overline{x}''_{t}}dx-\int_{\mathbb{R}^{N}}\Delta Z_{\overline{r},\overline{x}'',\lambda}\frac{\partial Y_{\overline{r},\overline{x}'',\lambda}}{\partial \overline{x}''_{t}}dx=O(\frac{m^{2}}{\lambda^{N-3}})+O(\frac{m^{2}}{\lambda^{N+\alpha-1}})=O(\frac{m}{\lambda^{1+\varepsilon}}).
$$
Similarly,
$$
\int_{\mathbb{R}^{N}}\Delta Y_{\overline{r},\overline{x}'',\lambda}^{\ast}\frac{\partial Z_{\overline{r},\overline{x}'',\lambda}^{\ast}}{\partial \overline{x}''_{t}}dx-\int_{\mathbb{R}^{N}}\Delta Y_{\overline{r},\overline{x}'',\lambda}\frac{\partial Z_{\overline{r},\overline{x}'',\lambda}}{\partial \overline{x}''_{t}}dx=O(\frac{m^{2}}{\lambda^{N-3}})+O(\frac{m^{2}}{\lambda^{N+\alpha-1}})=O(\frac{m}{\lambda^{1+\varepsilon}}).
$$
Therefore, the proof is finished.
\end{proof}
\begin{lem}\label{D2}
	We have
	$$
	\frac{\partial J(Z_{\overline{r},\overline{x}'',\lambda},Y_{\overline{r},\overline{x}'',\lambda})}{\partial \overline{x}''_{t}}=m\left(B_{1}\frac{\partial K_{1}(\overline{r},\overline{x}'') }{\partial \overline{x}''_{t}}+B_{2}\frac{\partial K_{2}(\overline{r},\overline{x}'') }{\partial \overline{x}''_{t}} +O(\frac{1}{\lambda^{1+\varepsilon}})\right),
	$$
	where $B_1,B_2>0$ is constant.
\end{lem}
\begin{proof}
By calculation we have
	\begin{equation}\label{ExpenL4.4}
		\aligned
		&\frac{\partial J(Z_{\overline{r},\overline{x}'',\lambda}^{\ast},Y_{\overline{r},\overline{x}'',\lambda}^{\ast})}{\partial \overline{x}''_{t}}=- \int_{\mathbb{R}^{N}}\Delta Z_{\overline{r},\overline{x}'',\lambda}^{\ast}\frac{\partial Y_{\overline{r},\overline{x}'',\lambda}^{\ast}}{\partial \overline{x}''_{t}}dx-\int_{\mathbb{R}^{N}}\Delta Y_{\overline{r},\overline{x}'',\lambda}^{\ast}\frac{\partial Z_{\overline{r},\overline{x}'',\lambda}^{\ast}}{\partial \overline{x}''_{t}}dx\\
		&-\int_{\mathbb{R}^N}
		K_{1}(r,x'')\left( |x|^{-(N-\alpha)}*K_{1}(r,x'')|Y_{\overline{r},\overline{x}'',\lambda}^{\ast}|^{2_{\alpha}^{*}}\right)
			(Y_{\overline{r},\overline{x}'',\lambda}^{\ast})^{2_{\alpha}^{*}-1}\frac{\partial Y_{\overline{r},\overline{x}'',\lambda}^{\ast}}{\partial \overline{x}''_{t}}dx\\
		&-\int_{\mathbb{R}^N}
		K_{2}(r,x'')\left( |x|^{-(N-\alpha)}*K_{2}(r,x'')|Z_{\overline{r},\overline{x}'',\lambda}^{\ast}|^{2_{\alpha}^{*}}\right)
		(Z_{\overline{r},\overline{x}'',\lambda}^{\ast})^{2_{\alpha}^{*}-1}\frac{\partial Z_{\overline{r},\overline{x}'',\lambda}^{\ast}}{\partial \overline{x}''_{t}}dx \\
     	&=\Big[\sum_{j=1}^{m}\int_{\mathbb{R}^{N}}\left(|x|^{-(N-\alpha)}* |V_{z_j,\lambda}^{\ast}|^{2_{\alpha}}\right)V_{z_j,\lambda}^{2_{\alpha}^{*}-1}\frac{\partial Y_{\overline{r},\overline{x}'',\lambda}^{\ast}}{\partial \overline{x}''_{t}}dx-\int_{\mathbb{R}^N}
     	\left( |x|^{-(N-\alpha)}*|Y_{\overline{r},\overline{x}'',\lambda}^{\ast}|^{2_{\alpha}^{*}}\right)
     	(Y_{\overline{r},\overline{x}'',\lambda}^{\ast})^{2_{\alpha}^{*}-1}\frac{\partial Y_{\overline{r},\overline{x}'',\lambda}^{\ast}}{\partial \overline{x}''_{t}}dx\\
     	&+\sum_{j=1}^{m}\int_{\mathbb{R}^{N}}\left(|x|^{-(N-\alpha)}* |U_{z_j,\lambda}|^{2_{\alpha}}\right)U_{z_j,\lambda}^{2_{\alpha}^{*}-1}\frac{\partial Z_{\overline{r},\overline{x}'',\lambda}^{\ast}}{\partial \overline{x}''_{t}}dx-\int_{\mathbb{R}^N}
     	\left( |x|^{-(N-\alpha)}*|Z_{\overline{r},\overline{x}'',\lambda}^{\ast}|^{2_{\alpha}^{*}}\right)
     	(Z_{\overline{r},\overline{x}'',\lambda}^{\ast})^{2_{\alpha}^{*}-1}\frac{\partial Z_{\overline{r},\overline{x}'',\lambda}^{\ast}}{\partial \overline{x}''_{t}}dx\Big]\\
     	&+\frac{1}{22_{\alpha}^{*}}\Big[ \int_{\mathbb{R}^N}
     	\left(1- K_{1}(r,x'')\right)\frac{\partial }{\partial \overline{x}''_{t}} \left( \left( |x|^{-(N-\alpha)}*(K_{1}(r,x'')+1)|Y_{\overline{r},\overline{x}'',\lambda}^{\ast}|^{2_{\alpha}^{*}}\right)(Y_{\overline{r},\overline{x}'',\lambda}^{\ast})^{2_{\alpha}^{*}}\right) dx \\
     	&-\int_{\mathbb{R}^N}
     	\left( 1-K_{2}(r,x'')\right)\frac{\partial }{\partial \overline{x}''_{t}}\left( \left( |x|^{-(N-\alpha)}*(K_{2}(r,x'')+1)|Z_{\overline{r},\overline{x}'',\lambda}^{\ast}|^{2_{\alpha}^{*}}\right)
     	|Z_{\overline{r},\overline{x}'',\lambda}^{\ast})|^{2_{\alpha}^{*}}\right)  dx\Big]:=\cp_1+\cp_2.
		\endaligned
	\end{equation}
Then we will estimate $\cp_1$. We have
\begin{equation}\label{5.1}
	\aligned
	&\int_{\mathbb{R}^{N}}\Bigg[\Big(|x|^{-(N-\alpha)}\ast |Z_{\overline{r},\overline{x}'',\lambda}^{\ast}|^{2_{\alpha}^{*}}\Big)	 (Z_{\overline{r},\overline{x}'',\lambda}^{\ast})^{2_{\alpha}^{*}-1}-\sum_{j=1}^{m}\Big(|x|^{-(N-\alpha)}\ast |U_{z_j,\lambda}|^{2_{\alpha}^{*}}\Big)U_{z_j,\lambda}^{2_{\alpha}^{*}-1}\Bigg] \frac{\partial Z_{\overline{r},\overline{x}'',\lambda}^{\ast}}{\partial \overline{x}''_{t}}dx\\
	&=m\Bigg( \int_{\Omega_{1}}\Bigg[2_{\alpha}^{*}\Big(|x|^{-(N-\alpha)}\ast |U_{z_1,\lambda}^{2_{\alpha}^{*}-1}\sum_{i=2}^{m}U_{z_i,\lambda}|\Big)	 U_{z_1,\lambda}^{2_{\alpha}^{*}-1}+(2_{\alpha}^{*}-1)\Big(|x|^{-(N-\alpha)}\ast |U_{z_1,\lambda}^{2_{\alpha}^{*}}|\Big)U_{z_1,\lambda}^{2_{\alpha}^{*}-2}\sum_{i=2}^{m}U_{z_i,\lambda}\Bigg] \frac{\partial U_{z_1,\lambda}}{\partial \overline{x}''_{t}}dx\\
	&\quad+O(\frac{1}{\lambda^{1+\varepsilon}})\Bigg) .
	\endaligned\end{equation}	
	According to Lemma \ref{B6}, we deduce
	$$\aligned
	& \int_{\Omega_{1}}\Big(|x|^{-(N-\alpha)}\ast |U_{z_1,\lambda}|^{2_{\alpha}}\Big)U_{z_1,\lambda}^{2_{\alpha}^{*}-2}U_{z_i,\lambda}\frac{\partial U_{z_1,\lambda}}{\partial \overline{x}''_{t}}dx \\
	=&-C\int_{\Omega_{1}}\frac{\lambda^{\frac{N-\alpha}{2}}}{(1+\lambda^{2}|x-z_{1} |^{2})^{\frac{N-\alpha}{2}}}\frac{\lambda^{\frac{N-2}{2}}}{(1+\lambda^{2}|x-z_{i} |^{2})^{\frac{N-2}{2}}}\frac{\lambda^{\frac{6+\alpha}{2}}(x_{t}-\overline{x}''_{t})}{(1+\lambda^{2}|x-z_{1} |^{2})^{\frac{4+\alpha}{2}}}dx \\
	\leq&-C\int_{\Omega_{1}}\frac{\lambda}{(1+|x-\lambda z_{1} |)^{N+3}}\frac{1}{(1+|x-\lambda z_{i} |)^{N-2}}dx\\
	\leq&-C\int_{\Omega_{1}}\frac{\lambda}{(1+|x-\lambda z_{1} |)^{N+2}}\frac{1}{(1+|x-\lambda z_{i} |)^{N-1}}dx
	=\frac{-C}{\lambda^{N-2}|z_{1}-z_{i}|^{N-1}},
	\endaligned
	$$
	where $i=2,\dots,m$ and $C>0$. Thus
	\begin{equation}\label{5.2}
		\aligned
		\left|\sum_{i=2}^{m}\int_{\Omega_{1}}\Big(|x|^{-(N-\alpha)}\ast |U_{z_1,\lambda}|^{2_{\alpha}}\Big)U_{z_1,\lambda}^{2_{\alpha}^{*}-2}U_{z_i,\lambda}\frac{\partial U_{z_1,\lambda}}{\partial \overline{x}''_{t}}dx\right|
		\leq\sum_{i=2}^{m}\frac{-C}{\lambda^{N-2}|z_{1}-z_{i}|^{N-1}}=O(\frac{1}{\lambda^{1+\varepsilon}}).
\endaligned\end{equation}	
Moreover,

$$
\aligned
&\int_{\Omega_{1}}\Big(|x|^{-(N-\alpha)}\ast |U_{z_1,\lambda}^{2_{\alpha}^{*}-1}U_{z_i,\lambda}|\Big)	U_{z_1,\lambda}^{2_{\alpha}^{*}-1} \frac{\partial U_{z_1,\lambda}}{\partial \overline{x}''_{t}}dx\\
&=\frac{1}{2_{\alpha}^{*}}\Bigg[\frac{\partial}{\partial \overline{x}''_{t}}\int_{\Omega_{1}}\Big(|x|^{-(N-\alpha)}\ast |U_{z_1,\lambda}|^{2_{\alpha}^{*}}\Big)	 U_{z_1,\lambda}^{2_{\alpha}^{*}-1}(y)U_{z_i,\lambda}(y)dy-\int_{\Omega_{1}}\Big(|x|^{-(N-\alpha)}\ast |U_{z_1,\lambda}|^{2_{\alpha}^{*}}\Big)\frac{\partial U_{z_1,\lambda}^{2_{\alpha}^{*}-1}}{\partial \overline{x}''_{t}}U_{z_i,\lambda}(y)dy\\
&-\int_{\Omega_{1}}\Big(|x|^{-(N-\alpha)}\ast |U_{z_1,\lambda}|^{2_{\alpha}^{*}}\Big)U_{z_1,\lambda}^{2_{\alpha}^{*}-1} \frac{\partial U_{z_i,\lambda}}{\partial \overline{x}''_{t}}dy\Bigg]\\
&= C\frac{\partial}{\partial \overline{x}''_{t}}\int_{\Omega_{1}}\frac{\lambda^{\frac{N-\alpha}{2}}}{(1+\lambda^{2}|y-z_{1} |^{2})^{\frac{N-\alpha}{2}}}\frac{\lambda^{\frac{N-2}{2}}}{(1+\lambda^{2}|y-z_{i} |^{2})^{\frac{N-2}{2}}}\frac{\lambda^{\frac{2+\alpha}{2}}}{(1+\lambda^{2}|y-z_{1} |^{2})^{\frac{2+\alpha}{2}}}dy\\
&+C\int_{\Omega_{1}}\frac{\lambda^{\frac{N-\alpha}{2}}}{(1+\lambda^{2}|y-z_{1} |^{2})^{\frac{N-\alpha}{2}}}\frac{\lambda^{\frac{N-2}{2}}}{(1+\lambda^{2}|y-z_{i} |^{2})^{\frac{N-2}{2}}} \frac{\partial U_{z_i,\lambda}^{2_{\alpha}^{*}-1}}{\partial \overline{x}''_{t}} dy\\
&+C\int_{\Omega_{1}}\frac{\lambda^{\frac{N-\alpha}{2}}}{(1+\lambda^{2}|y-z_{1} |^{2})^{\frac{N-\alpha}{2}}} \frac{\partial U_{z_i,\lambda}}{\partial \overline{x}''_{t}}\frac{\lambda^{\frac{2+\alpha}{2}}}{(1+\lambda^{2}|y-z_{1} |^{2})^{\frac{2+\alpha}{2}}}dy:=T_1+T_2+T_3.\\
%&=\frac{C}{\lambda^{N-3}|z_{1}-z_{i}|^{N-2}}+\frac{C}{\lambda^{N-2}|z_{1}-z_{i}|^{N-1}}=O(\frac{1}{\lambda^{1+\varepsilon}}).
\endaligned
$$
Next, we estimate each of these three items separately.
$$\aligned
|T_1|&=\left| C\frac{\partial}{\partial \overline{x}''_{t}}\int_{\Omega_{1}}\frac{\lambda^{\frac{N-\alpha}{2}}}{(1+\lambda^{2}|y-z_{1} |^{2})^{\frac{N-\alpha}{2}}}\frac{\lambda^{\frac{N-2}{2}}}{(1+\lambda^{2}|y-z_{i} |^{2})^{\frac{N-2}{2}}}\frac{\lambda^{\frac{2+\alpha}{2}}}{(1+\lambda^{2}|y-z_{1} |^{2})^{\frac{2+\alpha}{2}}}dy\right| \\
&=\left| \frac{\partial}{\partial \overline{x}''_{t}}\frac{C}{\lambda^{N-2}|z_{1}-z_{i}|^{N-2}}\right| \leq \frac{C}{\lambda^{N-2}|z_{1}-z_{i}|^{N-1}}=O(\frac{1}{\lambda^{1+\varepsilon}}),
\endaligned$$
$$\aligned
T_2&= -C\int_{\Omega_{1}}\frac{\lambda^{\frac{N-\alpha}{2}}}{(1+\lambda^{2}|y-z_{1} |^{2})^{\frac{N-\alpha}{2}}}\frac{\lambda^{\frac{N-2}{2}}}{(1+\lambda^{2}|y-z_{i} |^{2})^{\frac{N-2}{2}}}\frac{\lambda^{\frac{6+\alpha}{2}}(y_{t}-\overline{y}''_{t})}{(1+\lambda^{2}|y-z_{1} |^{2})^{\frac{4+\alpha}{2}}}dy\\
%&\leq -C\int_{\Omega_{1}}\frac{\lambda^{\frac{N+6}{2}}(y_{t}-\overline{y}''_{t})}{(1+\lambda^{2}|y-z_{1} |^{2})^{\frac{N+4}{2}}}\frac{\lambda^{\frac{N-2}{2}}}{(1+\lambda^{2}|y-z_{i} |^{2})^{\frac{N-2}{2}}}dy\\
%&\leq -C\int_{\Omega_{1}}\frac{\lambda}{(1+|x-\lambda z_{1} |)^{N+2}}\frac{1}{(1+|x-\lambda z_{i} |)^{N-1}}dx
&=\frac{-C}{\lambda^{N-2}|z_{1}-z_{i}|^{N-1}}=O(\frac{1}{\lambda^{1+\varepsilon}}),
\endaligned$$
and
$$\aligned
|T_3|&=\left| C\int_{\Omega_{1}}\frac{\lambda^{\frac{N-\alpha}{2}}}{(1+\lambda^{2}|y-z_{1} |^{2})^{\frac{N-\alpha}{2}}}\frac{\lambda^{\frac{N+2}{2}}(y_{t}-\overline{y}''_{t})}{(1+\lambda^{2}|y-z_{i} |^{2})^{\frac{N}{2}}}\frac{\lambda^{\frac{2+\alpha}{2}}}{(1+\lambda^{2}|y-z_{1} |^{2})^{\frac{2+\alpha}{2}}}dy\right| \\
&\leq C\int_{\Omega_{1}}\frac{\lambda}{(1+|x-\lambda z_{1} |)^{N+1}}\frac{1}{(1+|x-\lambda z_{i} |)^{N-1}}dx=\frac{C}{\lambda^{N-2}|z_{1}-z_{i}|^{N-1}}=O(\frac{1}{\lambda^{1+\varepsilon}}).
\endaligned$$
Combining these, we have
$$
\aligned
\sum_{i=2}^{m}\int_{\Omega_{1}}\Big(|x|^{-(N-\alpha)}\ast |U_{z_1,\lambda}^{2_{\alpha}^{*}-1}U_{z_i,\lambda}|\Big)	U_{z_1,\lambda}^{2_{\alpha}^{*}-1} \frac{\partial U_{z_1,\lambda}}{\partial \overline{x}''_{t}}dx=O(\frac{1}{\lambda^{1+\varepsilon}}),
\endaligned
$$
and hence that
	\begin{equation}\label{5.12}
	\int_{\mathbb{R}^{N}}\Bigg[\Big(|x|^{-(N-\alpha)}\ast |Z_{\overline{r},\overline{x}'',\lambda}^{\ast}|^{2_{\alpha}^{*}}\Big)	 (Z_{\overline{r},\overline{x}'',\lambda}^{\ast})^{2_{\alpha}^{*}-1}-\sum_{j=1}^{m}\Big(|x|^{-(N-\alpha)}\ast |U_{z_j,\lambda}|^{2_{\alpha}^{*}}\Big)U_{z_j,\lambda}\Bigg] \frac{\partial Z_{\overline{r},\overline{x}'',\lambda}^{\ast}}{\partial \overline{x}''_{t}}dx=mO(\frac{1}{\lambda^{1+\varepsilon}}).
	\end{equation}
We can obtain
	\begin{equation}\label{5.121}
	\int_{\mathbb{R}^{N}}\Bigg[\Big(|x|^{-(N-\alpha)}\ast |Y_{\overline{r},\overline{x}'',\lambda}^{\ast}|^{2_{\alpha}^{*}}\Big)	 (Y_{\overline{r},\overline{x}'',\lambda}^{\ast})^{2_{\alpha}^{*}-1}-\sum_{j=1}^{m}\Big(|x|^{-(N-\alpha)}\ast |V_{z_j,\lambda}|^{2_{\alpha}^{*}}\Big)V_{z_j,\lambda}\Bigg] \frac{\partial Y_{\overline{r},\overline{x}'',\lambda}^{\ast}}{\partial \overline{x}''_{t}}dx=mO(\frac{1}{\lambda^{1+\varepsilon}}),
\end{equation}
by the same way.

 Next we will estimate $\cp_1$, according to the assumption the functions $K_1$ and $K_2$ are bounded, then
	\begin{align*}
	&\int_{\mathbb{R}^N}(1-K_{1}(r,x''))\frac{\partial}{\partial \overline{x}''_{t}}\left( \left( |x|^{-(N-\alpha)}*(1+K_{1}(r,x''))|Y_{\overline{r},\overline{x}'',\lambda}^{\ast}|^{2_{\alpha}^{*}}\right)
		|Y_{\overline{r},\overline{x}'',\lambda}^{\ast}|^{2_{\alpha}^{*}}\right) dx\\
		&=m\left(-C \int_{\Omega_{1}}(K_{1}(r,x'')-1)\frac{\partial}{\partial \overline{x}''_{t}}\left(\left( |x|^{-(N-\alpha)}*|V_{z_1,\lambda}|^{2_{\alpha}^{*}}\right)
		|V_{z_1,\lambda}|^{2_{\alpha}^{*}}\right) dx+O(\frac{1}{\lambda^{1+\varepsilon}})\right) \\
		&=m\Big(C \Big( \int_{\Omega_{1}}\frac{\partial \left( K_{1}(x+z_1)-K_{1}(z_1)\right) }{\partial \overline{x}''_{t}}\left( |x|^{-(N-\alpha)}*|V_{0,\lambda}|^{2_{\alpha}^{*}}\right)
		|V_{0,\lambda}|^{2_{\alpha}^{*}} dx\\
		&\quad+\int_{\Omega_{1}}\frac{\partial K_{1}(z_1) }{\partial \overline{x}''_{t}}\left( |x|^{-(N-\alpha)}*|V_{0,\lambda}|^{2_{\alpha}^{*}}\right)
		|V_{0,\lambda}|^{2_{\alpha}^{*}} dx\Big)
		+O(\frac{1}{\lambda^{1+\varepsilon}}) \Big)\\
		&=m\left(C\frac{\partial K_{1}(\overline{r},\overline{x}'') }{\partial \overline{x}''_{t}}\int_{\Omega_{1}}\left( |x|^{-(N-\alpha)}*|V_{0,1}|^{2_{\alpha}^{*}}\right)
		|V_{0,1}|^{2_{\alpha}^{*}} dx
		+O(\frac{1}{\lambda^{1+\varepsilon}}) \right)\\
		&=m\left(B_{1}\frac{\partial K_{1}(\overline{r},\overline{x}'') }{\partial \overline{x}''_{t}}+ O(\frac{1}{\lambda^{1+\varepsilon}})\right) .
	\end{align*}
Similarly,
\begin{align}\nonumber
	&\int_{\mathbb{R}^N}
	(1-K_{2}(r,x''))\left( |x|^{-(N-\alpha)}*(1+K_{2}(r,x''))|Z_{\overline{r},\overline{x}'',\lambda}^{\ast}|^{2_{\alpha}^{*}}\right)
	(Z_{\overline{r},\overline{x}'',\lambda}^{\ast})^{2_{\alpha}^{*}-1}\frac{\partial Z_{\overline{r},\overline{x}'',\lambda}^{\ast}}{\partial \overline{x}''_{t}}dx\\&
	=m\left(B_{2}\frac{\partial K_{2}(\overline{r},\overline{x}'') }{\partial \overline{x}''_{t}}+ O(\frac{1}{\lambda^{1+\varepsilon}})\right).
\end{align}
Thus
\begin{equation}\label{5.122}
\cp_1=m\left(B_{1}\frac{\partial K_{1}(\overline{r},\overline{x}'') }{\partial \overline{x}''_{t}}+B_{2}\frac{\partial K_{2}(\overline{r},\overline{x}'') }{\partial \overline{x}''_{t}} +O(\frac{1}{\lambda^{1+\varepsilon}})\right),
\end{equation}
where $B_{1}>0, B_{2}>0$ are constants. Combining \eqref{5.12}, \eqref{5.121} and \eqref{5.122},  we can prove the conclusion.
\end{proof}

	\begin{lem}\label{D3}
		We have
		$$\aligned
		\int_{\mathbb{R}^{N}}&\Big(-\Delta u_{m}
		-K_{1}(|x'|,x'')(|x|^{-(N-\alpha)}\ast K_{1}(|x'|,x'')|v_{m}|^{2_{\alpha}^{*}})v_{m}^{2_{\alpha}^{*}-1}\Big)\frac{\partial Y_{\overline{r},\overline{x}'',\lambda}}{\partial \overline{x}''_{t}} dx\\
		+\int_{\mathbb{R}^{N}}&\Big(-\Delta v_{m}-K_{2}(|x'|,x'')(|x|^{-(N-\alpha)}\ast K_{2}(|x'|,x'')|u_{m}|^{2_{\alpha}^{*}})u_{m}^{2_{\alpha}^{*}-1}\Big)\frac{\partial Z_{\overline{r},\overline{x}'',\lambda}}{\partial \overline{x}''_{t}} dx\\
		=&m\left(B_{1}\frac{\partial K_{1}(\overline{r},\overline{x}'') }{\partial \overline{x}''_{t}}+B_{2}\frac{\partial K_{2}(\overline{r},\overline{x}'') }{\partial \overline{x}''_{t}} +O(\frac{1}{\lambda^{1+\varepsilon}})\right).\endaligned$$
	\end{lem}

	\begin{proof}
	By direct calculation, we have
	\begin{align*}\nonumber
		&\int_{\mathbb{R}^{N}}\Big(-\Delta u_{m}
	-K_{1}(|x'|,x'')(|x|^{-(N-\alpha)}\ast K_{1}(|x'|,x'')|v_{m}|^{2_{\alpha}^{*}})v_{m}^{2_{\alpha}^{*}-1}\Big)\frac{\partial Y_{\overline{r},\overline{x}'',\lambda}}{\partial \overline{x}''_{t}} dx\\
	&+\int_{\mathbb{R}^{N}}\Big(-\Delta v_{m}-K_{2}(|x'|,x'')(|x|^{-(N-\alpha)}\ast K_{2}(|x'|,x'')|u_{m}|^{2_{\alpha}^{*}})u_{m}^{2_{\alpha}^{*}-1}\Big)\frac{\partial Z_{\overline{r},\overline{x}'',\lambda}}{\partial \overline{x}''_{t}} dx\\
	&=\langle J'(Z_{\overline{r},\overline{x}'',\lambda},Y_{\overline{r},\overline{x}'',\lambda}),(\frac{\partial Y_{\overline{r},\overline{x}'',\lambda}}{\partial \overline{x}''_{t}},\frac{\partial Z_{\overline{r},\overline{x}'',\lambda}}{\partial \overline{x}''_{t}})\rangle+m\Big\langle \mathcal{L}_m (\phi,\varphi),
	(\frac{\partial Y_{z_1,\lambda}}{\partial \overline{x}''_{t}},\frac{\partial Z_{z_1,\lambda}}{\partial \overline{x}''_{t}})\Big\rangle\\
	&-\int_{\mathbb{R}^{N}}K_{1}(|x'|,x'')(|x|^{-(N-\alpha)}\ast K_{1}(|x'|,x'')|v_{m}|^{2_{\alpha}^{*}})v_{m}^{2_{\alpha}^{*}-1}\frac{\partial Y_{\overline{r},\overline{x}'',\lambda}}{\partial \overline{x}''_{t}} dx\\
	&-\int_{\mathbb{R}^{N}}K_{1}(|x'|,x'')(|x|^{-(N-\alpha)}\ast K_{1}(|x'|,x'')|Y_{\overline{r},\overline{x}'',\lambda}|^{2_{\alpha}^{*}})Y_{\overline{r},\overline{x}'',\lambda}^{2_{\alpha}^{*}-1}\frac{\partial Y_{\overline{r},\overline{x}'',\lambda}}{\partial \overline{x}''_{t}} dx\\
	&-(2_{\alpha}^{*}-1)\int_{\mathbb{R}^{N}}K_{1}(|x'|,x'')(|x|^{-(N-\alpha)}\ast K_{1}(|x'|,x'')|Y_{\overline{r},\overline{x}'',\lambda}|^{2_{\alpha}^{*}})Y_{\overline{r},\overline{x}'',\lambda}^{2_{\alpha}^{*}-2}\varphi\frac{\partial Y_{\overline{r},\overline{x}'',\lambda}}{\partial \overline{x}''_{t}} dx \\
	&-2_{\alpha}^{*}\int_{\mathbb{R}^{N}}K_{1}(|x'|,x'')(|x|^{-(N-\alpha)}\ast K_{1}(|x'|,x'')Y_{\overline{r},\overline{x}'',\lambda}^{2_{\alpha}^{*}-1}\varphi)Y_{\overline{r},\overline{x}'',\lambda}^{2_{\alpha}^{*}-1}\frac{\partial Y_{\overline{r},\overline{x}'',\lambda}}{\partial \overline{x}''_{t}} dx\\
	&+\int_{\mathbb{R}^{N}}K_{2}(|x'|,x'')(|x|^{-(N-\alpha)}\ast K_{2}(|x'|,x'')|u_{m}|^{2_{\alpha}^{*}})u_{m}^{2_{\alpha}^{*}-1}\frac{\partial Z_{\overline{r},\overline{x}'',\lambda}}{\partial \overline{x}''_{t}} dx\\
	&-\int_{\mathbb{R}^{N}}K_{2}(|x'|,x'')(|x|^{-(N-\alpha)}\ast K_{2}(|x'|,x'')|Z_{\overline{r},\overline{x}'',\lambda}|^{2_{\alpha}^{*}})Z_{\overline{r},\overline{x}'',\lambda}^{2_{\alpha}^{*}-1}\frac{\partial Z_{\overline{r},\overline{x}'',\lambda}}{\partial \overline{x}''_{t}} dx\\
	&-(2_{\alpha}^{*}-1)\int_{\mathbb{R}^{N}}K_{2}(|x'|,x'')(|x|^{-(N-\alpha)}\ast K_{2}(|x'|,x'')|Z_{\overline{r},\overline{x}'',\lambda}|^{2_{\alpha}^{*}})Z_{\overline{r},\overline{x}'',\lambda}^{2_{\alpha}^{*}-2}\phi\frac{\partial Z_{\overline{r},\overline{x}'',\lambda}}{\partial \overline{x}''_{t}} dx \\
	&-2_{\alpha}^{*}\int_{\mathbb{R}^{N}}K_{2}(|x'|,x'')(|x|^{-(N-\alpha)}\ast K_{2}(|x'|,x'')Z_{\overline{r},\overline{x}'',\lambda}^{2_{\alpha}^{*}-1}\phi)Z_{\overline{r},\overline{x}'',\lambda}^{2_{\alpha}^{*}-1}\frac{\partial Z_{\overline{r},\overline{x}'',\lambda}}{\partial \overline{x}''_{t}} dx\\
	&:=\langle J'(Z_{\overline{r},\overline{x}'',\lambda},Y_{\overline{r},\overline{x}'',\lambda}),(\frac{\partial Y_{\overline{r},\overline{x}'',\lambda}}{\partial \overline{x}''_{t}},\frac{\partial Z_{\overline{r},\overline{x}'',\lambda}}{\partial \overline{x}''_{t}})\rangle+mI_{1}-I_{2}.
\end{align*}
	From Lemma \ref{c1}, we have
	$$
	I_{1}=O(\frac{\|(\phi,\varphi)\|_{\ast}}{\lambda^{\varepsilon}})
	=O(\frac{1}{\lambda^{1+\varepsilon}}).
	$$
	By the definition of $N_{1}(\varphi)$ and Lemma \ref{C4}, we get
	$$\aligned
	&\int_{\mathbb{R}^{N}}K_{1}(|x'|,x'')\Big(|x|^{-(N-\alpha)}\ast K_{1}(|x'|,x'')|v_{m}|^{2_{\alpha}^{*}}v_{m}^{2_{\alpha}^{*}-1}\Big)\frac{\partial Y_{\overline{r},\overline{x}'',\lambda}}{\partial \overline{x}''_{t}} dx\\
	&-\int_{\mathbb{R}^{N}}K_{1}(|x'|,x'')(|x|^{-(N-\alpha)}\ast K_{1}(|x'|,x'')|Y_{\overline{r},\overline{x}'',\lambda}|^{2_{\alpha}^{*}})Y_{\overline{r},\overline{x}'',\lambda}^{2_{\alpha}^{*}-1}\frac{\partial Y_{\overline{r},\overline{x}'',\lambda}}{\partial \overline{x}''_{t}} dx\\
	&-(2_{\alpha}^{*}-1)\int_{\mathbb{R}^{N}}K_{1}(|x'|,x'')(|x|^{-(N-\alpha)}\ast K_{1}(|x'|,x'')|Y_{\overline{r},\overline{x}'',\lambda}|^{2_{\alpha}^{*}})Y_{\overline{r},\overline{x}'',\lambda}^{2_{\alpha}^{*}-2}\varphi\frac{\partial Y_{\overline{r},\overline{x}'',\lambda}}{\partial \overline{x}''_{t}} dx \\
	&-2_{\alpha}^{*}\int_{\mathbb{R}^{N}}K_{1}(|x'|,x'')(|x|^{-(N-\alpha)}\ast K_{1}(|x'|,x'')Y_{\overline{r},\overline{x}'',\lambda}^{2_{\alpha}^{*}-1}\varphi)Y_{\overline{r},\overline{x}'',\lambda}^{2_{\alpha}^{*}-1}\frac{\partial Y_{\overline{r},\overline{x}'',\lambda}}{\partial \overline{x}''_{t}} dx\\
	&\leq C\int_{\mathbb{R}^{N}}|N_{1}(\varphi)||\frac{\partial Y_{\overline{r},\overline{x}'',\lambda}}{\partial \overline{x}''_{t}}| dx\\
	&\leq C	 \|\varphi\|_{\ast}^{2}\lambda\int_{\mathbb{R}^{N}}\sum_{j=1}^{m}\frac{\lambda^{\frac{N+2}{2}}}{(1+\lambda|x-z_{j}|)^{\frac{N+2}{2}+\tau}}\sum_{j=1}^{m}\frac{\lambda^{\frac{N-2}{2}}}{(1+\lambda|x-z_{j}|)^{N-1}}dx=O(\frac{m}{\lambda^{1+\varepsilon}}).
	\endaligned$$
The rest of $I_2$ run as before.

Finally, we  conclude by Lemma \ref{D2} that
	$$
	\aligned
	\langle J'(Z_{\overline{r},\overline{x}'',\lambda}+\phi,Y_{\overline{r},\overline{x}'',\lambda}+\varphi),(\frac{\partial Y_{\overline{r},\overline{x}'',\lambda}}{\partial \overline{x}''_{t}},\frac{\partial Z_{\overline{r},\overline{x}'',\lambda}}{\partial \overline{x}''_{t}})\rangle&=\langle J'(Z_{\overline{r},\overline{x}'',\lambda},Y_{\overline{r},\overline{x}'',\lambda}),(\frac{\partial Y_{\overline{r},\overline{x}'',\lambda}}{\partial \overline{x}''_{t}},\frac{\partial Z_{\overline{r},\overline{x}'',\lambda}}{\partial \overline{x}''_{t}})\rangle+O(\frac{m}{\lambda^{1+\varepsilon}})\\
	&=m\left(B_{1}\frac{\partial K_{1}(\overline{r},\overline{x}'') }{\partial \overline{x}''_{t}}+B_{2}\frac{\partial K_{2}(\overline{r},\overline{x}'') }{\partial \overline{x}''_{t}} +O(\frac{1}{\lambda^{1+\varepsilon}})\right).
	\endaligned
	$$
	
\end{proof}
Moreover, we have the following lemma.
\begin{lem}
	We have
	\begin{equation}\aligned\label{dr0}
		\Big\langle J'(Z_{\overline{r},\overline{x}'',\lambda}+\phi,Y_{\overline{r},\overline{x}'',\lambda}+\varphi),(\frac{\partial Z_{\overline{r},\overline{x}'',\lambda}}{\partial \lambda},\frac{\partial Y_{\overline{r},\overline{x}'',\lambda}}{\partial \lambda})\Big\rangle
		=&m\Big(-\frac{B_3}{\lambda^{3}}+\sum_{j=2}^{m}
		\frac{B_{4}}{\lambda^{N-1}|z_{1}-z_{j}|^{N-2}}+O(\frac{1}{\lambda^{3+\varepsilon}})\Big)\\
		=&m\Big(-\frac{B_3}{\lambda^{3}}+
		\frac{B_{4}m^{N-2}}{\lambda^{N-1}}+O(\frac{1}{\lambda^{3+\varepsilon}})\Big),
		\endaligned\end{equation}
	and
	\begin{equation}\aligned\label{dr1}
		&\Big\langle J'(Z_{\overline{r},\overline{x}'',\lambda}+\phi,Y_{\overline{r},\overline{x}'',\lambda}+\varphi),(\frac{\partial Z_{\overline{r},\overline{x}'',\lambda}}{\partial \overline{r}},\frac{\partial Y_{\overline{r},\overline{x}'',\lambda}}{\partial \overline{r}})\Big\rangle
		=m\Big(B_{1}\frac{\partial K_{1}(\overline{r},\overline{x}'') }{\partial \overline{r}}+B_{2}\frac{\partial K_{2}(\overline{r},\overline{x}'') }{\partial \overline{r}}+
		\frac{B_{5}m^{N-2}}{\lambda^{N-2}}+O(\frac{1}{\lambda^{1+\varepsilon}})\Big),
		\endaligned\end{equation}
	where $B_i>0$, $i=1, 2,3,$ are constants.
\end{lem}
\begin{proof}
	The proof for \eqref{dr1} is exactly similar to Lemma \ref{D2}. For \eqref{dr0}, we only need to give the following proof  inspired by \cite{PWW}. Since the function $K_1$ and $K_2$ are bound and satisfy the condition $(\textbf{II})$, we have
		\begin{multline*}
		\int_{\mathbb{R}^N}(1-K_{1}(r,x''))\frac{\partial}{\partial \lambda}\left( \left( |x|^{-(N-\alpha)}*(1+K_{1}(r,x''))|Y_{\overline{r},\overline{x}'',\lambda}^{\ast}|^{2_{\alpha}^{*}}\right)
		(Y_{\overline{r},\overline{x}'',\lambda}^{\ast})^{2_{\alpha}^{*}}\right) dx\\
		=m\left(-C \int_{\Omega_{1}}(K_{1}(r,x'')-1)\frac{\partial}{\partial \lambda}\left(\left( |x|^{-(N-\alpha)}*|V_{z_1,\lambda}|^{2_{\alpha}^{*}}\right)
		|V_{z_1,\lambda}|^{2_{\alpha}^{*}}\right) dx+O(\frac{1}{\lambda^{1+\varepsilon}})\right)\\ 	
		=m\Bigg(-C \int_{B_{\lambda^{\frac{1}{2}+\varepsilon}}(x_{0})}\Big(\sum_{i,j=1}^{m}\frac{1}{2}\frac{\partial^{2}K_{1}(x_0)}{\partial x_{i}\partial x_{j}}(x_i-x_{0i})(x_j-x_{0j})\\
		+\sum_{i,j,k=1}^{m}\frac{1}{6}\frac{\partial^{3}K_{1}(x_0)}{\partial x_{i}\partial x_{j}\partial x_{k}}(x_i-x_{0i})(x_j-x_{0j})(x_k-x_{0k})\Big)\frac{1}{2_{\alpha}^{*}}\frac{\partial}{\partial \lambda}\left(\left( |x|^{-(N-\alpha)}*|V_{z_1,\lambda}|^{2_{\alpha}^{*}}\right)
		|V_{z_1,\lambda}|^{2_{\alpha}^{*}}\right)\\
		-C \int_{B_{\lambda^{\frac{1}{2}+\varepsilon}}^{C}(x_{0})}(K_{1}(r,x'')-1)\frac{1}{2_{\alpha}^{*}}\frac{\partial}{\partial \lambda}\left(\left( |x|^{-(N-\alpha)}*|V_{z_1,\lambda}|^{2_{\alpha}^{*}}\right)
		|V_{z_1,\lambda}|^{2_{\alpha}^{*}}\right) dx+O(\frac{1}{\lambda^{3+\varepsilon}})\Bigg)\\
		=m\Bigg(-C \int_{B_{\lambda^{\frac{3}{2}+\varepsilon}}(z_{1})}\Big(\sum_{i,j=1}^{m}\frac{1}{2}\frac{\partial^{2}K_{1}(x_0)}{\partial x_{i}\partial x_{j}}(x_i-x_{0i})(x_j-x_{0j})
		+O\left( \|x-x_{0}\|^{3}\right) \Big)\frac{1}{2_{\alpha}^{*}}\frac{\partial}{\partial \lambda}\left(\left( |x|^{-(N-\alpha)}*|V_{z_1,\lambda}|^{2_{\alpha}^{*}}\right)
		|V_{z_1,\lambda}|^{2_{\alpha}^{*}}\right)\\
		+O \Big(\int_{B_{\lambda^{\frac{1}{2}+\varepsilon}}^{C}(x_{0})}\frac{1}{\lambda}\left(\left( |x|^{-(N-\alpha)}*|V_{z_1,\lambda}|^{2_{\alpha}^{*}}\right)
		|V_{z_1,\lambda}|^{2_{\alpha}^{*}}\right) dx\Big)+O(\frac{1}{\lambda^{3+\varepsilon}})\Bigg)\\
			=m\Bigg(-C \int_{B_{\lambda^{\frac{3}{2}+\varepsilon}}(z_{1})}\Big(\sum_{i,j=1}^{m}\frac{1}{2}\frac{\partial^{2}K_{1}(x_0)}{\partial x_{i}\partial x_{j}}(x_i-x_{0i})(x_j-x_{0j})
		+O\left( \|x-x_{0}\|^{3}\right)\Big)\frac{1}{2_{\alpha}^{*}}\frac{\partial}{\partial \lambda}\left(\left( |x|^{-(N-\alpha)}*|V_{z_1,\lambda}|^{2_{\alpha}^{*}}\right)
		|V_{z_1,\lambda}|^{2_{\alpha}^{*}}\right)\\+O(\frac{1}{\lambda^{3+\varepsilon}})\Bigg)\\
			=m\Bigg(-\frac{C}{2_{\alpha}^{*}}\frac{\partial}{\partial \lambda} \int\Big(\sum_{i,j=1}^{m}\frac{1}{2}\frac{\partial^{2}K_{1}(x_0)}{\partial x_{i}\partial x_{j}}(\frac{y_{i}}{\lambda}+z_{1i}-x_{0i})(\frac{y_{j}}{\lambda}+z_{1j}-x_{0j})\Big)\left(\left( |y|^{-(N-\alpha)}*|V_{0,1}|^{2_{\alpha}^{*}}\right)
		|V_{0,1}|^{2_{\alpha}^{*}}\right)\\
		-\frac{C}{2_{\alpha}^{*}}\frac{\partial}{\partial \lambda} \int_{B_{\lambda^{\frac{3}{2}+\varepsilon}}^{C}(0)}\Big(\sum_{i,j=1}^{m}\frac{1}{2}\frac{\partial^{2}K_{1}(x_0)}{\partial x_{i}\partial x_{j}}(\frac{y_{i}}{\lambda}+z_{1i}-x_{0i})(\frac{y_{j}}{\lambda}+z_{1j}-x_{0j})\Big)\left(\left( |y|^{-(N-\alpha)}*|V_{0,1}|^{2_{\alpha}^{*}}\right)
		|V_{0,1}|^{2_{\alpha}^{*}}\right)\\
		+\frac{C}{2_{\alpha}^{*}}\frac{1}{\lambda} \int O\left( \|\frac{y}{\lambda}+z_{1}-x_{0}\|^{3}\right) \left(\left( |y|^{-(N-\alpha)}*|V_{0,1}|^{2_{\alpha}^{*}}\right)
		|V_{0,1}|^{2_{\alpha}^{*}}\right)+O(\frac{1}{\lambda^{3+\varepsilon}})\Bigg)\\
		=m\Bigg(-\frac{C}{2_{\alpha}^{*}}\frac{\partial}{\partial \lambda} \int\Big(\sum_{i,j=1}^{m}\frac{1}{2}\frac{\partial^{2}K_{1}(x_0)}{\partial x_{i}^{2}}(\frac{y_{i}}{\lambda}+z_{1i}-x_{0i})^{2}\Big)\left(\left( |y|^{-(N-\alpha)}*|V_{0,1}|^{2_{\alpha}^{*}}\right)
		|V_{0,1}|^{2_{\alpha}^{*}}\right)+O(\frac{1}{\lambda^{3+\varepsilon}})\Bigg)\\
			=m\Bigg(-\frac{C}{2_{\alpha}^{*}}\frac{1}{\lambda^3}\frac{\Delta K_{1}(x_0)}{N} \int y^{2}\left(\left( |y|^{-(N-\alpha)}*|V_{0,1}|^{2_{\alpha}^{*}}\right)
		|V_{0,1}|^{2_{\alpha}^{*}}\right)+O(\frac{1}{\lambda^{3+\varepsilon}})\Bigg)
		=m\Big(-\frac{B_3}{\lambda^{3}}+
		O(\frac{1}{\lambda^{3+\varepsilon}})\Big).
	\end{multline*}
	The proof is completed.
\end{proof}

In order to prove Theorem \ref{EXS1}, let us first show the following estimate.
\begin{lem}\label{D4}
We have
	\begin{equation}\label{d20}
		\int\nabla \phi\cdot\nabla \varphi dx+\int\left( \int_{\mathbb{R}^{N}}\frac{|\varphi|^{2_{\alpha}^{*}}}{|x-y|^{N-\alpha}}\right) |\varphi|^{2_{\alpha}^{*}}dx+\int\left( \int_{\mathbb{R}^{N}}\frac{|\phi|^{2_{\alpha}^{*}}}{|x-y|^{N-\alpha}}\right) |\phi|^{2_{\alpha}^{*}}dx=O(\frac{m}{\lambda^{2+\varepsilon}}).
	\end{equation}
\end{lem}
\begin{proof}
	As \eqref{c14}, we have
	$$
	\aligned
	\int_{\mathbb{R}^{N}}&\nabla \phi\cdot\nabla \varphi dx=\int_{\mathbb{R}^{N}}\left[ K_{1}(x)(|x|^{-(N-\alpha)}\ast K_{1}(x)|Y_{\overline{r},\overline{x}'',\lambda}+\varphi|^{2_{\alpha}^{*}})(Y_{\overline{r},\overline{x}'',\lambda}+\varphi)^{2_{\alpha}^{*}-1}+\Delta Z_{\overline{r},\overline{x}'',\lambda}\right]\phi\\
	&+\int_{\mathbb{R}^{N}}\left[ K_{2}(x)(|x|^{-(N-\alpha)}\ast K_{2}(x)|Z_{\overline{r},\overline{x}'',\lambda}+\phi|^{2_{\alpha}^{*}})(Z_{\overline{r},\overline{x}'',\lambda}+\phi)^{2_{\alpha}^{*}-1}+\Delta Y_{\overline{r},\overline{x}'',\lambda}\right]\varphi\\
	&=\int_{\mathbb{R}^{N}}\Big[ K_{1}(x)(|x|^{-(N-\alpha)}\ast K_{1}(x)|Y_{\overline{r},\overline{x}'',\lambda}+\varphi|^{2_{\alpha}^{*}})(Y_{\overline{r},\overline{x}'',\lambda}+\varphi)^{2_{\alpha}^{*}-1}\\
	&-K_{1}(x)(|x|^{-(N-\alpha)}\ast K_{1}(x)|Y_{\overline{r},\overline{x}'',\lambda}|^{2_{\alpha}^{*}})(Y_{\overline{r},\overline{x}'',\lambda})^{2_{\alpha}^{*}-1}\Big]\phi\\
	&+\int_{\mathbb{R}^{N}}\Big[ K_{1}(x)(|x|^{-(N-\alpha)}\ast K_{1}(x)|Y_{\overline{r},\overline{x}'',\lambda}|^{2_{\alpha}^{*}})(Y_{\overline{r},\overline{x}'',\lambda})^{2_{\alpha}^{*}-1}+\Delta Z_{\overline{r},\overline{x}'',\lambda}\Big]\phi\\
	&+\int_{\mathbb{R}^{N}}\Big[ K_{2}(x)(|x|^{-(N-\alpha)}\ast K_{2}(x)|Z_{\overline{r},\overline{x}'',\lambda}+\phi|^{2_{\alpha}^{*}})(Z_{\overline{r},\overline{x}'',\lambda}+\phi)^{2_{\alpha}^{*}-1}\\
	&-K_{2}(x)(|x|^{-(N-\alpha)}\ast K_{2}(x)|Z_{\overline{r},\overline{x}'',\lambda}|^{2_{\alpha}^{*}})(Z_{\overline{r},\overline{x}'',\lambda})^{2_{\alpha}^{*}-1}\Big]\varphi\\
	&+\int_{\mathbb{R}^{N}}\Big[ K_{2}(x)(|x|^{-(N-\alpha)}\ast K_{2}(x)|Z_{\overline{r},\overline{x}'',\lambda}|^{2_{\alpha}^{*}})(Z_{\overline{r},\overline{x}'',\lambda})^{2_{\alpha}^{*}-1}+\Delta Y_{\overline{r},\overline{x}'',\lambda}\Big]\varphi.
	\endaligned
	$$
	Let's denote the first three terms of the above by
	$$
	\aligned
	&\int_{\mathbb{R}^{N}}\Big[ K_{1}(x)(|x|^{-(N-\alpha)}\ast K_{1}(x)|Y_{\overline{r},\overline{x}'',\lambda}+\varphi|^{2_{\alpha}^{*}})(Y_{\overline{r},\overline{x}'',\lambda}+\varphi)^{2_{\alpha}^{*}-1}-K_{1}(x)(|x|^{-(N-\alpha)}\ast K_{1}(x)|Y_{\overline{r},\overline{x}'',\lambda}|^{2_{\alpha}^{*}})(Y_{\overline{r},\overline{x}'',\lambda})^{2_{\alpha}^{*}-1}\Big]\phi\\
	&+\int_{\mathbb{R}^{N}}\Big[ K_{1}(x)(|x|^{-(N-\alpha)}\ast K_{1}(x)|Y_{\overline{r},\overline{x}'',\lambda}|^{2_{\alpha}^{*}})(Y_{\overline{r},\overline{x}'',\lambda})^{2_{\alpha}^{*}-1}+\Delta Z_{\overline{r},\overline{x}'',\lambda}\Big]\phi:=I_{1}+I_{2},
	\endaligned
	$$
the rest terms can be estimated in a same way. By Lemma \ref{C4}, we have the following estimate
		$$
	\aligned
	\left| I_{1}\right|=&\int_{\mathbb{R}^{N}}\left|N_{1}(\varphi)\right|\phi+ \int_{\mathbb{R}^{N}}\Big[K_{1}(r,x^{''})(2_{\alpha}^{*}-1)\Big(|x|^{-(N-\alpha)}\ast K_{1}(r,x^{''})|Y_{\overline{r},\overline{x}'',\lambda}|^{2_{\alpha}^{*}}\Big)Y_{\overline{r},\overline{x}'',\lambda}^{2_{\alpha}^{*}-2}\varphi\\
	&+K_{1}(r,x^{''})2_{\alpha}^{*}\Big(|x|^{-(N-\alpha)}\ast K_{1}(r,x^{''}) (Y_{\overline{r},\overline{x}'',\lambda}^{2_{\alpha}^{*}-1}
	\varphi)\Big)Y_{\overline{r},\overline{x}'',\lambda}^{2_{\alpha}^{*}-1}\Big]\phi\\
	&\leq C\|\varphi\|_{\ast}^{2}\|\phi\|_{\ast}\int_{\mathbb{R}^{N}}\sum_{j=1}^{m}\frac{\lambda^{\frac{N+2}{2}}}{(1+\lambda|x-z_{j}|)^{\frac{N+2}{2}+\tau}}\sum_{j=1}^{m}\frac{\lambda^{\frac{N-2}{2}}}{(1+\lambda|x-z_{j}|)^{\frac{N-2}{2}+\tau}}dx\leq \frac{Cm}{\lambda^{3+3\varepsilon}}=O(\frac{m}{\lambda^{2+\varepsilon}}).
		\endaligned
	$$
	To estimate $I_2$, one has from Lemma \ref{C5} that
		$$
	\aligned
	I_{2}&\leq C\frac{\|\phi\|_{\ast}}{\lambda^{1+\varepsilon}}\int_{\mathbb{R}^{N}}\sum_{j=1}^{m}\frac{\lambda^{\frac{N+2}{2}}}{(1+\lambda|x-z_{j}|)^{\frac{N+2}{2}+\tau}}\sum_{j=1}^{m}\frac{\lambda^{\frac{N-2}{2}}}{(1+\lambda|x-z_{j}|)^{\frac{N-2}{2}+\tau}}dx\leq \frac{Cm}{\lambda^{2+2\varepsilon}}=O(\frac{m}{\lambda^{2+\varepsilon}}).
	\endaligned
	$$
	 Hence,
	$$
	\int_{\mathbb{R}^{N}}\nabla \phi\cdot\nabla \varphi dx=O(\frac{m}{\lambda^{2+\varepsilon}}).
$$
Similarly, we have
$$
\int_{\mathbb{R}^{N}}\left( \int_{\mathbb{R}^{N}}\frac{|\varphi|^{2_{\alpha}^{*}}}{|x-y|^{(N-\alpha)}}\right) |\varphi|^{2_{\alpha}^{*}}dx
\leq C\left( \int_{\mathbb{R}^{N}}|\nabla \varphi|^{2}dx\right)^{2_{\alpha}^{*}}\leq O(\frac{m}{\lambda^{2+\varepsilon}}),
$$
and
$$
\int_{\mathbb{R}^{N}}\left( \int_{\mathbb{R}^{N}}\frac{|\phi|^{2_{\alpha}^{*}}}{|x-y|^{(N-\alpha)}}\right) |\phi|^{2_{\alpha}^{*}}dx\leq C\left( \int_{\mathbb{R}^{N}}|\nabla \phi|^{2}dx\right)^{2_{\alpha}^{*}}\leq O(\frac{m}{\lambda^{2+\varepsilon}}).
$$
Combine all these together, we complete the proof.
\end{proof}

\noindent
{\bf Proof of Theorem \ref{EXS1}.}
The proof will be divided into four steps.

\textbf{Step 1:} Since $u_m=\phi$ and $v_m=\varphi$ on $\partial D_{\rho}$, it is easy to check that \eqref{d1} is equivalent to
	\begin{equation}\label{d12}
		\begin{split}
			&-(N-2)\int_{D_{\rho}}\nabla u_{m}\cdot\nabla v_{m}dx
			+\frac{1}{2_{\alpha}^{*}}\int_{D_{\rho}}\langle x,\nabla K_{1}(x)\rangle\int_{\R^N}\frac{|v_{m}(y)|^{2_{\alpha}^{*}}}{|x-y|^{(N-\alpha)}}dyv^{2_{\alpha}^{*}}_{m}(x)dx\\
			 &+\frac{N}{2_{\alpha}^{*}}\int_{D_{\rho}}\int_{\R^N}\frac{K_{1}(x)K_{1}(y)|v_{m}(y)|^{2_{\alpha}^{*}}v^{2_{\alpha}^{*}}_{m}(x)}{|x-y|^{(N-\alpha)}}dxdy-\frac{(N-\alpha)}{2_{\alpha}^{*}}\int_{D_{\rho}}\int_{\R^N}x(x-y)\frac{K_{1}(x)K_{1}(y)|v_{m}(y)|^{2_{\alpha}^{*}}v^{2_{\alpha}^{*}}_{m}(x)}{|x-y|^{(N-\alpha+2)}}dxdy
			\\
			&+\frac{1}{2_{\alpha}^{*}}\int_{D_{\rho}}\langle x,\nabla K_{2}(x)\rangle\int_{\R^N}\frac{|u_{m}(y)|^{2_{\alpha}^{*}}u^{2_{\alpha}^{*}}_{m}(x)}{|x-y|^{(N-\alpha)}}dxdy+\frac{N}{2_{\alpha}^{*}}\int_{D_{\rho}}\int_{\R^N}\frac{K_{2}(x)K_{2}(y)|u_{m}(y)|^{2_{\alpha}^{*}}u^{2_{\alpha}^{*}}_{m}(x)}{|x-y|^{(N-\alpha)}}dxdy\\
			&-\frac{(N-\alpha)}{2_{\alpha}^{*}}\int_{D_{\rho}}\int_{\R^N}x(x-y)\frac{K_{2}(x)K_{2}(y)|u_{m}(y)|^{2_{\alpha}^{*}}u^{2_{\alpha}^{*}}_{m}(x)}{|x-y|^{(N-\alpha+2)}}dxdy\\
			&=O\Big(\int_{\partial D_{\rho}}\nabla \phi\nabla \varphi ds
			+\int_{\partial D_{\rho}}\Big(\int_{\R^N} \frac{|\varphi(y)|^{2_{\alpha}^{*}}}{|x-y|^{(N-\alpha)}}dy\Big)|\varphi|^{2_{\alpha}^{*}}ds+\int_{\partial D_{\rho}}\Big(\int_{\R^N} \frac{|\phi(y)|^{2}}{|x-y|^{(N-\alpha)}}dy\Big)|\phi|^{2_{\alpha}^{*}}ds\Big),
		\end{split}
	\end{equation}
	by integration by parts.
	As
	$$\aligned
	\sum_{j=1}^{m}\Big\langle \Big((2_{\alpha}^{*}-1)\Big(|x|^{-(N-\alpha)}\ast |Y_{z_j,\lambda}|^{2_{\alpha}^{*}}\Big)Y_{z_j,\lambda}^{2_{\alpha}^{*}-2}Y_{j,l}+2_{\alpha}^{*}\Big(|x|^{-(N-\alpha)}\ast (Y_{z_j,\lambda}^{2_{\alpha}^{*}-1}
	Y_{j,l})\Big)Y_{z_j,\lambda}^{2_{\alpha}^{*}-1}\\
	\displaystyle	\hspace{20.14mm}(2_{\alpha}^{*}-1)\Big(|x|^{-(N-\alpha)}\ast |Z_{z_j,\lambda}|^{2_{\alpha}^{*}}\Big)Z_{z_j,\lambda}^{2_{\alpha}^{*}-2}Z_{j,l}+2_{\alpha}^{*}\Big(|x|^{-(N-\alpha)}\ast (Z_{z_j,\lambda}^{2_{\alpha}^{*}-1}
	Z_{j,l})\Big)Z_{z_j,\lambda}^{2_{\alpha}^{*}-1} \Big) , (\phi,\varphi) \Big\rangle=0
	\endaligned$$
	where $l=1,2,\dots,6$, and \eqref{c14},  we have
	\begin{equation}\label{d13}
		\aligned
		&2\int_{D_{\rho}}\nabla u_{m}\nabla v_{m}dx\\
		&=\int_{D_{\rho}}K_{1}(x)\Big(|x|^{-(N-\alpha)}\ast K_{1}(x) |v_{m}|^{2_{\alpha}^{*}}\Big)|v_{m}|^{2_{\alpha}^{*}} dx+\int_{D_{\rho}}K_{2}(x)\Big(|x|^{-(N-\alpha)}\ast K_{2}(x) |u_{m}|^{2_{\alpha}^{*}}\Big)|u_{m}|^{2_{\alpha}^{*}} dx\\
		&+\sum_{l=1}^{N}c_{l}\sum_{j=1}^{m}\int_{\mathbb{R}^{N}}
		\Big[(2_{\alpha}^{*}-1)\Big(|x|^{-(N-\alpha)}\ast |Y_{z_j,\lambda}|^{2_{\alpha}^{*}}\Big)Y_{z_j,\lambda}^{2_{\alpha}^{*}-2}Y_{j,l}+2_{\alpha}^{*}\Big(|x|^{-(N-\alpha)}\ast (Y_{z_j,\lambda}^{2_{\alpha}^{*}-1}
		Y_{j,l})\Big)Y_{z_j,\lambda}^{2_{\alpha}^{*}-1}\Big]Y_{\overline{r},\overline{x}'',\lambda}\\
	&+	\Big[(2_{\alpha}^{*}-1)\Big(|x|^{-(N-\alpha)}\ast |Z_{z_j,\lambda}|^{2_{\alpha}^{*}}\Big)Z_{z_j,\lambda}^{2_{\alpha}^{*}-2}Z_{j,l}+2_{\alpha}^{*}\Big(|x|^{-(N-\alpha)}\ast (Z_{z_j,\lambda}^{2_{\alpha}^{*}-1}
		Z_{j,l})\Big)Z_{z_j,\lambda}^{2_{\alpha}^{*}-1}\Big]Z_{\overline{r},\overline{x}'',\lambda}dx+O\Big(\int_{\partial D_{\rho}}\nabla \phi\nabla \varphi ds\Big).
		\endaligned\end{equation}
	Substituting \eqref{d13} into \eqref{d12} yields
	\begin{equation}\label{d14}
		\aligned
	&\frac{1}{2_{\alpha}^{*}}\left[ \int_{D_{\rho}}\langle x,\nabla K_{1}(x)\rangle\int_{\R^N}\frac{|v_{m}(y)|^{2_{\alpha}^{*}}|v_{m}(x)|^{2_{\alpha}^{*}}}{|x-y|^{(N-\alpha)}}dxdy+\int_{D_{\rho}}\langle x,\nabla K_{2}(x)\rangle\int_{\R^N}\frac{|u_{m}(y)|^{2_{\alpha}^{*}}|u_{m}(x)|^{2_{\alpha}^{*}}}{|x-y|^{(N-\alpha)}}dxdy\right] \\
	&=(N-2)\int_{D_{\rho}}\nabla u_{m}\nabla v_{m}dx-\frac{N}{2_{\alpha}^{*}}\Big[ \int_{D_{\rho}}\int_{\R^N}\frac{K_{1}(x)K_{1}(y)|v_{m}(y)|^{2_{\alpha}^{*}}|v_{m}(x)|^{2_{\alpha}^{*}}}{|x-y|^{(N-\alpha)}}dxdy\\
	&+\int_{D_{\rho}}\int_{\R^N}\frac{K_{2}(x)K_{2}(y)|u_{m}(y)|^{2_{\alpha}^{*}}|u_{m}(x)|^{2_{\alpha}^{*}}}{|x-y|^{(N-\alpha)}}dxdy\Big] +\frac{N-\alpha}{2_{\alpha}^{*}}\Big[ \int_{D_{\rho}}\int_{\R^N}x(x-y)\frac{K_{1}(x)K_{1}(y)|v_{m}(y)|^{2_{\alpha}^{*}}|v_{m}(x)|^{2_{\alpha}^{*}}}{|x-y|^{(N-\alpha+2)}}dxdy\\
	&+\int_{D_{\rho}}\int_{\R^N}x(x-y)\frac{K_{2}(x)K_{2}(y)|u_{m}(y)|^{2_{\alpha}^{*}}|u_{m}(x)|^{2_{\alpha}^{*}}}{|x-y|^{(N-\alpha+2)}}dxdy\Big] \\
	&+O\Big(\int_{\partial D_{\rho}}\nabla \phi\nabla \varphi ds
	+\int_{\partial D_{\rho}}\Big(\int_{\R^N} \frac{|\varphi(y)|^{2_{\alpha}^{*}}}{|x-y|^{(N-\alpha)}}dy\Big)|\varphi|^{2_{\alpha}^{*}}ds\Big)+\int_{\partial D_{\rho}}\Big(\int_{\R^N} \frac{|\phi(y)|^{2_{\alpha}^{*}}}{|x-y|^{(N-\alpha)}}dy\Big)|\phi|^{2_{\alpha}^{*}}ds\Big)\\
	\endaligned\end{equation}
	\begin{equation}\nonumber
		\aligned
	&=\frac{N-2}{2}\sum_{l=1}^{N}c_{l}\sum_{j=1}^{m}\int_{\mathbb{R}^{N}}
	\Big[(2_{\alpha}^{*}-1)\Big(|x|^{-(N-\alpha)}\ast |Y_{z_j,\lambda}|^{2_{\alpha}^{*}}\Big)Y_{z_j,\lambda}^{2_{\alpha}^{*}-2}Y_{j,l}+2_{\alpha}^{*}\Big(|x|^{-(N-\alpha)}\ast (Y_{z_j,\lambda}^{2_{\alpha}^{*}-1}
	Y_{j,l})\Big)Y_{z_j,\lambda}^{2_{\alpha}^{*}-1}\Big]Z_{\overline{r},\overline{x}'',\lambda}\\
	&+\Big[(2_{\alpha}^{*}-1)\Big(|x|^{-(N-\alpha)}\ast |Z_{z_j,\lambda}|^{2_{\alpha}^{*}}\Big)Z_{z_j,\lambda}^{2_{\alpha}^{*}-2}Z_{j,l}+2_{\alpha}^{*}\Big(|x|^{-(N-\alpha)}\ast (Z_{z_j,\lambda}^{2_{\alpha}^{*}-1}
	Z_{j,l})\Big)Z_{z_j,\lambda}^{2_{\alpha}^{*}-1}\Big]Y_{\overline{r},\overline{x}'',\lambda}dx\\
	&+\left(\frac{N-2}{2}-\frac{N}{2_{\alpha}^{*}}\right) \left[ \int_{D_{\rho}}\int_{\R^N}\frac{K_{1}(x)K_{1}(y)|v_{m}(y)|^{2_{\alpha}^{*}}|v_{m}(x)|^{2_{\alpha}^{*}}}{|x-y|^{(N-\alpha)}}dxdy+\int_{D_{\rho}}\int_{\R^N}\frac{K_{2}(x)K_{2}(y)|u_{m}(y)|^{2_{\alpha}^{*}}|u_{m}(x)|^{2_{\alpha}^{*}}}{|x-y|^{(N-\alpha)}}dxdy\right]\\
	&+\frac{N-\alpha}{2_{\alpha}^{*}}\Big[ \int_{D_{\rho}}\int_{\R^N}x(x-y)\frac{K_{1}(x)K_{1}(y)|v_{m}(y)|^{2_{\alpha}^{*}}|v_{m}(x)|^{2_{\alpha}^{*}}}{|x-y|^{(N-\alpha+2)}}dxdy+\int_{D_{\rho}}\int_{\R^N}x(x-y)\frac{K_{2}(x)K_{2}(y)|u_{m}(y)|^{2_{\alpha}^{*}}|u_{m}(x)|^{2_{\alpha}^{*}}}{|x-y|^{(N-\alpha+2)}}dxdy\Big] \\
&+O\Big(\int_{\partial D_{\rho}}\nabla \phi\nabla \varphi ds
+\int_{\partial D_{\rho}}\Big(\int_{\R^N} \frac{|\varphi(y)|^{2_{\alpha}^{*}}}{|x-y|^{(N-\alpha)}}dy\Big)|\varphi|^{2_{\alpha}^{*}}ds+\int_{\partial D_{\rho}}\Big(\int_{\R^N} \frac{|\phi(y)|^{2_{\alpha}^{*}}}{|x-y|^{(N-\alpha)}}dy\Big)|\phi|^{2_{\alpha}^{*}}ds\Big)\\
&=\frac{N-2}{2}\sum_{l=1}^{N}c_{l}\sum_{j=1}^{m}\int_{\mathbb{R}^{N}}
\Big[(2_{\alpha}^{*}-1)\Big(|x|^{-(N-\alpha)}\ast |Y_{z_j,\lambda}|^{2_{\alpha}^{*}}\Big)Y_{z_j,\lambda}^{2_{\alpha}^{*}-2}Y_{j,l}+2_{\alpha}^{*}\Big(|x|^{-(N-\alpha)}\ast (Y_{z_j,\lambda}^{2_{\alpha}^{*}-1}
Y_{j,l})\Big)Y_{z_j,\lambda}^{2_{\alpha}^{*}-1}\Big]Z_{\overline{r},\overline{x}'',\lambda}\\
	&+\Big[(2_{\alpha}^{*}-1)\Big(|x|^{-(N-\alpha)}\ast |Z_{z_j,\lambda}|^{2_{\alpha}^{*}}\Big)Z_{z_j,\lambda}^{2_{\alpha}^{*}-2}Z_{j,l}+2_{\alpha}^{*}\Big(|x|^{-(N-\alpha)}\ast (Z_{z_j,\lambda}^{2_{\alpha}^{*}-1}
Z_{j,l})\Big)Z_{z_j,\lambda}^{2_{\alpha}^{*}-1}\Big]Y_{\overline{r},\overline{x}'',\lambda}dx\\
&+O\Big(\int_{\partial D_{\rho}}\nabla \phi\nabla \varphi ds
+\int_{\partial D_{\rho}}\Big(\int_{\R^N} \frac{|\varphi(y)|^{2_{\alpha}^{*}}}{|x-y|^{(N-\alpha)}}dy\Big)|\varphi|^{2_{\alpha}^{*}}ds+\int_{\partial D_{\rho}}\Big(\int_{\R^N} \frac{|\phi(y)|^{2_{\alpha}^{*}}}{|x-y|^{(N-\alpha)}}dy\Big)|\phi|^{2_{\alpha}^{*}}ds\Big)\\
&+O\Big(\int_{ D_{\rho}}\int_{\R^N\backslash D_{\rho}}(x_i-y_i)\frac{|\varphi(x)|^{2_{\alpha}^{*}}|\varphi(y)|^{2_{\alpha}^{*}}}{|x-y|^{(N-\alpha+2)}}dxdy+\int_{ D_{\rho}}\int_{\R^N\backslash D_{\rho}}(x_i-y_i)\frac{|\phi(y)|^{2_{\alpha}^{*}}|\phi(x)|^{2_{\alpha}^{*}}}{|x-y|^{(N-\alpha+2)}}dxdy\Big)+O(\frac{1}{\lambda^{2+\varepsilon}}).
\endaligned\end{equation}
where $i=3,\cdot\cdot\cdot,N$.
	
	\textbf{Step 2:} By the same method as in Lemma \ref{D1}, we have
	$$\aligned
	\sum_{j=1}^{m}\int_{\mathbb{R}^{N}}
	\Big[(2_{\alpha}^{*}-1)\Big(|x|^{-(N-\alpha)}\ast |Y_{z_j,\lambda}|^{2_{\alpha}^{*}}\Big)Y_{z_j,\lambda}^{2_{\alpha}^{*}-2}Y_{j,l}+2_{\alpha}^{*}\Big(|x|^{-(N-\alpha)}\ast (Y_{z_j,\lambda}^{2_{\alpha}^{*}-1}
	Y_{j,l})\Big)Y_{z_j,\lambda}^{2_{\alpha}^{*}-1}\Big]\frac{\partial Y_{\overline{r},\overline{x}'',\lambda}}{\partial \overline{x}''_{t}}
	&=O(m\lambda^{2}),
	\endaligned$$
	
		$$\aligned
	&\sum_{j=1}^{m}\int_{\mathbb{R}^{N}}
	\Big[(2_{\alpha}^{*}-1)\Big(|x|^{-(N-\alpha)}\ast |Y_{z_j,\lambda}|^{2_{\alpha}^{*}}\Big)Y_{z_j,\lambda}^{2_{\alpha}^{*}-2}Y_{j,1}+2_{\alpha}^{*}\Big(|x|^{-(N-\alpha)}\ast (Y_{z_j,\lambda}^{2_{\alpha}^{*}-1}
	Y_{j,1})\Big)Y_{z_j,\lambda}^{2_{\alpha}^{*}-1}\Big]\frac{\partial Y_{\overline{r},\overline{x}'',\lambda}}{\partial \overline{x}''_{t}}
	=O(m),
	\endaligned$$
		$$\aligned
	\sum_{j=1}^{m}\int_{\mathbb{R}^{N}}
	\Big[(2_{\alpha}^{*}-1)\Big(|x|^{-(N-\alpha)}\ast |Z_{z_j,\lambda}|^{2_{\alpha}^{*}}\Big)Z_{z_j,\lambda}^{2_{\alpha}^{*}-2}Z_{j,l}+2_{\alpha}^{*}\Big(|x|^{-(N-\alpha)}\ast (Z_{z_j,\lambda}^{2_{\alpha}^{*}-1}
	Z_{j,l})\Big)Z_{z_j,\lambda}^{2_{\alpha}^{*}-1}\Big]\frac{\partial Z_{\overline{r},\overline{x}'',\lambda}}{\partial \overline{x}''_{t}}
	&=O(m\lambda^{2}),
	\endaligned$$
	
	$$\aligned
	&\sum_{j=1}^{m}\int_{\mathbb{R}^{N}}
	\Big[(2_{\alpha}^{*}-1)\Big(|x|^{-(N-\alpha)}\ast |Z_{z_j,\lambda}|^{2_{\alpha}^{*}}\Big)Z_{z_j,\lambda}^{2_{\alpha}^{*}-2}Z_{j,1}+2_{\alpha}^{*}\Big(|x|^{-(N-\alpha)}\ast (Z_{z_j,\lambda}^{2_{\alpha}^{*}-1}
	Z_{j,1})\Big)Y_{z_j,\lambda}^{2_{\alpha}^{*}-1}\Big]\frac{\partial Z_{\overline{r},\overline{x}'',\lambda}}{\partial \overline{x}''_{t}}
	=O(m).
	\endaligned$$
	 where $l=2,\cdot\cdot\cdot,N.$ Combining these and Lemma \ref{D3}, $c_{i}=o(\frac{1}{\lambda^{2}})c_{1},\ i=2,\cdot\cdot\cdot, N$, \eqref{dr0}, \eqref{dr1} and \eqref{c14}, we have
	\begin{equation}\label{d15}
		c_i=O(\frac{1}{\lambda^{3+\varepsilon}}),  i=2,\cdot\cdot\cdot,N,
	\end{equation}
	and
	\begin{equation}\label{d16}
		c_1=O(\frac{1}{\lambda^{1+\varepsilon}}). \
	\end{equation}

Moreover, we have
\begin{equation}\nonumber
	\aligned
		\sum_{j=1}^{m}\int_{\mathbb{R}^{N}}
	\Big[(2_{\alpha}^{*}-1)\Big(|x|^{-(N-\alpha)}\ast |Y_{z_j,\lambda}|^{2_{\alpha}^{*}}\Big)Y_{z_j,\lambda}^{2_{\alpha}^{*}-2}Y_{j,l}+2_{\alpha}^{*}\Big(|x|^{-(N-\alpha)}\ast (Y_{z_j,\lambda}^{2_{\alpha}^{*}-1}
	Y_{j,l})\Big)Y_{z_j,\lambda}^{2_{\alpha}^{*}-1}\Big]Y_{\overline{r},\overline{x}'',\lambda}
	&=O(m\lambda),
	\endaligned
\end{equation}
\begin{equation}\nonumber
	\aligned
	&\sum_{j=1}^{m}\int_{\mathbb{R}^{N}}
\Big[(2_{\alpha}^{*}-1)\Big(|x|^{-(N-\alpha)}\ast |Y_{z_j,\lambda}|^{2_{\alpha}^{*}}\Big)Y_{z_j,\lambda}^{2_{\alpha}^{*}-2}Y_{j,1}+2_{\alpha}^{*}\Big(|x|^{-(N-\alpha)}\ast (Y_{z_j,\lambda}^{2_{\alpha}^{*}-1}
Y_{j,1})\Big)Y_{z_j,\lambda}^{2_{\alpha}^{*}-1}\Big] Y_{\overline{r},\overline{x}'',\lambda}dx=O(\frac{m}{\lambda}),
	\endaligned\end{equation}
\begin{equation}\nonumber
	\aligned
	\sum_{j=1}^{m}\int_{\mathbb{R}^{N}}
\Big[(2_{\alpha}^{*}-1)\Big(|x|^{-(N-\alpha)}\ast |Z_{z_j,\lambda}|^{2_{\alpha}^{*}}\Big)Z_{z_j,\lambda}^{2_{\alpha}^{*}-2}Z_{j,l}+2_{\alpha}^{*}\Big(|x|^{-(N-\alpha)}\ast (Z_{z_j,\lambda}^{2_{\alpha}^{*}-1}
Z_{j,l})\Big)Z_{z_j,\lambda}^{2_{\alpha}^{*}-1}\Big] Z_{\overline{r},\overline{x}'',\lambda}dx&=O(m\lambda),
	\endaligned
\end{equation}
and
\begin{equation}\nonumber
	\aligned
		&\sum_{j=1}^{m}\int_{\mathbb{R}^{N}}
	\Big[(2_{\alpha}^{*}-1)\Big(|x|^{-(N-\alpha)}\ast |Z_{z_j,\lambda}|^{2_{\alpha}^{*}}\Big)Z_{z_j,\lambda}^{2_{\alpha}^{*}-2}Z_{j,1}+2_{\alpha}^{*}\Big(|x|^{-(N-\alpha)}\ast (Z_{z_j,\lambda}^{2_{\alpha}^{*}-1}
	Z_{j,1})\Big)Y_{z_j,\lambda}^{2_{\alpha}^{*}-1}\Big] Y_{\overline{r},\overline{x}'',\lambda}dx=O(\frac{m}{\lambda}).
	\endaligned\end{equation}
 where $l=2,\cdot\cdot\cdot,N.$ According to the above work, we can find \eqref{d14} is equivalent to
\begin{equation}\label{d17}
	\aligned
	&\frac{1}{2_{\alpha}^{*}}\left[ \int_{D_{\rho}}\langle x,\nabla K_{1}(x)\rangle\int_{\R^N}\frac{|v_{m}(y)|^{2_{\alpha}^{*}}|v_{m}(x)|^{2_{\alpha}^{*}}}{|x-y|^{(N-\alpha)}}dxdy+\int_{D_{\rho}}\langle x,\nabla K_{2}(x)\rangle\int_{\R^N}\frac{|u_{m}(y)|^{2_{\alpha}^{*}}|u_{m}(x)|^{2_{\alpha}^{*}}}{|x-y|^{(N-\alpha)}}dxdy\right] \\
&=O\Big(\int_{\partial D_{\rho}}\nabla \phi\nabla \varphi ds
+\int_{\partial D_{\rho}}\Big(\int_{\R^N} \frac{|\varphi(y)|^{2_{\alpha}^{*}}}{|x-y|^{(N-\alpha)}}dy\Big)|\varphi|^{2_{\alpha}^{*}}ds+\int_{\partial D_{\rho}}\Big(\int_{\R^N} \frac{|\phi(y)|^{2_{\alpha}^{*}}}{|x-y|^{(N-\alpha)}}dy\Big)|\phi|^{2_{\alpha}^{*}}ds\Big)\\
&+O\Big(\int_{ D_{\rho}}\int_{\R^N\backslash D_{\rho}}(x_i-y_i)\frac{|\varphi(x)|^{2_{\alpha}^{*}}|\varphi(y)|^{2_{\alpha}^{*}}}{|x-y|^{(N-\alpha+2)}}dxdy+\int_{ D_{\rho}}\int_{\R^N\backslash D_{\rho}}(x_i-y_i)\frac{|\phi(y)|^{2_{\alpha}^{*}}|\phi(x)|^{2_{\alpha}^{*}}}{|x-y|^{(N-\alpha+2)}}dxdy\Big)\\
&+O(\frac{1}{\lambda^{2+\varepsilon}})+O(\frac{m}{\lambda^{2+\varepsilon}})
\endaligned\end{equation}
for some small $\varepsilon>0$.

\textbf{Step 3:} Integrating by parts, we can rewrite \eqref{d2} as
\begin{equation}\label{d11}
	\begin{split}
&\int_{D_{\rho}}\frac{\partial K_{1}(\overline{r},\overline{x}'')}{\partial x_{i}}\int_{\R^N}\frac{K_{1}(y)|v_{m}(y)|^{2_{\alpha}^{*}}|v_{m}(x)|^{2_{\alpha}^{*}}}{|x-y|^{(N-\alpha)}}dxdy+\int_{D_{\rho}}\frac{\partial K_{2}(\overline{r},\overline{x}'')}{\partial x_{i}}\int_{\R^N}\frac{K_{2}(y)|u_{m}(y)|^{2_{\alpha}^{*}}|u_{m}(x)|^{2_{\alpha}^{*}}}{|x-y|^{(N-\alpha)}}dxdy\\
&=O\Big(\int_{\partial D_{\rho}}\nabla \phi\nabla \varphi ds
+\int_{\partial D_{\rho}}\Big(\int_{\R^N} \frac{|\varphi(y)|^{2_{\alpha}^{*}}}{|x-y|^{(N-\alpha)}}dy\Big)|\varphi|^{2_{\alpha}^{*}}ds+\int_{\partial D_{\rho}}\Big(\int_{\R^N} \frac{|\phi(y)|^{2_{\alpha}^{*}}}{|x-y|^{(N-\alpha)}}dy\Big)|\phi|^{2_{\alpha}^{*}}ds\Big)\\
&+O\Big(\int_{ D_{\rho}}\int_{\R^N\backslash D_{\rho}}(x_i-y_i)\frac{|\varphi(x)|^{2_{\alpha}^{*}}|\varphi(y)|^{2_{\alpha}^{*}}}{|x-y|^{(N-\alpha+2)}}dxdy+\int_{ D_{\rho}}\int_{\R^N\backslash D_{\rho}}(x_i-y_i)\frac{|\phi(y)|^{2_{\alpha}^{*}}|\phi(x)|^{2_{\alpha}^{*}}}{|x-y|^{(N-\alpha+2)}}dxdy\Big).
\end{split}
\end{equation}

Therefore, \eqref{d17} is equivalent to
\begin{equation}\label{d19}
	\begin{split}
&\int_{D_{\rho}}\frac{r}{2_{\alpha}^{*}}\frac{\partial K_{1}(\overline{r},\overline{x}'')}{\partial r}\int_{\R^N}\frac{ K_{1}(y)|v_{m}(y)|^{2_{\alpha}^{*}}|v_{m}(x)|^{2_{\alpha}^{*}}}{|x-y|^{(N-\alpha)}}dxdy+\int_{D_{\rho}}\frac{r}{2_{\alpha}^{*}}\frac{\partial K_{2}(\overline{r},\overline{x}'')}{\partial r}\int_{\R^N}\frac{K_{2}(y)|u_{m}(y)|^{2_{\alpha}^{*}}|u_{m}(x)|^{2_{\alpha}^{*}}}{|x-y|^{(N-\alpha)}}dxdy\\
&=o(\frac{m}{\lambda^{2}})+O\Big(\int_{\partial D_{\rho}}\nabla \phi\nabla \varphi ds
+\int_{\partial D_{\rho}}\Big(\int_{\R^N} \frac{|\varphi(y)|^{2_{\alpha}^{*}}}{|x-y|^{(N-\alpha)}}dy\Big)|\varphi|^{2_{\alpha}^{*}}ds+\int_{\partial D_{\rho}}\Big(\int_{\R^N} \frac{|\phi(y)|^{2_{\alpha}^{*}}}{|x-y|^{(N-\alpha)}}dy\Big)|\phi|^{2_{\alpha}^{*}}ds\Big)\\
&+O\Big(\int_{ D_{\rho}}\int_{\R^N\backslash D_{\rho}}(x_i-y_i)\frac{|\varphi(x)|^{2_{\alpha}^{*}}|\varphi(y)|^{2_{\alpha}^{*}}}{|x-y|^{(N-\alpha+2)}}dxdy+\int_{ D_{\rho}}\int_{\R^N\backslash D_{\rho}}(x_i-y_i)\frac{|\phi(y)|^{2_{\alpha}^{*}}|\phi(x)|^{2_{\alpha}^{*}}}{|x-y|^{(N-\alpha+2)}}dxdy\Big).
\end{split}
\end{equation}
Estimate similar to that in the proof of Lemma 3.5 in\cite{PWY} gives that
$$
\int_{D_{4\delta}\backslash D_{3\delta}}\nabla\phi\nabla\varphi dx=
O(\frac{m}{\lambda^{2+\varepsilon}}) .
$$
Moreover, by Lemma \ref{D4} we have
$$\aligned
&\int_{D_{4\delta}\backslash D_{3\delta}}\nabla \phi\nabla \varphi ds
+\int_{D_{4\delta}\backslash D_{3\delta}}\Big(\int_{ D_{\rho}} \frac{|\varphi(y)|^{2_{\alpha}^{*}}}{|x-y|^{(N-\alpha)}}dy\Big)|\varphi|^{2_{\alpha}^{*}}ds+\int_{D_{4\delta}\backslash D_{3\delta}}\Big(\int_{ D_{\rho}} \frac{|\phi(y)|^{2_{\alpha}^{*}}}{|x-y|^{(N-\alpha)}}dy\Big)|\phi|^{2_{\alpha}^{*}}ds\\
&+\int_{D_{4\delta}\backslash D_{3\delta}}\int_{\R^N\backslash D_{\rho}}(x_i-y_i)\frac{|\varphi(x)|^{2_{\alpha}^{*}}|\varphi(y)|^{2_{\alpha}^{*}}}{|x-y|^{(N-\alpha+2)}}dxdy+\int_{D_{4\delta}\backslash D_{3\delta}}\int_{\R^N\backslash D_{\rho}}(x_i-y_i)\frac{|\phi(y)|^{2_{\alpha}^{*}}|\phi(x)|^{2_{\alpha}^{*}}}{|x-y|^{(N-\alpha+2)}}dxdy=O(\frac{m}{\lambda^{2+\varepsilon}}),
\endaligned$$
where $i=3,\cdot\cdot\cdot,N$. Finally, we can find a $\rho\in(3\delta,4\delta)$ such that
$$\aligned
&\int_{\partial D_{\rho}}\nabla \phi\nabla \varphi ds
+\int_{\partial D_{\rho}}\Big(\int_{\R^N} \frac{|\varphi(y)|^{2_{\alpha}^{*}}}{|x-y|^{(N-\alpha)}}dy\Big)|\varphi|^{2_{\alpha}^{*}}ds+\int_{\partial D_{\rho}}\Big(\int_{\R^N} \frac{|\phi(y)|^{2_{\alpha}^{*}}}{|x-y|^{(N-\alpha)}}dy\Big)|\phi|^{2_{\alpha}^{*}}ds\\
&+\int_{ D_{\rho}}\int_{\R^N\backslash D_{\rho}}(x_i-y_i)\frac{|\varphi(x)|^{2_{\alpha}^{*}}|\varphi(y)|^{2_{\alpha}^{*}}}{|x-y|^{(N-\alpha+2)}}dxdy+\int_{ D_{\rho}}\int_{\R^N\backslash D_{\rho}}(x_i-y_i)\frac{|\phi(y)|^{2_{\alpha}^{*}}|\phi(x)|^{2_{\alpha}^{*}}}{|x-y|^{(N-\alpha+2)}}dxdy=O(\frac{m}{\lambda^{2+\varepsilon}})
\endaligned$$
where $i=3,\cdot\cdot\cdot,N$.

$\textbf{Step 4:}$ For any $C^1$ function $g(r, x'')$, by Lemma 4.4 in \cite{CYZ} it holds
$$\aligned
\int_{D_{\rho}}g(r, x'')&\int_{\R^N}\frac{ K_{1}(y)|v_{m}(y)|^{2_{\alpha}^{*}}|v_{m}(x)|^{2_{\alpha}^{*}}}{|x-y|^{(N-\alpha)}}dxdy\\
&=m(g(\overline{r},\overline{x}'')\int_{\R^N}\int_{\mathbb{R}^{N}}\frac{ K_{1}(z_{j}+\frac{y}{\lambda})|V_{0,1}(y)|^{2_{\alpha}^{*}}|V_{0,1}|^{2_{\alpha}^{*}}}{|x-y|^{(N-\alpha)}}dxdy+o(\frac{1}{\lambda^{1-\varepsilon}})),
\endaligned$$
and
$$\aligned
\int_{D_{\rho}}g(r, x'')&\int_{\R^N}\frac{ K_{2}(y)|u_{m}(y)|^{2_{\alpha}^{*}}|u_{m}(x)|^{2_{\alpha}^{*}}}{|x-y|^{(N-\alpha)}}dxdy\\
&=m(g(\overline{r},\overline{x}'')\int_{\R^N}\int_{\mathbb{R}^{N}}\frac{ K_{2}(z_{j}+\frac{y}{\lambda})|U_{0,1}(y)|^{2_{\alpha}^{*}}|U_{0,1}|^{2_{\alpha}^{*}}}{|x-y|^{(N-\alpha)}}dxdy+o(\frac{1}{\lambda^{1-\varepsilon}})).
\endaligned$$
Therefore, from \eqref{d11} and \eqref{d19} we deduce that
$$\aligned
&m\Big( \frac{\partial (K_{1}(\overline{r},\overline{x}''))}{\partial \overline{x}_{i}}\int_{\mathbb{R}^{N}}\int_{\mathbb{R}^{N}}\frac{ K_{1}(z_{j}+\frac{y}{\lambda})|V_{0,1}(y)|^{2_{\alpha}^{*}}|V_{0,1}|^{2_{\alpha}^{*}}}{|x-y|^{(N-\alpha)}}dxdy\\
&+\frac{\partial (K_{2}(\overline{r},\overline{x}''))}{\partial \overline{x}_{i}}\int_{\R^N}\int_{\mathbb{R}^{N}}\frac{ K_{2}(z_{j}+\frac{y}{\lambda})|U_{0,1}(y)|^{2_{\alpha}^{*}}|U_{0,1}|^{2_{\alpha}^{*}}}{|x-y|^{(N-\alpha)}}dxdy+o(\frac{1}{\lambda^{1-\varepsilon}})\Big) =o(\frac{m}{\lambda^{2}}),
\endaligned$$
and
$$\aligned
&m\Big(\frac{r}{2_{\alpha}^{*}}\frac{\partial K_{1}(\overline{r},\overline{x}'')}{\partial \overline{r}}\int_{\mathbb{R}^{N}}\int_{\mathbb{R}^{N}}\frac{ K_{1}(z_{j}+\frac{y}{\lambda})|V_{0,1}(y)|^{2_{\alpha}^{*}}|V_{0,1}|^{2_{\alpha}^{*}}}{|x-y|^{(N-\alpha)}}dxdy\\
&+\frac{r}{2_{\alpha}^{*}}\frac{\partial K_{2}(\overline{r},\overline{x}'')}{\partial \overline{r}}\int_{\R^N}\int_{\mathbb{R}^{N}}\frac{ K_{2}(z_{j}+\frac{y}{\lambda})|U_{0,1}(y)|^{2_{\alpha}^{*}}|U_{0,1}|^{2_{\alpha}^{*}}}{|x-y|^{(N-\alpha)}}dxdy+o(\frac{1}{\lambda^{1-\varepsilon}})\Big) =o(\frac{m}{\lambda^{2}}),
\endaligned$$
hence that the equations to determine $(\overline{r},\overline{x}'')$ are
\begin{equation}\label{d21}
	\frac{\partial (K_{1}(\overline{r},\overline{x}'')+K_{2}(\overline{r},\overline{x}''))}{\partial \overline{x}_{i}}=o(\frac{1}{\lambda^{1-\varepsilon}}),  i=3,\cdot\cdot\cdot,N,
\end{equation}
and
\begin{equation}\label{d22}
	\frac{\partial (K_{1}(\overline{r},\overline{x}'')+K_{2}(\overline{r},\overline{x}''))}{\partial \overline{r}}=o(\frac{1}{\lambda^{1-\varepsilon}}).
\end{equation}
Combing these we have already proved that \eqref{d1}, \eqref{d2} and \eqref{d3} are equivalent to \eqref{d21}, \eqref{d22} and
$$
-\frac{B_1}{\lambda^{3}}+\frac{m^{N-2}B_3}{\lambda^{N-1}}=O(\frac{1}{\lambda^{3+\varepsilon}}).
$$

Let $\lambda=tm^{\frac{N-2}{N-4}}$, we have $t\in[L_0, L_1]$ by $\lambda\in[L_0m^{\frac{N-2}{N-4}}, L_1m^{\frac{N-2}{N-4}}]$, and thus
\begin{equation}\label{d23}
	-\frac{B_1}{t^{3}}+\frac{B_3}{t^{N-1}}=o(1), t\in[L_0, L_1].
\end{equation}
Set
$$
F(t, \overline{r},\overline{x}'')=\Big(\nabla_{\overline{r},\overline{x}''}(K_{1}(\overline{r},\overline{x}'')+K_{2}(\overline{r},\overline{x}'')),
-\frac{2B_1}{t^{3}}+\frac{B_3}{t^{N-1}}\Big).
$$
Then we obtain
$$
\mbox{deg}(F(t, \overline{r},\overline{x}''), [L_0, L_1]\times B_{\theta}((r_{0},x_{0}'')))
=-\mbox{deg}(\nabla_{\overline{r},\overline{x}''}(K_{1}(\overline{r},\overline{x}'')+K_{2}(\overline{r},\overline{x}'')),B_{\frac{1}{\lambda^{1-\theta}}}((r_{0},x_{0}'')))\neq0.
$$
Finally,  \eqref{d21}, \eqref{d22} and \eqref{d23} have a solution $t_{m}\in[L_0, L_1]$, $(\overline{r}_{m},\overline{x}_{m}'')\in B_{\frac{1}{\lambda^{1-\theta}}}((r_{0},x_{0}''))$.
$\hfill{} \Box$

\vspace{0.5cm}

\noindent \textbf{Data Availability Statement.} Data sharing is not applicable to this article as no datasets were generated or analyzed during the current study.

\vspace{0.5cm}

\noindent\textbf{Conflict Of Interest Statement.} The authors have no competing interests to declare that are relevant to the content of this article.

\end{document}